\documentclass[12pt]{amsart}
\usepackage{amsmath,amsthm,amsfonts,amssymb,amscd,euscript}
\usepackage{color}
\usepackage{graphics}
\usepackage{tikz}
\input{xy}
\xyoption{all}
\newlength{\defbaselineskip}
\theoremstyle{plain}
\newtheorem{theorem}{Theorem}[section]

\newtheorem{proposition}[theorem]{Proposition}
\newtheorem{corollary}[theorem]{Corollary}
\newtheorem{lemma}[theorem]{Lemma}

\theoremstyle{definition}
\newtheorem{definition}[theorem]{Definition}

\newtheorem{remark}[theorem]{Remark}

\newtheorem{conjecture}[theorem]{Conjecture}

\newcommand{\ZZ}{\mathbb{Z}}

\newcommand{\QQ}{\mathbb{Q}}

\newcommand{\gchoose}[2]{\left[\begin{array}{c}{#1 } \\{#2} \end{array}\right]  }
\newcommand{\nuisone}{1}
\newcommand{\newr}{{ r}}
\newcommand{\news}{{ \xi}}
\newcommand{\newt}{{ \omega}}

\newcommand{\ralf}{}

\begin{document}

\title{Positivity for cluster algebras}
\author{Kyungyong Lee}
\address{Department of Mathematics, Wayne State University, Detroit, MI 48202 and
Center for Mathematical Challenges, Korea Institute for Advanced Study, 85 Hoegiro, Dongdaemun-gu, Seoul 130-722 Republic of Korea}
\email{klee@math.wayne.edu}
\author{Ralf Schiffler} 
\address{Department of Mathematics, University of Connecticut,  Storrs, CT 06269-3009}
\email{schiffler@math.uconn.edu}
\thanks{{The first author is  partially
   supported by the NSF grant DMS 0901367. The second author is partially
   supported by the NSF grants DMS-1001637 and DMS-1254567.}}

\date{\today} 
\dedicatory{To the memory of Andrei Zelevinsky     }

\maketitle
\begin{abstract}  We prove the positivity conjecture for all skew-symmetric cluster algebras.
\end{abstract}

\section{Introduction}\label{sect intro}

Cluster algebras have been introduced by Fomin and Zelevinsky in \cite{FZ} in the context of total positivity and canonical bases in Lie theory. Since then cluster algebras have been shown to be related to various fields in mathematics including representation theory of finite dimensional algebras, Teichm\"uller theory, Poisson geometry, combinatorics, Lie theory, tropical geometry and mathematical physics.

A cluster algebra is a subalgebra of a field of rational functions in $N$ variables $x_1,x_2,\ldots,x_N$, given by specifying a set of generators, the so-called \emph{cluster variables}. These generators are constructed in a recursive way, starting from the initial variables $x_1,x_2,\ldots,x_N$, by a procedure called \emph{mutation}, which is determined by the choice of a skew-symmetric $N\times N$ integer matrix $B$ or, equivalently, by a quiver $Q$.
Although each mutation is an elementary operation, it is very difficult to compute cluster variables in general because of the recursive character of the construction.

Finding explicit computable direct formulas for the cluster variables is one of the main open problems in the theory of cluster algebras and has been studied by many mathematicians. In 2002, Fomin and Zelevinsky showed that every cluster variable is a Laurent polynomial in the initial variables $x_1,x_2,\ldots, x_N$, and they conjectured that this Laurent polynomial has positive coefficients \cite{FZ}.

 This \emph{positivity conjecture} has been proved in the following special cases
 \begin{itemize}
 \item \emph{Acyclic cluster algebras.} These are cluster algebras given by a quiver that is mutation equivalent to a quiver without oriented cycles. In this case, positivity has been shown in \cite{KQ}, building on \cite{BZ,HL,N,Q}, using monoidal categorifications of quantum cluster algebras and  perverse sheaves over graded quiver varieties. The bipartite case has been shown first in \cite{N}.  In the special case where the initial seed is acyclic, a different proof has been given later in \cite{Ef}  using Donaldson-Thomas theory. Very recently, after our proof of positivity was available, this approach has also been used to prove positivity for all rank 4 cluster algebras in \cite{DMSS}.
\item \emph{Cluster algebras from surfaces.}  In this case, positivity has been shown in \cite{MSW} building on \cite{S2,ST,S3}, using the fact that each cluster variable in such a cluster algebra corresponds to a curve in an oriented Riemann surface, and the Laurent expansion of the cluster variable is determined by the crossing pattern of the curve with a fixed triangulation of the surface \cite{FG,FST}. The construction and the proof of the positivity conjecture have been generalized to non skew-symmetric cluster algebras from orbifolds in \cite{FeShTu}.
\end{itemize}

Our approach in this paper is different. We prove positivity almost exclusively by elementary algebraic computation. The advantage of this approach is that we do not need to restrict to a special type of cluster algebras. Our main result is the following.

\begin{theorem} \label{thm 1.1} The positivity conjecture holds in every skew-symmetric  cluster algebra.
\end{theorem}

Our argument provides a method for the computation of the Laurent expansions of cluster variables, and some examples of explicit computations were given in our earlier work \cite{LS3}.
We point out that
direct formulas for the Laurent polynomials have been obtained earlier in several special cases. The most general results are the following:
\begin{itemize}
\item a formula involving the Euler-Poincar\'e characteristic of quiver Grassmannians obtained in \cite{FK,DWZ2} using categorification and generalizing results in \cite{CC,CK1}. While this formula shows a very interesting connection between cluster algebras and geometry, it is of limited computational use, since the Euler-Poincar\'e characteristics of quiver Grassmannians are hard to compute. In particular, this formula does not show positivity. On the other hand, the positivity result in this paper proves the positivity of the Euler-Poincar\'e characteristics of the quiver Grassmannians involved, see section~\ref{sect Grass}.
\item an elementary combinatorial formula for cluster algebras from surfaces given in \cite{MSW}.
\item a formula for cluster variables corresponding to string modules as a product of $2\times 2$ matrices obtained in \cite{ADSS}, generalizing a result in \cite{ARS}.
\end{itemize}

The main tools of the proof of Theorem \ref{thm 1.1} are modified versions of two formulas for the rank two case, one obtained by the first author in \cite{L} and the other obtained by both authors in \cite{LS}. These formulas allow for the computation of  the Laurent expansions of a given cluster variable with respect to any seed close enough to the variable, in the sense that there is a sequence of mutations $\mu_d,\mu_e,\mu_d,\mu_e\ldots$ using only two directions $d$ and $e$ which links seed and variable. The general result then follows by inductive reasoning.

 {\ralf We actually show the stronger result, Theorem \ref{thm1}, that for every cluster variable $u$ and for every cluster $\mathbf{x}$, there exists a connected rank 2 subtree $\mathbb{T}_2$ of the exchange tree containing $\mathbf{x}$ 
 such that $u$  can be expressed as a sum of four positive Laurent polynomials in the variables of four clusters closest  to $\mathbf{x}$ in $\mathbb{T}_2$, and such that the variables that are not contained in all four clusters appear only with positive powers. Because of rank 2 positivity, this implies that $u$ is a positive Laurent polynomial in every cluster in $\mathbb{T}_2$. 
 
 The proof of Theorem \ref{thm1} is by induction on the number of rank 2 mutation subsequences of the mutation sequence from $u$ to $\mathbf{x}$.
 It uses two different rank 2 formulas, to express cluster variables that are two rank 2 sequences away from $\mathbf{x}$ as Laurent polynomials in four clusters including $\mathbf{x}$. First we compute the Laurent expansion $\mathcal{L}_1$ of the cluster variable after one rank 2 mutation sequence as a sum of two Laurent polynomials in two adjacent clusters such that the variables that are not contained in both clusters appear only with positive powers. 
 Then we compute the Laurent expansions in four clusters including $\mathbf{x} $ of all cluster variables appearing in $\mathcal{L}_1$ using the second rank 2 sequence and then substitute these in $\mathcal{L}_1$. We then show that the variables  that are not contained in all four clusters appear only with positive powers.
 }

 If the cluster algebra is not skew-symmetric, it is shown in \cite{R,LLZ} that (an adaptation of)  the second rank two formula still holds. We therefore expect that our argument can be generalized to prove the positivity conjecture for non skew-symmetric cluster algebras.
A non-commutative version of the formula has been given in \cite{LS2,R}.

The article is organized as follows. We start by recalling some definitions and results from the theory of cluster algebras in section \ref{sect 2}. 
{\ralf In section \ref{sect rank 2}, we study mutation sequences of rank 3 quivers, and we recall the definition of compatible pairs as well as  results from  our previous work \cite{LS3} in section \ref{sect rank 2}. The positivity conjecture is proved, using Theorem~\ref{thm1}, in  section \ref{sect 3},  and Theorem~\ref{thm1} is proved in section \ref{sect proof} following the outline above. }
As an application, we  show that certain quiver Grassmannians have positive Euler-Poincar\'e characteristic in section \ref{sect Grass}. 

\noindent \emph{Acknowledgements.} We are grateful to  Giovanni Cerulli Irelli and Daniel Labardini-Fragoso for pointing out to us that it is not necessary to assume that the cluster algebra is of geometric type.
We also thank Frank Ban, Sergey Fomin, Andrea Gatica, Li Li, Hiraku Nakjima, Hugh Thomas and Lauren Williams for suggesting improvements to the presentation of the article. We are especially grateful to the anonymous referees for making many valuable suggestions.

\section{Cluster algebras}\label{sect 2}
In this section, we review some notions from the theory of cluster algebras 
  introduced by Fomin and Zelevinsky in \cite{FZ}.
Our definition follows the exposition in \cite{FZ4}.

To define  a cluster algebra~$\mathcal{A}$ we must first fix its
ground ring.
Let $(\mathbb{P},\oplus, \cdot)$ be a \emph{semifield}, i.e.,
an abelian multiplicative group endowed with a binary operation of
\emph{(auxiliary) addition}~$\oplus$ which is commutative, associative, and
distributive with respect to the multiplication in~$\mathbb{P}$.
The group ring~$\ZZ\mathbb{P}$ will be
used as a \emph{ground ring} for~$\mathcal{A}$.

One important choice for $\mathbb{P}$ is the tropical semifield; in this case we say that the
corresponding cluster algebra is of {\it geometric type}.
Let $\textup{Trop} (u_1, \dots, u_{m})$ be an abelian group (written
multiplicatively) freely generated by the $u_j$.
We define  $\oplus$ in $\textup{Trop} (u_1,\dots, u_{m})$ by
\begin{equation}
\label{eq:tropical-addition}
\prod_j u_j^{a_j} \oplus \prod_j u_j^{b_j} =
\prod_j u_j^{\min (a_j, b_j)} \,,
\end{equation}
and call $(\textup{Trop} (u_1,\dots,u_{m}),\oplus,\cdot)$ a \emph{tropical
 semifield}.
Note that the group ring of $\textup{Trop} (u_1,\dots,u_{m})$ is the ring of Laurent
polynomials in the variables~$u_j\,$.

As an \emph{ambient field} for
$\mathcal{A}$, we take a field $\mathcal{F}$
isomorphic to the field of rational functions in $N$ independent
variables (here $N$ is the \emph{rank} of~$\mathcal{A}$),
with coefficients in~$\QQ \mathbb{P}$.
Note that the definition of $\mathcal{F}$ does not involve
the auxiliary addition
in~$\mathbb{P}$.

\begin{definition}
\label{def:seed}
A \emph{labeled seed} in~$\mathcal{F}$ is
a triple $(\mathbf{x}, \mathbf{y}, B)$, where
\begin{itemize}
\item
$\mathbf{x} = (x_1, \dots, x_N)$ is an $N$-tuple
from $\mathcal{F}$
forming a \emph{free generating set} over $\QQ \mathbb{P}$,
\item
$\mathbf{y} = (y_1, \dots, y_N)$ is an $N$-tuple
from $\mathbb{P}$, and
\item
$B = (b_{ij})$ is an $N\!\times\! N$ integer matrix
which is \emph{skew-symmetrizable}.
\end{itemize}
That is, $x_1, \dots, x_N$
are algebraically independent over~$\QQ \mathbb{P}$, and
$\mathcal{F} = \QQ \mathbb{P}(x_1, \dots, x_N)$.
We refer to~$\mathbf{x}$ as the (labeled)
\emph{cluster} of a labeled seed $(\mathbf{x}, \mathbf{y}, B)$,
to the tuple~$\mathbf{y}$ as the \emph{coefficient tuple}, and to the
matrix~$B$ as the \emph{exchange matrix}.
\end{definition}

We  use the notation
$[x]_+ = \max(x,0)$,
$[1,N]=\{1, \dots, N\}$, and
\begin{align*}
\textup{sgn}(x) &=
\begin{cases}
-1 & \text{if $x<0$;}\\
0  & \text{if $x=0$;}\\
 1 & \text{if $x>0$.}
\end{cases}
\end{align*}

\begin{definition}
\label{def:seed-mutation}
Let $(\mathbf{x}, \mathbf{y}, B)$ be a labeled seed in $\mathcal{F}$,
and let $k \in [1,N]$.
The \emph{seed mutation} $\mu_k$ in direction~$k$ transforms
$(\mathbf{x}, \mathbf{y}, B)$ into the labeled seed
$\mu_k(\mathbf{x}, \mathbf{y}, B)=(\mathbf{x}', \mathbf{y}', B')$ defined as follows:
\begin{itemize}
\item
The entries of $B'=(b'_{ij})$ are given by
\begin{equation}
\label{eq:matrix-mutation}
b'_{ij} =
\begin{cases}
-b_{ij} & \text{if $i=k$ or $j=k$;} \\[.05in]
b_{ij} + \textup{sgn}(b_{ik}) \ [b_{ik}b_{kj}]_+
 & \text{otherwise.}
\end{cases}
\end{equation}
\item
The coefficient tuple $\mathbf{y}'=(y_1',\dots,y_N')$ is given by
\begin{equation}
\label{eq:y-mutation}
y'_j =
\begin{cases}
y_k^{-1} & \text{if $j = k$};\\[.05in]
y_j y_k^{[b_{kj}]_+}
(y_k \oplus 1)^{- b_{kj}} & \text{if $j \neq k$}.
\end{cases}
\end{equation}
\item
The cluster $\mathbf{x}'=(x_1',\dots,x_N')$ is given by
$x_j'=x_j$ for $j\neq k$,
whereas $x'_k \in \mathcal{F}$ is determined
by the \emph{exchange relation}
\begin{equation}
\label{exchange relation}
x'_k = \frac
{y_k \ \prod x_i^{[b_{ik}]_+}
+ \ \prod x_i^{[-b_{ik}]_+}}{(y_k \oplus 1) x_k} \, .
\end{equation}
\end{itemize}
\end{definition}

We say that two exchange matrices $B$ and $B'$ are {\it mutation-equivalent}
if one can get from $B$ to $B'$ by a sequence of mutations. A sequence of mutations $\mu_d,\mu_e,\mu_d,\mu_e\ldots$ using only mutations in two directions $d$ and $e$ is called a rank 2 mutation sequence.
\begin{definition}
\label{def:patterns}
Consider the \emph{$N$-regular tree}~$\mathbb{T}_N$
whose edges are labeled by the numbers $1, \dots, N$,
so that the $N$ edges emanating from each vertex receive
different labels.
A \emph{cluster pattern}  is an assignment
of a labeled seed $\Sigma_t=(\mathbf{x}_t, \mathbf{y}_t, B_t)$
to every vertex $t \in \mathbb{T}_N$, such that the seeds assigned to the
endpoints of any edge $t {k\over\quad} t'$ are obtained from each
other by the seed mutation in direction~$k$.
The components of $\Sigma_t$ are written as:
\begin{equation}
\label{eq:seed-labeling}
\mathbf{x}_t = (x_{1;t}\,,\dots,x_{N;t})\,,\quad
\mathbf{y}_t = (y_{1;t}\,,\dots,y_{N;t})\,,\quad
B_t = (b^t_{ij})\,.
\end{equation}
\end{definition}

Clearly, a cluster pattern  is uniquely determined
by an arbitrary  seed.

\begin{definition}
\label{def:cluster-algebra}
Given a cluster pattern, we denote
\begin{equation}
\label{eq:cluster-variables}
\mathcal{X}
= \bigcup_{t \in\mathbb{T}_N} \mathbf{x}_t
= \{ x_{i;t}\,:\, t \in \mathbb{T}_N\,,\ 1\leq i\leq N \} \ ,
\end{equation}
the union of clusters from all seeds in the pattern.
The elements $x_{i;t}\in \mathcal{X}$ are called \emph{cluster variables}.
The
\emph{cluster algebra} $\mathcal{A}$ associated with a
given pattern is the $\ZZ \mathbb{P}$-subalgebra of the ambient field $\mathcal{F}$
generated by all cluster variables: $\mathcal{A} = \ZZ \mathbb{P}[\mathcal{X}]$.
We denote $\mathcal{A} = \mathcal{A}(\mathbf{x}, \mathbf{y}, B)$, where
$(\mathbf{x},\mathbf{y},B)$
is any seed in the underlying cluster pattern.
\end{definition}

The cluster algebra is called \emph{skew-symmetric} if the matrix $B$ is skew-symmetric. In this case, it is often convenient to represent the $N\times N$ matrix $B$ by a quiver $Q_B$ with vertices $1,2,\ldots, N$ and $[b_{ij}]_+$ arrows from vertex $i$ to vertex $j$.

In \cite{FZ}, Fomin and Zelevinsky proved the remarkable {\it Laurent phenomenon} and posed the following {\it positivity conjecture}.
\begin{theorem} [Laurent Phenomenon]
\label{Laurent}
 For any cluster algebra $\mathcal{A}$ and any seed $\Sigma_t$, each cluster variable $x$  is a Laurent polynomial over $\mathbb{ZP}$ in the cluster variables from $\mathbf{x}_t = (x_{1;t} , . . . , x_{N;t} )$.\end{theorem}

\begin{conjecture}[Positivity Conjecture] For any cluster algebra $\mathcal{A}$, any seed $\Sigma_t$, and any cluster variable $x$, the Laurent polynomial  expansion of $x$ in the cluster $\mathbf{x}_t$ has coefficients which are nonnegative integer linear combinations of elements in $\mathbb{P}$.
	\end{conjecture}
	
Our main result is the proof of this conjecture for skew-symmetric cluster algebras.

\section{Expansion formulas}\label{sect rank 2}
In this section, we recall from \cite{LS3} how  to compute the  Laurent expansions of those cluster variables which are obtained from the initial cluster by a mutation sequence involving only two vertices. The main tools are the rank 2 formula from \cite{LS} (in the parametrization of \cite{LLZ}) and the rank 2 formula from \cite{L}.

Let $r$ be a positive integer and let $\{c_n^{[r]}\}_{n\in\mathbb{Z}}$ be the sequence  defined by the recurrence relation $$c_n^{[r]}=rc_{n-1}^{[r]} -c_{n-2}^{[r]},$$ with the initial condition $c^{[r]}_1=0$, $c^{[r]}_2=1$. 
For example, if $r=2$ then $c^{[r]}_n=n-1$;
if $r=3$, the sequence $c_n^{[r]}$ takes the following values:
\[\ldots,-3,-1,0,1,3,8,21,55,144,\ldots\]

\begin{lemma}\cite[Lemma 3.1]{LS3}
 \label{lem cn} Let $n\geq 3$. We have
 $c_{n-1}^{[r]}c_{n+k-3}^{[r]} - c_{n+k-2}^{[r]}c_{n-2}^{[r]}=c_k^{[r]}$ for $k\in \mathbb{Z}$. In particular, we have
 $(c_{n-1}^{[r]})^2 -c_n^{[r]}c_{n-2}^{[r]}=1$.
\end{lemma}
%

\subsection{Non-acyclic mutation classes of rank 3} 
We start by collecting some basic results on quivers of rank 3. First let us recall how mutations act on a rank 3 quiver. Given a quiver 
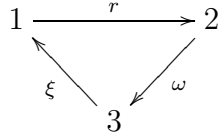
\[\xymatrix{1\ar[rr]^\newr&&2\ar[ld]^\newt\\&3\ar[lu]^\news}\]
 where $\newr,\news,\newt\ge 0$ denote the number of arrows, then its mutation in 1 is the quiver
 
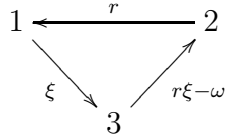
\[\xymatrix{1\ar[rd]_\news&&2\ar[ll]_\newr \\ & 3\ar[ru]_{\newr\news-\newt}}\]
 where we agree that if $\newr\news-\newt<0$ then there are $|\newr\news-\newt|$ arrows from 2 to 3. 
 \begin{lemma} 
 \label{lem nonacyclic} 
Let  $Q$ be the quiver  
 
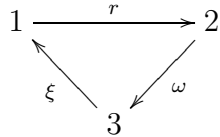
\[\xymatrix{1\ar[rr]^\newr&&2\ar[ld]^\newt\\&3\ar[lu]^\news}\]
 where $\newr\ge 0$ and $\news,\newt\in \mathbb{Z}$ denote the number of arrows and suppose that one of $\news, \newt$ is non-negative. Consider the mutation sequence 
 $$\xymatrix{Q=Q_0\ar@{-}[r]^-1&Q_1\ar@{-}[r]^2&Q_2\ar@{-}[r]^1&Q_3\ar@{-}[r]^2&\cdots }$$ 
 Then $Q_n$ is
 
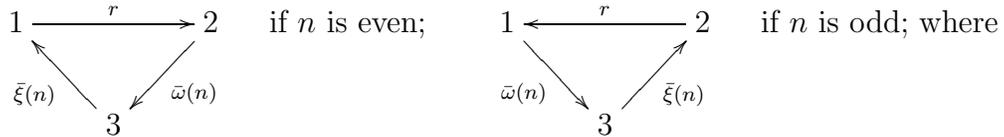
\[\xymatrix{1\ar[rr]^\newr&&2\ar[ld]^{\bar \newt(n)}\\&3\ar[lu]^{\bar \news(n)}} \textup{ \quad if $n$ is even;}\qquad
 \xymatrix{1\ar[rd]_{\bar \newt(n)}&&2\ar[ll]_{\newr}\\&3\ar[ru]_{\bar \news(n)}} \textup{ \quad if $n$ is odd;  where }\]
\[\left\{\begin{array}{llll} 
\bar \news(n)= {c_{n+2}^{[\newr]}\news-c_{n+1}^{[\newr]}\newt} & \textup{and} & \bar \newt(n)=c_{n+1}^{[\newr]}\news-c_{n}^{[\newr]}\newt, & \textup{ if } c_{ \ell+1}^{[\newr]}\news-c_{ \ell}^{[\newr]}\newt>0 \textup{ for } 1\le \ell\le n; \\
 \bar \news(n)= {c_{n-1}^{[\newr]}\newt-c_{n}^{[\newr]}\news}  & \textup{and} & \bar \newt(n)=c_{n+1}^{[\newr]}\news-c_{n}^{[\newr]}\newt, 
 & \textup{ if }   c_{n+1}^{[\newr]}\news- c_{n}^{[\newr]}\newt\leq 0 \textup{ and }\\
 &&&\ c_{ \ell}^{[\newr]}\news -c_{ \ell-1}^{[\newr]}\newt> 0 \textup{ for } 1\le  \ell\le n; \\ 
 \bar \news(n)= {c_{n-1}^{[\newr]}\newt-c_{n}^{[\newr]}\news} & \textup{and} & \bar \newt(n)=c_{n-2}^{[\newr]}\newt-c_{n-1}^{[\newr]}\news, &\textup{ if }
 c_{n}^{[\newr]}\news-c_{n-1}^{[\newr]}\newt\leq 0 \textup{ and $n\ge 2$}.
 \end{array} \right.\] 

\noindent If both $\news<0$ and $\newt<0$ then 
\[ \bar \news(n)= {-c_{n+1}^{[\newr]}\newt-c_{n}^{[\newr]}\news} \quad \textup{and} \quad \bar \newt(n)=-c_{n}^{[\newr]}\newt-c_{n-1}^{[\newr]}\news. \]
\end{lemma}

\begin{remark}
One may rephrase in terms of the following notation: let
$$
\bar{s}(r,\xi,\omega,n)=\bar{s}(n):=c_{n+1}^{[r]}\xi-c_{n}^{[r]}\omega
$$  
for $n\geq 0$. Then we have
$$
\left\{
\begin{array}{ll}
\bar{\xi}(n)=\bar{s}(n+1), \ \bar{\omega}(n)=\bar{s}(n) & \text{ if }\bar{s}(1),...,\bar{s}(n)>0;\\
\bar{\xi}(n)=-\bar{s}(n-1), \ \bar{\omega}(n)=\bar{s}(n) & \text{ if }\bar{s}(n)\leq 0\text{ and }\bar{s}(1),...,\bar{s}(n-1)>0;\\
\bar{\xi}(n)=-\bar{s}(n-1), \ \bar{\omega}(n)=-\bar{s}(n-2) & \text{ if }\bar{s}(n-1)\leq 0\text{ and }n\geq 2.
\end{array}
\right.
$$
\end{remark}

\begin{remark} 
(1) The three cases  in the Lemma, when one of $\news,\newt$ is negative,  arise from the different possibilities for the orientation of the arrows at vertex 3. 
 In the first case, all quivers $Q_0,\ldots,Q_{n-1}$ are cyclically oriented; in the second case, the quivers $Q_0,\ldots,Q_{n-2}$ are cyclic and the quivers $Q_{n-1},Q_n,Q_{n+1}$ are acyclic. In the third case, if $\newr>1$, there exists $m<n$ such that $Q_{m-1},Q_m,Q_{m+1}$  are acyclic, and the quivers $Q_{p}$, with $p>m+1$, are cyclic.
 
 (2) If $\newr>1$ the three cases exhaust all the possibilities. Indeed, if $c_{n+1}^{[\newr]}\news-c_n^{[\newr]}\omega\le 0$ and $c_n^{[\newr]} \news- c_{n-1}^{[\newr]}\omega > 0$, then 
 \[\frac{c_{n+1}^{[\newr]}}{c_{n}^{[\newr]}} \le \frac{\omega}{\news} <  
 \frac{c_{n}^{[\newr]}}{c_{n-1}^{[\newr]}}
 <\frac{c_{n-1}^{[\newr]}}{c_{n-2}^{[\newr]}}
 <\cdots \quad \textup{ if $\news>0$},\] 
  \[\frac{c_{n}^{[\newr]}} {c_{n+1}^{[\newr]}}\ge \frac{\news}{\omega}>   
 \frac{c_{n-1}^{[\newr]}}{c_{n}^{[\newr]}}
 >\frac{c_{n-2}^{[\newr]}}{c_{n-1}^{[\newr]}}
 >\cdots \quad \textup{ if $\omega>0$},\]  and thus $c_\ell^{[\newr]} \news- c_{\ell-1}^{[\newr]}\omega > 0$, for $1\le\ell\le n$.

 If $\newr=0,1$ then $c_n^{[\newr]}$ is periodic. More precisely,  \[c_n^{[\newr]}=\left\{\begin{array}{ll} 
 \ldots 0,1,1,0,-1,-1,0,1,1,0,-1,-1,\ldots  & \textup{for $\newr=1 $}; \\
\ldots 0,1,0,-1,0,1,0,-1,\ldots & \textup{for $\newr=0 $.} \end{array}\right.\]

(3) The case where both $\news<0$ and $\newt<0$ is listed here for the sake of completeness, but this case is not used in the rest of the paper, because the quiver obtained after one mutation belongs to the other case. Note that if  $ \ell=1$ the condition  $c_{ \ell+1}^{[\newr]}\news-c_{ \ell}^{[\newr]}\newt>0$ becomes $\news>0$, and the condition $ c_{ \ell-1}^{[\newr]}\newt-c_{ \ell}^{[\newr]}\news<0$ becomes $\omega>0$.
\end{remark}

\begin{proof}
This is generalization of a result in \cite{BBH}, where  the first case is considered.     If $c_{ \ell+1}^{[\newr]}\news-c_{ \ell}^{[\newr]}\newt>0$, for $1\le  \ell\le n$, we proceed by induction on $n$. 
For $n=1$, the quiver $Q_1$, obtained from $Q$ by mutation in 1, is the following:
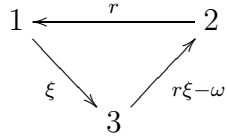
\[ \xymatrix{1\ar[rd]_{\news}&&2\ar[ll]_{\newr}\\&3\ar[ru]_{\newr\news- \newt}}\]
 and for $n=2$, the quiver $Q_2$, obtained from $Q$ by mutation in 1 and 2, is the following:
 
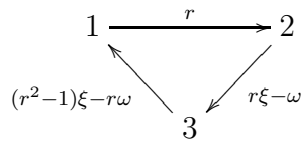
\[\xymatrix{1\ar[rr]^\newr&&2\ar[ld]^{\newr\news-\newt}\\&3\ar[lu]^{(\newr^2-1)\news-\newr\newt}} \]
In both cases, the result follows from 
$c_1^{[\newr]}=0, c_2^{[\newr]}=1, c_3^{[\newr]}=\newr, c_4^{[\newr]}=\newr^2-1$. 

Suppose that $n>2$. If $n$ is odd then by induction we know that the quiver $Q_n$ is obtained by mutating the following quiver in  vertex 1:
\[\xymatrix{1\ar[rr]^\newr&&2\ar[ld]^{\bar \newt(n-1)}\\&3\ar[lu]^{\bar \news(n-1)}}\]
and the result follows from $\bar \newt(n)=\bar \news(n-1) $ and 
\begin{equation}\label{20130223eq1}
\newr \bar \news(n-1)-\bar \newt(n-1)= 
{\newr c_{n+1}^{[\newr]}\news-\newr c_{n}^{[\newr]}\newt -c_{n}^{[\newr]}\news+c_{n-1}^{[\newr]}\newt}
={c_{n+2}^{[\newr]}\news-c_{n+1}^{[\newr]}\newt}
=\bar \news(n).
\end{equation}
The proof is similar in the case where $n$ is even.

 If $c_{n+1}^{[\newr]}\news-c_{n}^{[\newr]}\newt\leq 0$  and  $c_{ \ell-1}^{[\newr]}\newt-c_{ \ell}^{[\newr]}\news<0$, for $1\le  \ell\le n$, and if $n$ is odd, then  the quiver $Q_n$  is obtained by mutating the following quiver in  vertex 1:
\[\xymatrix{1\ar[rr]^\newr\ar[rd]_{-\bar \news(n-1)}&&2\ar[ld]^{\bar \newt(n-1)}\\&3}\]
and the result follows from $\bar \newt(n)=\bar \news(n-1) $ and $\bar \news(n)=-\bar \newt(n-1)$. So the resulting quiver  $Q_n$ is 
\[\xymatrix{1&&2\ar[ll]_{\newr}\ar[ld]^{-\bar \news(n)}\\&3\ar[lu]^{-\bar \newt(n)}}\]
and $c_{n}^{[\newr]}\newt-c_{n+1}^{[\newr]}\news\geq 0$. This shows the statement of the Lemma in the second case. 

Now mutating in vertex 2 results in $Q_{n+1}$
\[\xymatrix{1\ar[rr]^\newr&&2\\&3\ar[lu]^{\bar \news(n+1)}\ar[ru]_{-\bar \newt(n+1)}}\]
where $\bar \newt(n+1)=\bar \news(n)=c_{n-1}^{[\newr]}\newt-c_{n}^{[\newr]}\news $ and $\bar \news(n+1)=-\bar \newt(n)=c_{n}^{[\newr]}\newt-c_{n+1}^{[\newr]}\news$.
This proves the third case of the Lemma.
 
 If $n$ is even then the proof is similar. 
 
 For the case where both $\news<0$ and $\newt<0$, it is easy to check the claimed identity for $n=1,2$, and $Q_n$ is non-acyclic for $n\ge 2$, thus the same proof as above applies.
\end{proof}
\begin{lemma}
 \label{rem 12}
 In the situation of Lemma \ref{lem nonacyclic}, if $\newr\ge 2$ and $\news\ge\omega>0$, then $Q_n$ is cyclically oriented for all $n\ge 0$.
\end{lemma}
\begin{proof}
 An easy induction, using $c_{ \ell+1}^{[\newr]} =\newr c_{ \ell}^{[\newr]} -c_{ \ell-1}^{[\newr]}$, shows that we never quit the first case of Lemma~\ref{lem nonacyclic}.
\end{proof}
\begin{definition}\label{def 12} Let $Q_0$ be the quiver $\xymatrix@R=5pt{1\ar[rr]^\newr&&2\ar[ld]^{ \newt}\\&3\ar[lu]^{ \news}}$
with $\newr\ge 0$ and $\newt,\news\in \mathbb{Z}$ and let
$$\xymatrix{Q_0\ar@{-}[r]^1&Q_1\ar@{-}[r]^2&Q_2\ar@{-}[r]^1&Q_3\ar@{.}[r] &Q_m}$$
 be a sequence of quiver mutations in directions 1 and 2.
 The sequence $(Q_0,...,Q_m)$ of quivers is said to be of \emph{almost cyclic type} if one of the following holds:\begin{enumerate}
\item $\newr\geq 2$ and $c_{n}^{[\newr]}\news-c_{n-1}^{[\newr]}\newt>0$ for $1\leq n\leq m$;

\item  $\newr\geq 2$ and $c_{n}^{[\newr]}\news-c_{n-1}^{[\newr]}\newt\leq 0$
for $1\leq n\leq m$;

\item $\newr=1$, $m\leq 2$ and $c_{n}^{[\newr]}\news-c_{n-1}^{[\newr]}\newt>0$ for $1\leq n\leq m$;

\item $\newr=1$, $m\leq 2$ and  $c_{n}^{[\newr]}\news-c_{n-1}^{[\newr]}\newt\leq 0$ for $1\leq n\leq m$.
\item $\newr=0$.
\end{enumerate}

 The sequence $(Q_0,...,Q_m)$ of quivers is said to be of \emph{acyclic type} if one of the following holds:
\begin{itemize}
\item[{\rm (6)}] $\newr\geq 2$, $m\geq 2$, and $c_{n+1}^{[\newr]}\news-c_{n}^{[\newr]}\newt\leq 0 \textup{ and } c_{n-1}^{[\newr]}\newt-c_{n}^{[\newr]}\news<0$  for some $1\leq n\leq m-1$;

\item[{\rm (7)}] $\newr=1$, $m=2$, $\news\leq 0$ and $\newt>0$.
\end{itemize}
\end{definition}

\begin{remark}
 Conditions (3) and (4) are equivalent to conditions (3') and (4') below, respectively.
 \[ \textup{(3')} \quad \newr=1, m=1, \newt> 0 \textup{ or } \newr=1, m=2,\newt> 0, \news> 0\]
 \[ \textup{(4')} \quad \newr=1, m=1, \newt\le 0 \textup{ or } \newr=1, m=2,\newt \le 0, \news\le 0\]
 \end{remark}
 
\begin{remark} {\ralf The quantities  $c_{n}^{[\newr]}\news-c_{n-1}^{[\newr]}\newt$ are the number of arrows in the quivers $Q_0,\ldots,Q_m$, see Lemma~\ref{lem nonacyclic}.}
 If each of the quivers $Q_1,\ldots,Q_{m-1}$ has an oriented cycle then the sequence $(Q_0,\ldots,Q_m)$ is of almost cyclic type. Thus being of almost cyclic type does not depend on the cyclicity of first and the last quiver.
 {\ralf Condition (1) of the definition means that quivers $Q_0,\ldots,Q_{m-2}$ are cyclic, and 
 condition (2) means that quivers $Q_2,\ldots,Q_{m}$ are cyclic.
 } 
 
 Observe that it is possible that certain quivers in an almost cyclic sequence are acyclic. For example the sequence 
 \[\xymatrix@R5pt{&Q_0&&&&Q_1&&&&Q_2\\1\ar[rr]&&2 &&1&&2\ar[ll] &&1\ar[rr]&&2\ar[ld]\\
 &3\ar[ru]&&&&3\ar[ru] &&&&3\ar[lu]}\]
 satisfies condition (4), and is therefore almost cyclic.
\end{remark}
\begin{remark}
 If $\newr\ge 2$ then conditions (1), (2) and (6) exhaust all possibilities. Thus in this case the sequence $(Q_0, \ldots, Q_m)$ is either of almost cyclic type or of acyclic type.
\end{remark}
\begin{remark}
 In the case where $\newr=1$ we only consider sequences with $m\le 2$. The reason for this is that in this case the quiver $Q_5$ is the same as the quiver $Q_0$ with the vertices 1 and 2 switched. Thus the quivers obtained from $Q_0$ by a sequence of length $m$ with $3\le m\le 5$ are the same as the quivers obtained from $Q_5 $ by a sequence of length $5-m$ with $0\le 5-m\le 2$.
\end{remark}
\begin{remark}
 \label{Rem acyclic}
 If the sequence is of acyclic type satisfying condition (6) with some $n\le m-1$ then all quivers $Q_{n+1},\ldots,Q_m$ satisfy the third condition of Lemma~\ref{lem nonacyclic}.
\end{remark}

\subsection{Compatible pairs} 
Let $(a_1, a_2)$ be a pair of nonnegative integers.
A \emph{Dyck path} of type $a_1\times a_2$ is a lattice path
from $(0, 0)$  to $(a_1,a_2)$ that
never goes above the main diagonal joining $(0,0)$ and $(a_1,a_2)$.
Among the Dyck paths of a given type $a_1\times a_2$, there is a (unique) \emph{maximal} one denoted by
$\mathcal{D} = \mathcal{D}^{a_1\times a_2}$.
It is defined by the property that any lattice point strictly above $\mathcal{D}$ is also strictly above the main diagonal.

{\ralf It will be convenient to extend this definition to negative integers $a_1,a_2$. If $a_1<0$ then the notation 
$ \mathcal{D}^{a_1\times a_2}$ means  $\mathcal{D}^{0 \times a_2}$,
and, similarly, if $a_2<0$ then the notation 
$ \mathcal{D}^{a_1\times a_2}$ means  $\mathcal{D}^{a_1 \times 0}$. If both $a_1,a_2<0$ then $ \mathcal{D}^{a_1\times a_2}$ means  $\mathcal{D}^{0 \times 0}$.
}

Let $\mathcal{D}=\mathcal{D}^{a_1\times a_2}$.  Let $\mathcal{D}_1=\mathcal{D}^{a_1\times a_2}_1=\{u_1,\dots,u_{a_1}\}$ be the set of horizontal edges of $\mathcal{D}$ indexed from left to right, and $\mathcal{D}_2=\mathcal{D}^{a_1\times a_2}_2=\{v_1,\dots, v_{a_2}\}$ the set of vertical edges of $\mathcal{D}$ indexed from bottom to top.
Given any points $A$ and $B$ on $\mathcal{D}$, let $AB$ be the subpath starting from $A$, and going in the Northeast direction until it reaches $B$ (if we reach $(a_1,a_2)$ first, we continue from $(0,0)$). By convention, if $A=B$, then $AA$ is the subpath that starts from $A$, then passes $(a_1,a_2)$ and ends at $A$. If we represent a subpath of $\mathcal{D}$ by its set of edges, then for $A=(i,j)$ and $B=(i',j')$, we have
$$
AB=
\begin{cases}
\{u_k, v_\ell: i < k \leq i', j < \ell \leq j'\}, \quad\textrm{if $B$ is to the Northeast of $A$};\\
\mathcal{D} - \{u_k, v_\ell: i' < k \leq i, j' < \ell \leq j\}, \quad\textrm{otherwise}.
\end{cases}
$$
We denote by $(AB)_1$ the set of horizontal edges in $AB$, and by $(AB)_2$ the set of vertical edges in $AB$.
Also let $AB^\circ$ denote the set of lattice points on the subpath $AB$ excluding the endpoints $A$ and $B$ (here $(0,0)$ and $(a_1,a_2)$ are regarded as the same point).

Here is an example for $(a_1,a_2)=(6,4)$.

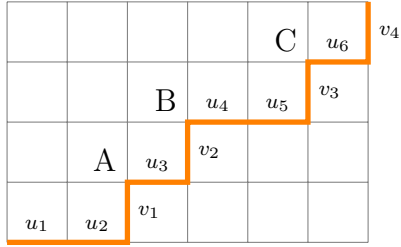
\begin{figure}[h]
\begin{tikzpicture}[scale=.8]
\draw[step=1,color=gray] (0,0) grid (6,4);
\draw[line width=2,color=orange] (0,0)--(2,0)--(2,1)--(3,1)--(3,2)--(5,2)--(5,3)--(6,3)--(6,4);
%
\draw (0.5,0) node[anchor=south]  {\tiny$u_1$};
\draw (1.5,0) node[anchor=south]  {\tiny$u_2$};
\draw (2.5,1) node[anchor=south]  {\tiny$u_3$};
\draw (3.5,2) node[anchor=south]  {\tiny$u_4$};
\draw (4.5,2) node[anchor=south]  {\tiny$u_5$};
\draw (5.5,3) node[anchor=south]  {\tiny$u_6$};
\draw (6,3.5) node[anchor=west]  {\tiny$v_4$};
\draw (5,2.5) node[anchor=west]  {\tiny$v_3$};
\draw (3,1.5) node[anchor=west]  {\tiny$v_2$};
\draw (2,.5) node[anchor=west]  {\tiny$v_1$};
\draw (2,1) node[anchor=south east] {A};
\draw (3,2) node[anchor=south east] {B};
\draw (5,3) node[anchor=south east] {C};
\end{tikzpicture}
\caption{A maximal Dyck path.}
\label{fig:Dyck-path}
\end{figure}

Let $A=(2,1)$, $B=(3,2)$ and $C=(5,3)$. Then
$$
(AB)_1=\{u_3\}, \,\, (AB)_2=\{v_2\}, \,\, 
(BA)_1=\{u_4,u_5,u_6,u_1,u_2\}, \, \, (BA)_2=\{v_3,v_4,v_1\} \ . 
$$
The point $C$ is in $BA^\circ$ but not in $AB^\circ$. The subpath $AA$ has length 10 (not 0).

\begin{definition}
\label{df:compatible} Let $r$ be a positive integer.
For $S_1\subseteq \mathcal{D}_1$, $S_2\subseteq \mathcal{D}_2$, we say that the pair $(S_1,S_2)$ is $r$-\emph{compatible} if for every $u\in S_1$ and $v\in S_2$, denoting by $E$ the left endpoint of $u$ and $F$ the upper endpoint of $v$, there exists a lattice point $A\in EF^\circ$ such that
\begin{equation}
\label{0407df:comp}
|(AF)_1|=r|(AF)_2\cap S_2|\textrm{\; or\; }|(EA)_2|=r|(EA)_1\cap S_1|.\end{equation}
\end{definition}

\begin{remark}
 We often say \emph{compatible} instead of $r$-\emph{compatible} if $r$ is clear from the context.
\end{remark}

{\ralf For $r=3$,  the pair $(\{u_1,u_2\}, \{v_3,v_4\})$ is not compatible in $\mathcal{D}^{6\times 4}$ (Figure 2),  but is compatible in $\mathcal{D}^{7\times 4}$ (Figure 3).}
$$\begin{tikzpicture}[scale=.5]
\draw[step=1,color=gray] (0,0) grid (6,4);
\draw[line width=2,color=orange] (0,0)--(2,0)--(2,1)--(3,1)--(3,2)--(5,2)--(5,3)--(6,3)--(6,4);
\draw (0,0) node[anchor=east]  {$E$};
\draw (6,4) node[anchor=south]  {$F$};
%
\draw[line width=2pt] (0,0)--(1,0);
\draw (0.5,1) node[anchor=north]  {\tiny$u_1$};
\draw[line width=2pt] (1,0)--(2,0);
\draw (1.5,1) node[anchor=north]  {\tiny$u_2$};
\draw[line width=2pt] (6,4)--(6,3);
\draw (6,3.5) node[anchor=west]  {\tiny$v_4$};
\draw[line width=2pt] (5,3)--(5,2);
\draw (5,2.5) node[anchor=west]  {\tiny$v_3$};
\draw (3,-1) node[anchor=north] {Figure 2};
\begin{scope}[shift={(12,0)}]
\draw[step=1,color=gray] (0,0) grid (7,4);
\draw[line width=2,color=orange] (0,0)--(2,0)--(2,1)--(4,1)--(4,2)--(6,2)--(6,3)--(7,3)--(7,4);
\draw (0,0) node[anchor=east]  {$E$};
\draw (7,4) node[anchor=south]  {$F$};
%
\draw[line width=2pt] (0,0)--(1,0);
\draw (0.5,1) node[anchor=north]  {\tiny$u_1$};
\draw[line width=2pt] (1,0)--(2,0);
\draw (1.5,1) node[anchor=north]  {\tiny$u_2$};
\draw[line width=2pt] (7,4)--(7,3);
\draw (7,3.5) node[anchor=west]  {\tiny$v_4$};
\draw[line width=2pt] (6,3)--(6,2);
\draw (6,2.5) node[anchor=west]  {\tiny$v_3$};
\draw (3,-1) node[anchor=north] {Figure 3};
\end{scope}
\end{tikzpicture}
$$

\subsection{Expansion formulas}\label{sect 3.3}
The formulas in this subsection are proved in \cite{LS3}.

Let $\mathcal{A}$ be a skew-symmetric cluster algebra of geometric type 
of arbitrary rank $N$.  Let $\mathbf{x}_t=\{x_{1;t},\ldots,x_{N;t}\}$ and $\mathbf{x}_{t'}=\{x_{1;t'},\ldots,x_{N;t'}\}$ be two clusters such that there exists a sequence $\mu$ of mutations in directions 1 and 2, transforming $\mathbf{x}_{t'}$ into $\mathbf{x}_t$.  Suppose that the last mutation in $\mu $ is in direction 1. Thus $\mu$ is of one of the following two forms.
 \[\xymatrix@C=10pt{\mu =&t'\ar@{-}[r]^1&\cdot\ar@{-}[r]^2&\cdot\ar@{-}[r]^1&\cdot\ar@{.}[r]&\cdot\ar@{-}[r]^2&\cdot\ar@{-}[r]^1&t   
 &&\textup{or}&&
\mu =&t'\ar@{-}[r]^2&\cdot\ar@{-}[r]^1&\cdot\ar@{-}[r]^2&\cdot\ar@{.}[r]&\cdot\ar@{-}[r]^2&\cdot\ar@{-}[r]^1&t}\]
Observe that $x_{f,t'}=x_{f,t}$ for all $f=3,4,\ldots,N$.
Denote by $n$ the number of seeds in the sequence $\mu$ including $t'$ and $t$. 
Let $Q_t$ be the quiver of the seed at $t$, let $r$ be the number of arrows $1\to 2$, where we suppose without loss of generality that $r\ge 0$, and let $\xi_f\in\mathbb{Z}$ be the number of arrows $f\to 1$, and $\omega_f\in\mathbb{Z}$  the number of arrows $2\to f$ in $Q_t$.
Given $p,q\ge 0$, define
$$A_i=A_i(p,q)=\left\{\begin{array}{ll}
p c_{i+1}^{[r]} +q c_{i}^{[r]} &\textup{if }\mu =\xymatrix@C=10pt{t'\ar@{-}[r]^1&\cdot\ar@{-}[r]^2&\cdot\ar@{-}[r]^1&\cdot\ar@{.}[r]&\cdot\ar@{-}[r]^2&\cdot\ar@{-}[r]^1&t} ;\\
q c_{i+1}^{[r]} +p c_{i}^{[r]} &\textup{if }\mu =\xymatrix@C=10pt{t'\ar@{-}[r]^2&\cdot\ar@{-}[r]^1&\cdot\ar@{-}[r]^2&\cdot\ar@{.}[r]&\cdot\ar@{-}[r]^2&\cdot\ar@{-}[r]^1&t};
\end{array}\right. $$

and
$$\alpha=\left\{\begin{array}{ll}
q  &\textup{if }\mu =\xymatrix@C=10pt{t'\ar@{-}[r]^1&\cdot\ar@{-}[r]^2&\cdot\ar@{-}[r]^1&\cdot\ar@{.}[r]&\cdot\ar@{-}[r]^2&\cdot\ar@{-}[r]^1&t} ;\\
p &\textup{if }\mu =\xymatrix@C=10pt{t'\ar@{-}[r]^2&\cdot\ar@{-}[r]^1&\cdot\ar@{-}[r]^2&\cdot\ar@{.}[r]&\cdot\ar@{-}[r]^2&\cdot\ar@{-}[r]^1&t}.
\end{array}\right. $$

  The following Lemma is a straightforward consequence of Lemma \ref{lem cn}.

\begin{lemma}\label{negone}\cite[Lemma 3.12]{LS3}
For any $i$, we have
\begin{itemize}
\item [(a)] $ A_i=rA_{i-1}-A_{i-2},$
\item [(b)] $A_{i}^2 - A_{i+1}A_{i-1}= p^2 + q^2 + rpq.$\qed
\end{itemize}
\end{lemma}

We are now ready to state our first expansion formula.
\begin{theorem}\label{mainthm12052011}\cite[Theorem 3.13]{LS3} For all $p,q\ge 0$, we have
\[x_{1;t'}^p x_{2;t'}^q = \sum_{(S_1,S_2)}x_{1;t}^{r|S_2|-A_{n-1}}x_{2;t}^{r|S_1|-A_{n-2}} \prod_{f=3}^N x_{f;t}^{\xi_f(A_{n-1}-|S_1|)-\omega_f|S_2|-M_f} \]
where 
$M_f =\textup{min}_{(S_1,S_2)}\, \{\xi_f(A_{n-1}-|S_1|)-\omega_f|S_2| \}$, and
the sum is over all $(S_1=\cup_{i=1}^{p+q}\, S_{1}^i,S_2=\cup_{i=1}^{p+q}\, S_{2}^i)$ such that  
$$(S_{1}^i,S_{2}^i) \textup{ is a compatible pair in }\left\{\begin{array}{ll}\mathcal{D}^{c^{[r]} _{n-1}\times c^{[r]} _{n-2}}  &\textup{ if $1\le i \le \alpha$}; \\  
 \mathcal{D}^{c^{[r]} _{n}\times c^{[r]} _{n-1}} &\textup{ if $\alpha+1\le i\le p+q$.}
\end{array}\right.$$
\end{theorem}

\begin{remark}\label{remark 3.14}
 It can be shown that the summation on the right hand side in Theorem \ref{mainthm12052011} can be taken over all compatible pairs in $\mathcal{D}^{A_{n-1}\times A_{n-2}}$ instead, without changing the sum, see \cite[Theorem 1.11]{LLZ}.
  \end{remark}

\begin{remark}
 The  term $M_f$ in the exponent of $x_{f;t}$ in Theorem  \ref{mainthm12052011} comes from Fomin-Zelevinsky's separation of addition formula \cite[Theorem 3.7]{FZ4}. 
\end{remark}

We shall need a precise value for $M_f$. As a first step, we determine which   pair $(S_1,S_2)$ can realize the minimum $M_f$.
  Let 
\[\begin{array}{rcl} a_{1,i}&=&\left\{\begin{array}{ll}
c^{[r]} _{n-1} &\textup{if }1\leq i\leq \alpha;\\
             c^{[r]} _n       &\textup{if } \alpha+1\leq i\leq p+q.
\end{array}\right. \\
a_{2,i}&=& \left\{\begin{array}{ll}
c^{[r]} _{n-2}&\textup{if }1\leq i\leq \alpha;\\
             c^{[r]} _{n-1} &\textup{if } \alpha+1\leq i\leq p+q.
             \end{array}\right. 
\end{array}\]

\begin{lemma}\label{12312011}\cite[Lemma 3.8]{LS3}
In the setting of Theorem~\ref{mainthm12052011}, 
 consider the values 
$\xi_f(A_{n-1} - |S_1|) - \omega_f|S_2|$ obtained from the following three cases:
\begin{itemize}
\item $ S^i_1=\mathcal{D}^{a_{1,i}\times a_{2,i}}_1$ and $S^i_2=\emptyset$ for all $1\leq i\leq p+q$
\item $S^i_1=\emptyset$ and $S^i_2=\emptyset$ for all $1\leq i\leq p+q$
\item $S^i_1=\emptyset$ and $S^i_2=\mathcal{D}^{a_{1,i}\times a_{2,i}}_2$ for all $1\leq i\leq p+q$
\end{itemize}
Then one of the (possibly non-distinct) three values is equal to $M_f$. 
\end{lemma}
\begin{proof} This follows from the proof of Lemma 3.8 in \cite{LS3} replacing $c^{[r]} _n$ by $A_n$. \end{proof}

\bigskip
Before stating our second formula, we need to introduce some notation.
For arbitrary (possibly negative) integers $A, B$, we define the modified binomial coefficient as follows.
$$\left[\begin{array}{c}{A } \\{B} \end{array}\right] := \left\{ \begin{array}{ll}  \prod_{i=0}^{A-B-1} \frac{A-i}{A-B-i}, & \text{ if }A > B\\ \, & \,  \\   1, & \text{ if }A=B \\ \, & \, \\  0, & \text{ if }A<B.  \end{array}  \right.$$  

If $A\geq 0$ then $\left[\begin{array}{c}{A } \\{B} \end{array}\right]=\gchoose{A}{A-B}$ is just the usual binomial coefficient, in particular  $\left[\begin{array}{c}{A } \\{B} \end{array}\right]= 0$ if $A\ge0$ and $B<0$.

For a sequence of integers $(\tau_j)$ (respectively $(\tau'_j)$), we define a sequence  of weighted partial sums $(s_i)$ (respectively $(s_i'))$ as follows:
$$ \begin{array}{rcccl}
s_0=0,&\quad &s_i=\sum_{j=0}^{i-1} c^{[r]} _{i-j+1} \tau_j = c^{[r]} _{i+1}\tau_0+c^{[r]} _i\tau_1+\cdots+c^{[r]} _2\tau_{i-1};\\ \\
s'_0=0,&\quad &s'_i=\sum_{j=0}^{i-1} c^{[r]} _{i-j+1} \tau_j' = c^{[r]} _{i+1}\tau'_0+c^{[r]} _i\tau'_1+\cdots+c^{[r]} _2\tau'_{i-1}.\end{array}
$$
For example, $s_1=c^{[r]} _2\tau_0=\tau_0$, $s_2= c^{[r]} _3\tau_0+c^{[r]} _2\tau_1=r\tau_0+\tau_1$.

\begin{lemma}
 \label{lem sn}\cite[Lemma 3.15]{LS3}
 $s_n=r s_{n-1} -s_{n-2} +\tau_{n-1}$.
\end{lemma}

\bigskip

\begin{definition}\label{20120121}Let $\mathcal{L}( \tau_0,\tau_1,\cdots,\tau_{n-3} ) $
denote the set of all $( \tau'_0,\tau'_1,\cdots,\tau'_{n-3} )\in \mathbb{Z}^{n-2} $ satisfying the conditions
$$\begin{array}{ll} \textup{(1)} & 0\leq\tau'_i\leq \tau_i \text{ for }0\leq i\leq n-4, \\ \\
\textup{(2)}&   s'_{n-3}=k c^{[r]} _{n-2}\textup{ and }
     s'_{n-2} =k c^{[r]} _{n-1} \text{ for some integer }0\leq k\leq p.\end{array}
$$
\end{definition}

We define a partial order on $\mathcal{L}( \tau_0,\tau_1,\cdots,\tau_{n-3} )$ by
$$( \tau'_0,\tau'_1,\cdots,\tau'_{n-3} ) \leq_{\mathcal{L}} ( \tau''_0,\tau''_1,\cdots,\tau''_{n-3} )\text{ if and only if }\tau'_i\leq \tau''_i \text{ for }0\leq i\leq n-4.
$$
Then let $
\mathcal{L}_{\max}( \tau_0,\tau_1,\cdots,\tau_{n-3} )$ be the set of the maximal elements of $
\mathcal{L}( \tau_0,\tau_1,\cdots,\tau_{n-3} )$ with respect to $\leq_{\mathcal{L}}$.

Our second expansion formula is the following.

\begin{theorem}\label{cor 26}\cite[Theorem 3.17]{LS3}
 \noindent Let $\widetilde{x_{1;t}} $ be the cluster variable obtained by mutating $\mathbf{x}_t$ in direction 1.    Let $\omega'_f=\xi_f$ and $\xi'_f$ be the number of arrows from 1 to $f$, or $f$ to 2 respectively, in the quiver obtained from $Q_t$ by mutating in the vertex 1. Then
 \begin{equation}\label{mainformula2b0303}
x_{1;t'}^p x_{2;t'}^q=  \sum_{ \tau_0,\tau_1,\cdots,\tau_{n-3}}\aligned &\left( \prod_{i=0}^{n-3} \left[\begin{array}{c}{{A_{i+1} - rs_i    } } \\{\tau_i} \end{array}\right]   \right)x_{2;t}^ {rs_{n-3}-A_{n-2}} \widetilde{x_{1;t}}^{A_{n-1}-rs_{n-2}}  \prod_{f=3}^N x_{f;t}^{\xi_f' s_{n-2}-\omega_f's_{n-3}-M'_f},
\endaligned\end{equation}
where  the summation runs over all integers $ \tau_0,...,\tau_{n-3}$ satisfying
\begin{equation}\label{cond502b0303}\left\{
\begin{array}{l} 0\leq \tau_i \leq A_{i+1} - rs_i \, (0\leq i\leq n-4),  \ \tau_{n-3} \leq A_{n-2} - rs_{n-3} , \text{ and } \\
\left(s_{n-2}-s'_{n-2}\right) A_{n-3} \geq \left(s_{n-3} -s_{n-3}'\right) A_{n-2} \text{ for any }(\tau'_0,...,\tau'_{n-3})\in\mathcal{L}_{\max}( \tau_0,\cdots,\tau_{n-3} ),
 \end{array} \right.\end{equation}
and
$M_f'=
0 $ if the sequence of full subquivers on vertices 1, 2, $f$  in $\mu$ from $t'$ to $t$ is of almost cyclic type, and 
$M_f'=  \xi'_f A_{n-2} - \omega'_f A_{n-3},$ if the sequence is  of acyclic type from $t'$ up to $\mu_1(t)$. 
In particular, the exponent of $x_{f;t}$ is non-negative.\end{theorem}
\begin{proof}
 The proof is exactly the same as the proof of \cite[Theorem 3.17]{LS3} except for the $M_f'$ in the exponent of $x_{f,t}$. The precise value for $M_f'$ follows by comparing terms with the formula in Theorem~\ref{mainthm12052011} and using Lemma~\ref{12312011}. The statement about non-negative exponents also follows from Theorem~\ref{mainthm12052011}.
\end{proof}
Combining the two formulas of Theorems \ref{mainthm12052011}
and \ref{cor 26}, we get the following mixed formula which has the advantage that the exponents of $\widetilde{x_{1;t}}, x_{1;t}$ and $x_{f;t}$   are nonnegative.

\begin{theorem}\label{mainthm12062011}\cite[Theorem 3.21]{LS3}  
\begin{equation}\label{mainformula22}\aligned
x_{1;t'}^p x_{2;t'}^q= & \sum_{\begin{array}{c} \scriptstyle\tau_0,\tau_1,\cdots,\tau_{n-3} \\ \scriptstyle A_{n-1} -r s_{n-2}\geq 0  \end{array}  }\aligned &\left( \prod_{i=0}^{n-3}{{A_{i+1} - rs_i    }  \choose{\tau_i} }   \right){\widetilde{x_{1;t}}}^ {A_{n-1}-rs_{n-2}} x_{2;t}^{rs_{n-3}-A_{n-2}}    \prod_{f=3}^N x_{f;t}^{\xi_f' s_{n-2}-\omega_f's_{n-3}-M'_f}, 
\endaligned
\\
&+  \sum_{\begin{array}{c} \scriptstyle  (S_1,S_2)\\ 
 \scriptstyle r|S_2|-A_{n-1} >0 \end{array}  }
 x_{1;t}^{r|S_2|-A_{n-1}}x_{2;t}^{r|S_1|-A_{n-2}}  \prod_{f=3}^N x_{f;t}^{\xi_f(A_{n-1}-|S_1|)-\omega_f|S_2|-M_f},
\endaligned\end{equation}
where $(S_1,S_2)$ are as in Theorem~\ref{mainthm12052011}, and where
$M_f'=
0 $ if the sequence of full subquivers on vertices 1, 2, $f$  in $\mu$ is of almost cyclic type, and 
$ M_f= \xi_f A_{n-1} - \omega_f A_{n-2}$ if  the sequence of full subquivers on vertices 1, 2, $f$  in $\mu$ is of acyclic type. In particular, the exponents of $x_{f;t}$ are non-negative.
\end{theorem}

\begin{corollary}
 \label{cor 2.11} For any cluster monomial $u$ in the variables of $\mathbf{x'}$, there exist two  polynomials 
 \[f_1\in \mathbb{Z}_{\ge 0}\mathbb{P}[ \widetilde{x_{1;t}},x_{2;t}^{\pm1},x_{3;t},\ldots, x_{N;t}] \quad\textup{and} \quad  f_2 \in \mathbb{Z}_{\ge 0}\mathbb{P}[ {x_{1;t}},x_{2;t}^{\pm1},x_{3;t},\ldots, x_{N;t}]\] such that 
 \[u= f_1 +f_2.\]
\end{corollary}
\bigskip
We end this subsection with the following rank 2 result which we will need later.

\begin{theorem}\label{thm01312012}\cite[Theorem 3.26]{LS3} Let $a\geq \frac{A_n}{r}$ be an integer. Then the sum
$$\sum_{ \begin{array}{c}\scriptstyle \tau_0,\tau_1,\cdots,\tau_{n-2}\\ \scriptstyle s_{n-1}=a \end{array} }\prod_{i=0}^{n-2} \left[\begin{array}{c}{{A_{i+1} - rs_i    } } \\{\tau_i} \end{array}\right]   x_1^ {rs_{n-2}-A_{n-1}} x_2^{r(A_{n-1}-a)-A_{n-2}}$$ 
where the summation runs over all integers $\tau_0,\ldots,\tau_{n-2}$ satisfying (\ref{cond502b0303}) with $n$ replaced by $n+1$,
is divisible by $(1+x_1^r)^{ra-A_n}$ and the resulting quotient has nonnegative coefficients.
\end{theorem}

\section{Main result}\label{sect 3}
In this section we present our main results. The positivity conjecture (Theorem~\ref{11192011thm}) follows from the following result.

\begin{theorem}
 \label{thm1}
 Let $\mathcal{A}$ be a skew-symmetric cluster algebra of geometric type 
 and let $\mathbf{x}_{t}=\{x_{1;t},x_{2;t},\ldots,x_{N;t}\}$ be a cluster in $\mathcal{A}$.
Let $u$ be a cluster variable and let  $\mathbf{x}_{t_0} $ be a cluster containing $u$  such that the distance between $t_0$ and $t$ in the exchange tree of labeled seeds is minimal. Let $\mu$ be the unique sequence of mutations relating the seeds $t_0$ and $t$ in the exchange tree of labeled seeds. Denote by $d',d$ the directions of the last two mutations in the sequence $\mu$, thus \[\xymatrix{\mu =&t_0\ar@{.}[r]&\cdot\ar@{-}[r]&t''\ar@{-}[r]^{d'}&t'\ar@{-}[r]^d&t}
\] 
and let $e\in\{1,2,\ldots,n\}, e\ne d,d'$.  
  Let $\widetilde {x_{d;t}}$, $\widetilde {x_{e;t}}$ be the cluster variables obtained from $\mathbf{x}_t$ by mutation in direction $d,e$, respectively, and let  $\widetilde{\widetilde {x_{e;t}}}$ be the cluster variable obtained from $\mathbf{x}_t$ by the two step mutation first in $d$  and then in $e$.
   Then  there exist  polynomials
   \[
\begin{array}{rclcrcl}
A_t\in\mathbb{Z}_{\ge 0}\mathbb{P} [ \widetilde {x_{d;t}},\widetilde{\widetilde {x_{e;t}}};(x_{f;t}^{\pm 1})_{ f\ne d,e}],
&\quad&
B_t\in\mathbb{Z}_{\ge 0}\mathbb{P} [ \widetilde {x_{d;t}}, {{x_{e;t}}};(x_{f;t}^{\pm 1})_{ f\ne d,e}],
\\
\\
C_t\in\mathbb{Z}_{\ge 0}\mathbb{P} [ {x_{d;t}},{{x_{e;t}}};(x_{f;t}^{\pm 1})_{ f\ne d,e}],
&\quad&
D_t\in\mathbb{Z}_{\ge 0}\mathbb{P} [ {x_{d;t}},  \widetilde{x_{e;t}};(x_{f;t}^{\pm 1})_{ f\ne d,e}].
\end{array}
   \]
   such that  
   \[u= A_t+B_t+C_t+D_t.\] 
   Moreover,  
   the polynomials $A_t,B_t,C_t,D_t$ are unique up to intersection of polynomial rings.
   In particular 
   \[ u\in \mathbb{Z}_{\ge 0}\mathbb{P} [ {x_{d;t}},{{x_{e;t}}},  \widetilde {x_{d;t}},\widetilde {x_{e;t}},\widetilde{\widetilde {x_{e;t}}};(x_{f;t}^{\pm 1})_{ f\ne d,e}].\] 
   \end{theorem}
The proof of this theorem is given in section \ref{sect proof}. The positivity conjecture follows easily.

\begin{theorem}[Positivity Conjecture]\label{11192011thm}
Let $\mathcal{A}(Q)$ be a skew-symmetric   cluster algebra,
 let $\mathbf{x}_{t}$ be any cluster and let $u$ be any cluster variable. Then the Laurent expansion of $u$  with respect to the cluster $\mathbf{x}_t$ is a Laurent polynomial in $\mathbf{x}_t$ whose coefficients are non-negative integer linear combinations of elements of $\mathbb{P}$.
\end{theorem}
\begin{proof} Because of Fomin-Zelevinsky's separation of addition formula \cite[Theorem 3.7]{FZ4}, it suffices to prove the result in the case where $\mathcal{A}(Q)$ is of geometric type. Let $\mathbf{x}_{t_0}$ be an arbitrary cluster, and let $u\in\mathbf{x}_{t_0}$ be a cluster variable. Let  $\mu$ be the unique sequence of mutations relating the seed $t_0$ to the  seed $t$ in the exchange tree of labeled seeds. Let $d,e$ be the last two directions in the sequence $\mu$. Consider the maximal rank two mutation subsequence in directions $e,d$ at the end of $\mu$. This subsequence connects $t$ to a seed $t_1'$, and we denote by $t_1$ the seed one step away from $t_1'$ on this subsequence. Thus we have
 \[\xymatrix{\mu =&t_0\ar@{.}[r]&\cdot\ar@{-}[r]^{d'}&t_1'\ar@{-}[r]^{e}&t_1\ar@{-}[r]^d&\cdot\ar@{-}[r]^e&\cdot\ar@{-}[r]^d&\cdot\ar@{.}[r]&\cdot\ar@{-}[r]^d&\cdot\ar@{-}[r]^e&t}
\] 
or 
 \[\xymatrix{\mu =&t_0\ar@{.}[r]&\cdot\ar@{-}[r]^{d'}&t_1'\ar@{-}[r]^{d}&t_1\ar@{-}[r]^e&\cdot\ar@{-}[r]^d&\cdot\ar@{-}[r]^e&\cdot\ar@{.}[r]&\cdot\ar@{-}[r]^d&\cdot\ar@{-}[r]^e&t}
\] 
with $d'\ne d,e$. Then
applying Theorem \ref{thm1} at the seed $t_1$ with respect to the directions $d$ and $e$, we get that  $$u\in 
\mathbb{Z}_{\ge 0}\mathbb{P}[x_{d;t_1},x_{e;t_1},\widetilde {x_{d;t_1}},\widetilde {x_{e;t_1}},\widetilde{\widetilde {x_{d;t_1}}}; (x_{f;t_1}^{\pm 1})_{ f\ne d,e}],$$
or  $$u\in 
\mathbb{Z}_{\ge 0}\mathbb{P}[x_{d;t_1},x_{e;t_1},\widetilde {x_{d;t_1}},\widetilde {x_{e;t_1}},\widetilde{\widetilde {x_{e;t_1}}}; (x_{f;t_1}^{\pm 1})_{ f\ne d,e}].$$
Moreover, if ${f\ne d,e}$ then $x_{f;t'}=x_{f;t}$  is a cluster variable in $\mathbf{x}_t$. On the other hand, each of the variables
$x_{d;t_1},x_{e;t_1},\widetilde {x_{d;t_1}},\widetilde {x_{e;t_1}},$  $\widetilde{\widetilde {x_{d;t_1}}}$, and $\widetilde{\widetilde {x_{e;t_1}}}$ is obtained from the  cluster  $\mathbf{x}_t$ by a mutation 
sequence using only  the two directions $e$ and $d$, and therefore Theorem \ref{mainthm12052011} implies that these variables are Laurent polynomials in $\mathbf{x}_t$ with coefficients in $\mathbb{Z}_{\ge 0}\mathbb{P}$. It follows that, after substitution into $u$, we get an expansion for $u$ as a Laurent polynomial with non-negative coefficients in the initial cluster $\mathbf{x}_t$.
\end{proof}
\begin{remark}
  For $N>2$, one can prove Theorem~\ref{11192011thm}  directly from Theorem~\ref{thm1} without using induction. We prefer the proof above, since it illustrates the inductive nature of Theorem \ref{thm1}.
\end{remark}

\section{Proof of Theorem \ref{thm1}}\label{sect proof}
We use induction on the length $\ell$ of the sequence of mutations $\mu$. 

\noindent If $\ell =0$, then $u=x_{i;t}$ for some $i$, which is of the form $C_t$.

\noindent If $\ell =1$, then $\mu$ consists of a single mutation in direction $d$, and $u=\widetilde{x_{d;t}}$ is of the form $B_t$.

\noindent If $\ell =2 $, then $\mu$ is a sequence of two mutations $\xymatrix{t_0\ar@{-}[r]^{d'}&t'\ar@{-}[r]^d&t}$, with $d'\ne d$ and $e\ne d,d'$. Then
\[ u= (\textup{binomial in $\mathbf{x}_{t'}$})\,x_{d';t}^{-1} =(\textup{binomial in $\left(\mathbf{x}_{t}\setminus\{x_{d;t}\}\right)\cup\{\widetilde{x_{d;t}}\}$})\,x_{d',t}^{-1} 
\]
and this is of the form $B_t$, since $d'\ne d,e $ and the binomial has coefficients in $\mathbb{Z}_{\ge 0}\mathbb{P}$. 

 \noindent Now let $\ell\ge 3$.
 Then $\mu$ is a sequence of the form $\xymatrix{t_0\ar@{.}[r]&\cdot\ar@{-}[r]&t''\ar@{-}[r]^{d'}&t'\ar@{-}[r]^d&t}$, with $d'\ne d$. Consider the maximal rank two mutation subsequence in directions $d,d'$ at the end of $\mu$. This subsequence connects $t$ to a seed $t^{**}$, and we denote by $t^*$ the seed one step away from $t^{**}$ on this subsequence. Thus we have
 \[\xymatrix{\mu =&t_0\ar@{.}[r]&\cdot\ar@{-}[r]^{d''}&t^{**}\ar@{-}[r]^{d}&t^*\ar@{-}[r]^{d'}&\cdot\ar@{-}[r]^d&\cdot\ar@{-}[r]^{d'}&\cdot\ar@{.}[r]&\cdot\ar@{-}[r]^{d'}&\cdot\ar@{-}[r]^d&t}
\] or
 \[\xymatrix{\mu =&t_0\ar@{.}[r]&\cdot\ar@{-}[r]^{d''}&t^{**}\ar@{-}[r]^{d'}&t^*\ar@{-}[r]^d&\cdot\ar@{-}[r]^{d'}&\cdot\ar@{-}[r]^d&\cdot\ar@{.}[r]&\cdot\ar@{-}[r]^{d'}&\cdot\ar@{-}[r]^d&t}
\] 
with $d''\ne d,d'$. 

Consider the subsequence of mutations $\mu^*$ connecting $t_0$ to $t^*$. Since this sequence is shorter than the sequence $\mu$, we can conclude by induction that the statement holds in the seed $t^*$ with directions $d,d'$. Thus, {\ralf if $\mu $ is as in the first case,}
   \begin{equation}\label{eq41} u= A_{t^*}+B_{t^*}+C_{t^*}+D_{t^*},\end{equation}
   where 
   \[
\begin{array}{rclcrcl}
A_{t^*}\in\mathbb{Z}_{\ge 0}\mathbb{P} [ {\widetilde{\widetilde{x_{d';t^*}}}}, \widetilde{x_{d;t^*}};(x_{f;t^*}^{\pm 1})_{ f\ne d,d'}],
&\quad&
B_{t^*}\in\mathbb{Z}_{\ge 0}\mathbb{P} [  {x_{d';t^{*}}}, {\widetilde{x_{d;t^{*}}}};(x_{f;t^*}^{\pm 1})_{ f\ne d,d'}],
\\
\\
C_{t^*}\in\mathbb{Z}_{\ge 0}\mathbb{P} [ {x_{d';t^*}},{{x_{d;t^*}}};(x_{f;t^*}^{\pm 1})_{ f\ne d,d'}],
&\quad&
D_{t^*}\in\mathbb{Z}_{\ge 0}\mathbb{P} [ \widetilde{x_{d';t^*}},  {x_{d;t^*}};(x_{f;t^*}^{\pm 1})_{ f\ne d,d'}].
\end{array}
   \]
   
{\ralf   If the sequence $\mu$ is as in the second case, the roles of $d$ and $d'$ are interchanged. Without	loss of generality, we assume we are in the first case.}
   
 Consider the variables appearing in these expressions one by one. If $f\ne d,d'$ then  $x_{f;t^*}=x_{f;t}$ is in $\mathbf{x}_t$ and  may have a negative exponent in the desired expression for $u$. 
The variables
 ${\widetilde{\widetilde{x_{d';t^*}}}}, \widetilde{x_{d;t^*}},{x_{d';t^{*}}},{{x_{d;t^*}}}$ and $\widetilde{x_{d';t^*}}$ lie on a rank 2 mutation sequence from $t$ in the directions $d$ and $d'$, and  Corollary \ref{cor 2.11}  implies that all the cluster monomials involving such variables (up to the $x_f^\pm$) have expansions of the form $f_1+f_2$ with 
  \[f_1\in \mathbb{Z}_{\ge 0}\mathbb{P}[ \widetilde{x_{d;t}},x_{d';t}^{\pm1}; x_{f;t} :f\ne d,d'] \quad\textup{and} \quad  f_2 \in \mathbb{Z}_{\ge 0}\mathbb{P}[ {x_{d;t}},x_{d';t}^{\pm1}; x_{f;t} :f\ne d,d'] .\]
 Substituting these expansions into (\ref{eq41}) shows that 
 \begin{equation}\label{eq B'C'} u=B_t'+C_t',\end{equation}
  with  \[
\begin{array}{rclcrcl}
B_t'\in\mathbb{Z}_{\ge 0}\mathbb{P} [ {{\widetilde{x_{d;t}}}}, x_{d';t}^{\pm 1};(x_{f;t^*}^{\pm 1})_{ f\ne d,d'}],
&\quad&
C_t'\in\mathbb{Z}_{\ge 0}\mathbb{P} [  {x_{d;t}}, x_{d';t}^{\pm 1};(x_{f;t^*}^{\pm 1})_{ f\ne d,d'}],
\end{array}
   \]
  We now have to study the exponents of $x_{e;t}$ in $B_t'$ and $C_t'$. If these exponents are non-negative then $u$ is of the form $B_t+C_t$ and we are done. But if $x_{e;t}$ appears with negative exponents in $B_t'$, then we have to rewrite $B_t'$ in the form $A_t+B_t$, and if $x_{e;t}$ appears with negative exponents in $C_t'$, then we have to rewrite $C_t'$ in the form $C_t+D_t$, with $A_t, B_t, C_t, D_t$ as in the statement of Theorem~\ref{thm1}. We prove that this is always possible in Proposition~\ref{cor2.17}. To do so, we have to go back in the mutation sequence $\mu$ up to the last mutations in direction $e$.
  
   More precisely, 
consider the last maximal rank 2 subsequence $\nu$ of $\mu$ containing $e$. Let $e'$ be the other direction occurring in $\nu$. If
\[\xymatrix{\nu =&t_{w_B}\ar@{-}[r]^{e'}&t_{w_C}\ar@{-}[r]^{e}&t_{w_D}\ar@{-}[r]^{e'}&\ar@{-}[r]^e&\ar@{.}[r]&\ar@{-}[r]^{e'}&t_{w_Z}}
\] 
then let $t_{w_A}$ be the seed obtained from $t_{w_B}$ by mutating in direction $e$.
If
\[\xymatrix{\nu =&t_{w_B}\ar@{-}[r]^{e}&t_{w_C}\ar@{-}[r]^{e'}&t_{w_D}\ar@{-}[r]^e&\ar@{-}[r]^{e'}&\ar@{.}[r]&\ar@{-}[r]^{e'}&t_{w_Z}}\]  
then let $t_{w_A}$ be the seed obtained from $t_{w_B}$ by mutating in direction ${e'}$. We do not label any of the seeds between $t_{w_D}$ and $t_{w_E}$.

Without loss of generality, we may assume that $t_{w_Z}=t^{*}$ and that ${e'}=d$. Indeed, otherwise we can change the sequence $\mu$ by inserting two consecutive mutations in direction $e$ as follows
 \[\xymatrix{\mu' = \ t_0\ar@{.}[r]&\cdot\ar@{-}[r]^{d''}&t^{**}\ar@{-}[r]^{e}&t_{w_B}\ar@{-}[r]^{e}&t^{**}\ar@{-}[r]^{d}&t^*\ar@{-}[r]^{d'}&\cdot\ar@{.}[r]&\cdot\ar@{-}[r]^{d'}&\cdot\ar@{-}[r]^d&t}.
\] and letting $\nu$ be the mutation sequence $\xymatrix{t_{w_B}\ar@{-}[r]^{e}&t^{**}\ar@{-}[r]^{d}&t^*}$.

Thus, without loss of generality, we assume the sequence $\mu$ is as follows. 
 \begin{equation}\label{mu}
 \xymatrix{\mu = \ t_0\ar@{.}[r]&t_{w_B}\ar@{-} [r]&t_{w_C}\ar@{.}[r]&\cdot\ar@{-}[r]^d&\cdot\ar@{-}[r]^{e}&t^{**}\ar@{-}[r]^{d}&t^*\ar@{-}[r]^{d'}&\cdot\ar@{.}[r]&\cdot\ar@{-}[r]^{d'}&\cdot\ar@{-}[r]^d&t}.
\end{equation}

\begin{proposition}\label{prop2.13} The variable
$x_{e;t}$ may have negative exponents in one of the expressions $B'_t$ or $C'_t$ in equation~(\ref{eq B'C'}), but not in both. 
\end{proposition}
\begin{proof}This is proved in section~\ref{proof prop2.13}.
\end{proof}

Let $A_{i;1}=A_i$ (respectively $\tau_{i;1}=\tau_i, s_{i;1}=s_1$ ) be the sequence of integers defined in section~\ref{sect 3.3} with respect to $r$ the number of arrows between $e$ and $d$ at the seed $t_{w_E}$, where $E=A,B,C$, or $D$. \begin{proposition}\label{prop2.14}
Suppose that at least one monomial of $B'_t$ has a negative exponent  in $x_{e;t}$. 
Then each cluster monomial of the form $x_{e;t_{w_E}}^p x_{d;t_{w_E}}^q$, where $E=A,B,C$, or $D$, has the following form 
\begin{equation}\label{eq2.14}
\aligned
&\sum_{v\geq 0} (\text{Laurent polynomial in cluster variables {\ralf of $\mathbf{x}_{t}\cup\mathbf{x}_{t'} \setminus\{x_{e;t}\}$}})\,x_{e;t}^{v}\\
+&\sum_{\theta>0}x_{e;t'}^{-\theta}\sum_{\varsigma\geq 0} \sum_ \mathfrak{r} \lambda_\mathfrak{r} \mathfrak{r}\\
&\times \sum_{\begin{array}{c} \scriptstyle \tau_{0;\nuisone},\tau_{1;\nuisone},\cdots,\tau_{n_1-2;\nuisone}\\\scriptstyle  s_{n_\nuisone-1;\nuisone}=A_{n_\nuisone-1;\nuisone}-\varsigma\end{array}} 
\sum_{j=0}^{\sum_{w=1}^{n_2-3} \tau_{w;2}} d_j {{ \left\lfloor (A_{n_\nuisone-1;\nuisone}-\varsigma) \frac{A_{n_\nuisone-1;\nuisone}}{A_{n_\nuisone;\nuisone}}\right\rfloor -s_{n_\nuisone-2;\nuisone}}
 \choose j}  \\
 &\times \left( \prod_{w=0}^{n_\nuisone-2}\gchoose{A_{w+1;\nuisone} - r_\nuisone s_{w;\nuisone}}{\tau_{w;\nuisone}}   \right)\left(\frac{\prod_i x_{i;t'}^{ [b_{i,e}^{t'}]_+ }  }{\prod_i x_{i;t'}^{ [-b_{i,e}^{t'}]_+ } }\right)^{\emph{sgn}(2b_{d,e}^{t'}+1) \left(\left\lfloor (A_{n_\nuisone-1;\nuisone}-\varsigma) \frac{A_{n_\nuisone-1;\nuisone}}{A_{n_\nuisone;\nuisone}}\right\rfloor -s_{n_\nuisone-2;\nuisone}\right)},
\endaligned \end{equation}
where $\mathfrak{r}\in \mathbb{ZP}_{\geq 0}[\{x_{d;t'}\}\cup (\mathbf{x}^{\pm1}_{t'}\setminus\{x_{d;t'}^{\pm1},x_{e;t'}^{\pm1}\})]$, $\lambda_{ \mathfrak{r}}\in\{0,1\}$  and   $d_j$ are nonnegative integers depending  on the summation indices,
{\ralf and the integers $\tau_{i,\ell}$ satisfy condition (\ref{cond502b0303}) with $n=n_1+1$ if $\ell=1$; and $n=n_2$ if $\ell=2$, for some integers $n_1,n_2$. The second subindex $\ell=1,2$ in $s_{i,;\ell}, A_{i;\ell}, \tau_{i;\ell}$ refers to the fact that these integers are defined in terms of $n_\ell$.}
 \end{proposition}
\begin{proof} This is proved in section \ref{proof prop2.14}.\end{proof}
\begin{remark}
The condition that at least one monomial of $B_t'$ has a negative exponent in $x_{e;t}$ does not depend on the variable $u$ but rather on the orientation of the quivers in the mutation sequence $\nu$.
\end{remark}

\begin{proposition}\label{prop2.16}
Suppose that at least one monomial of $C'_t$ has a negative exponent  in $x_{e;t}$. 
Then each cluster monomial of the form $x_{e;t_{w_E}}^p x_{f;t_{w_E}}^q$, where $E=A,B,C$, or $D$, has the following form 
\begin{equation}\label{20121225double star}
\aligned
&\sum_{v\geq 0} (\text{Laurent polynomial in cluster variables {\ralf of $\mathbf{x}_{t}\cup\mathbf{x}_{t'} \setminus\{x_{e;t}\}$}})\,x_{e;t}^{v}\\
+&\sum_{\theta>0}x_{e;t'}^{-\theta}\sum_{\varsigma\geq 0}  \sum_ \mathfrak{r} \lambda_\mathfrak{r} \mathfrak{r}\\
&\times \sum_{\begin{array}{c} \scriptstyle \tau_{0;\nuisone},\tau_{1;\nuisone},\cdots,\tau_{n_1-2;\nuisone}\\\scriptstyle  s_{n_\nuisone-1;\nuisone}=A_{n_\nuisone-1;\nuisone}-\varsigma\end{array}} \sum_{j=0}^{A_{n_\nuisone-2;\nuisone} - r_\nuisone \varsigma -\theta} d_j {{ \left\lfloor (A_{n_\nuisone-1;\nuisone}-\varsigma) \frac{A_{n_\nuisone-1;\nuisone}}{A_{n_\nuisone;\nuisone}}\right\rfloor -s_{n_\nuisone-2;\nuisone}}
 \choose j}  \\
 &\times \left( \prod_{w=0}^{n_\nuisone-2}\gchoose{A_{w+1;\nuisone} - r_\nuisone s_{w;\nuisone}}{\tau_{w;\nuisone}}   \right)\left(\frac{\prod_i x_{i;t}^{ [b_{i,e}^{t}]_+ }  }{\prod_i x_{i;t}^{ [-b_{i,e}^{t}]_+ } }\right)^{\emph{sgn}(2b_{d,e}^{t}+1) \left(\left\lfloor (A_{n_\nuisone-1;\nuisone}-\varsigma) \frac{A_{n_\nuisone-1;\nuisone}}{A_{n_\nuisone;\nuisone}}\right\rfloor -s_{n_\nuisone-2;\nuisone}\right)},
\endaligned \end{equation}
where $\mathfrak{r}\in \mathbb{ZP}_{\geq 0}[\{x_{d;t}\}\cup (\mathbf{x}^{\pm1}_{t}\setminus\{x_{d;t}^{\pm1},x_{f;t}^{\pm1}\})]$, $\lambda_{ \mathfrak{r}}\in\{0,1\}$ and   $d_j$ are nonnegative integers depending on the summation indices, {\ralf and the integers $\tau_{i,\ell}$ satisfy condition (\ref{cond502b0303}) with $n=n_1+1$ if $\ell=1$, for some integer $n_1$.}  
\end{proposition}
\begin{proof}
 The proof of Proposition \ref{prop2.16} is similar to the proof of Proposition \ref{prop2.14}.
\end{proof}

\begin{proposition}\label{cor2.17} With the notation in equation (\ref{eq B'C'}) we have
\begin{enumerate}
\item
$B'_t$ is of the form $A_t+B_t$;
\item
$C'_t$ is of the form $C_t +D_t$; 
\end{enumerate}
with $A_t, B_t, C_t, D_t$ as in the statement of Theorem~\ref{thm1}.
\end{proposition}

\begin{proof} This is proved in section~\ref{sect 5.3}.
\end{proof}

This completes the proof of Theorem \ref{thm1} modulo Propositions \ref{prop2.13}, \ref{prop2.14} and \ref{cor2.17}.

\subsection{Proof of Proposition \ref{prop2.13}}\label{proof prop2.13}  
 By induction, we can assume that $u$ can be written as a Laurent polynomial in the clusters $t_{w_A},t_{w_B},t_{w_C},$ and $t_{w_D}$ in such a way that the variables $x_{d;-}$ and $x_{e;-}$ appear only with non-negative exponents.
Thus, in order to prove Proposition~\ref{prop2.13}, we must compute the $\mathbf{x}_t$-expansions of cluster monomials  $x_{d;t_{w_E}}^p x_{e;t_{w_E}}^q$ with $p,q\ge 0.$ 

Recall from (\ref{mu})  that  $t^*=t_{w_Z}$ and our mutation sequence is of the following form. 

 \begin{equation}\label{muagain}
 \xymatrix{t_{w_E}\ar@{.}[r]&\cdot\ar@{-}[r]^d&\cdot\ar@{-}[r]^{e}&t^{**}\ar@{-}[r]^{d}&t^*\ar@{-}[r]^{d'}&\cdot\ar@{.}[r]&\cdot\ar@{-}[r]^{d'}&t'\ar@{-}[r]^d&t}.
\end{equation}

Let $Q_0, Q_1$ and $Q_2$, respectively, be the full subquivers on vertices $d,d',e$ of the quiver at the seeds $t_{w_E}, t^{**}$ and $t'$, respectively. Note that the notation here is not the same as in Lemma \ref{lem nonacyclic}. Denote by $n_1$ the number of seeds between $t_{w_E}$ and $t^*$ inclusively, and  by $n_2$ the number of seeds between $t^{**}$ and $t$ inclusively. Note that $n_2$ is even.
We define $s_{i,;\ell}, A_{i;\ell}$ and $\tau_{i;\ell}$ with $\ell=1,2$ 
 as in section \ref{sect 3.3}, but replacing $n$ with $n_\ell-2$.

We shall often use Lemma \ref{lem nonacyclic} to compute the relations between the number of arrows in the quivers $Q_0,Q_1 $ and $Q_2$. The integer $n$ in the statement of  Lemma \ref{lem nonacyclic}  denotes the number of mutations between two quivers; thus we have
\[n=
\left\{\begin{array}
 {ll}
 n_1-2 &\textup{between $Q_0$ and $Q_1$}\\
 n_2-2 &\textup{between $Q_1$ and $Q_2$}
\end{array}\right.\]
Let $$Q_1=
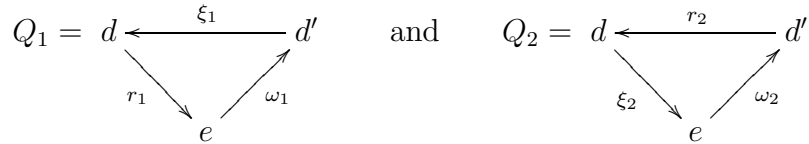
\xymatrix{d\ar[dr]_{r_1}&&d'\ar[ll]_{\xi_1}\\&e\ar[ru]_{\omega_1}} \qquad \textup{and} \qquad Q_2=\xymatrix{d\ar[dr]_{\xi_2}&&d'\ar[ll]_{r_2}\\&e\ar[ru]_{\omega_2}}$$ where $r_1$ (respectively $\omega_1$ and $\xi_1$) is the number of arrows from $d$ to $e$ (respectively from $e$ to $d'$, and from  $d'$ to $d$) in $Q_1$, and  $r_2$ (respectively $\omega_2$ and $\xi_2$) is the number of arrows from $d'$ to $d$ (respectively from $e$ to $d'$, and from  $d$ to $e$) in $Q_2$. Without loss of generality, we assume that $r_1\ge0$ and $r_2\ge 0.$ Thus we also have $\xi_1\ge 0$, since $r_2=\xi_1$. Note however, that $\xi_2,\omega_1$ and $\omega_2$ may be negative.
For the rest of this proof, we set $x_{f;t_{w_E}}=1$ for all $f\ne d,d',e$. 

Recall that the  sequences of quivers of {\em almost cyclic type} and of {\em acyclic type} were defined in Definition~\ref{def 12}.
 We need to consider four cases. 
\begin{itemize}
\item [(a)] Both the sequence of quivers from $t_{w_E} $ to $t_{w_Z}=t^*$ and the sequence of quivers from $\mu_{d}(t_{w_Z})=t^{**}$ to $t$ are of almost cyclic type.
\item [(b)]  The sequence of quivers from $t_{w_E}$ to $t_{w_Z}=t^*$ is of almost cyclic type, and the sequence of quivers from $\mu_{d}(t_{w_Z})=t^{**}$ to $t $ is of acyclic type. 

\item [(c)] The sequence of quivers from $t_{w_E}$ to $ t_{w_Z}=t^*$ is of acyclic type, and the sequence of quivers from $\mu_{d}(t_{w_Z})=t^{**}$ to $t$ is of almost cyclic type.
\item [(d)] Both the sequence of quivers from $t_{w_E} $ to $t_{w_Z}=t^*$ and the sequence of quivers from $\mu_{d}(t_{w_Z})=t^{**}$ to $t$ are of acyclic type.
\end{itemize}

Let us suppose first that we are in the case (a) or (b), that is,  the sequence of quivers from $t_{w_B}$ to $t_{w_Z}=t^*$ is of almost cyclic type. Using Theorem \ref{mainthm12062011}, we see that $x_{d;t_{\omega_E}}^p x_{e;t_{\omega_E}}^q$ is equal to 
\begin{equation}\label{0303eq5c}
 \sum_{\begin{array}{c}\scriptstyle\tau_{0;1},\cdots,\tau_{n_1-3;1}\\ \scriptstyle  A_{n_1-1;1} - r_1s_{n_1-2;1}\geq 0\end{array}}  \!\!\!\!\!\!\!\!\!\!\!\!\!\! \left( \prod_{w=0}^{n_1-3}\!\!
 \left(\begin{array}{c}A_{w+1;1} - r_1s_{w;1}\!\!\!
 \\{\tau_{w;1}} \!\! \end{array} \right)\!\!\!\right){\widetilde{x_{d;t_{\omega_Z}}}}^ {A_{n_1-1;1}-r_1s_{n_1-2;1}}x_{e;t_{\omega_Z}}^ {r_1s_{n_1-3;1}-A_{n_1-2;1}   }  x_{d';t_{\omega_Z}}^{\omega_1 s_{n_1-2;1} - \xi_1  s_{n_1-3;1}}
 \end{equation}
\begin{equation}\label{0303eq5d}
+ \,\, x_{d;t_{\omega_Z}}^{-A_{n_1-1;1}} x_{e;t_{\omega_Z}}^{-A_{n_1-2;1}}
\sum_{\begin{array}{c}\scriptstyle(S_1,S_2)\\ \scriptstyle-A_{n_1-1;1}+r_1|S_2|> 0\end{array}}x_{d;t_{\omega_Z}}^{r_1|S_2|}x_{e;t_{\omega_Z}}^{r_1|S_1|} x_{d';t_{\omega_Z}}^{\xi_1(A_{n_1-1;1}-|S_1|)-(\xi_1r_1-\omega_1)|S_2|-M_{d'}}.
\end{equation}

From now on, we restrict ourselves to the most difficult case where
the full rank 3 subquivers  with vertices $d,d',e$ at  both seeds $t_{w_E}$ and $t$ are non-acyclic, in other words,
     $\omega_1 c^{[r_1]}_{{n_1}} -\xi_1 c^{[r_1]}_{{n_1}-1}> 0$ 
and   
    $c_{n_2}^{[\xi_1]}r_1- c_{n_2-1}^{[\xi_1]}\omega_1 >0$. 
    In particular, $\xi_2\ge 0$.
But the same argument can be adapted to case where the full rank~3 subquiver at either $t_{w_E}$ or $t$ is acyclic. We focus on the former case, because the positivity conjecture for acyclic cluster algebras was already proved in \cite{KQ}.
Thus assume that $(\omega_1c_{n_1}^{[r_1]}-\xi_1c_{n_1-1}^{[r_1]})>0$. In this case $M_{d'}=0$. 


Let $p_2= {A_{n_1-1;1}-r_1s_{n_1-2;1}}$
and $q_2= {\omega_1 s_{n_1-2;1} - \xi_1  s_{n_1-3;1}   }$ be the exponents of
$\widetilde{x_{d;t_{\omega_Z}}} $ 
and 
$x_{d';t_{\omega_Z}} $ 
in (\ref{0303eq5c}), respectively. Define $A_{i;2}=p_2c_{i+1}^{[r_2]}+q_2c_{i}^{[r_2]}$.

In case (a), applying Theorem~\ref{mainthm12062011} to $\widetilde{x_{d;t_{\omega_Z}}}^ {p_2} x_{d';t_{\omega_Z}}^{q_2}
$ in (\ref{0303eq5c}), we have that the part of (\ref{0303eq5c}) that contributes to $C'_t$ is equal to
\begin{eqnarray}\label{new18}
&& \displaystyle\sum_{\begin{array}{c}\scriptstyle\tau_{0;1},\cdots,\tau_{n_1-3;1}\\ \scriptstyle  A_{n_1-1;1} - r_1s_{n_1-2;1}\geq 0\end{array}}   \left( \prod_{w=0}^{n_1-3}\gchoose{A_{w+1;1} - r_1s_{w;1}}{\tau_{w;1}}   \right) x_{e;t_{\omega_Z}}^ {r_1s_{n_1-3;1}-A_{n_1-2;1}   }  \\
&&\displaystyle\times \sum_{\begin{array}{c} \scriptstyle  (S_1,S_2)\\ 
 \scriptstyle r_2|S_2|-A_{n_2-1;2} >0 \end{array}  } x_{d;t}^{r_2|S_2|-A_{n_2-1;2}}x_{d';t}^{r_2|S_1|-A_{n_2-2;2}} x_{e;t}^{\xi_2(A_{n_2-1;2}-|S_1|)-(\xi_2 r_2 - \omega_2)|S_2|}\nonumber\\
&= &\displaystyle\sum_{\begin{array}{c}\scriptstyle\tau_{0;1},\cdots,\tau_{n_1-3;1}\\ \scriptstyle  A_{n_1-1;1} - r_1s_{n_1-2;1}\geq 0\end{array}}   \left( \prod_{w=0}^{n_1-3}\gchoose{A_{w+1;1} - r_1s_{w;1}}{\tau_{w;1}}   \right)\nonumber \\
&&\displaystyle\times \sum_{\begin{array}{c} \scriptstyle  (S'_1,S'_2)\\ 
 \scriptstyle r_2|S'_2|-A_{n_2-1;2} >0 \end{array}  } \!\!\!\!\!x_{d;t}^{r_2|S'_2|-A_{n_2-1;2}}x_{d';t}^{r_2|S'_1|-A_{n_2-2;2}} x_{e;t}^{\xi_2(A_{n_2-1;2}-|S'_1|)-(\xi_2 r_2 - \omega_2)|S'_2|+{r_1s_{n_1-3;1}-A_{n_1-2;1}   }},\nonumber
\end{eqnarray}
where $(S'_1,S'_2)$ is a family of compatible pairs satisfying the condition in Theorem~\ref{mainthm12062011}.

Since $ r_2|S'_2|-A_{n_2-1;2} >0$ in this last expression, then 
$A_{n_2-1;2} / r_2|S'_2| <1$ and thus
\begin{equation}\label{20140625eq5} \frac{r_1A_{n_2-1;2}}{c_{n_2-1}^{[r_2]} \,r_2} =
 \frac{r_1|S'_2|}{c_{n_2-1}^{[r_2]} } \frac{\,A_{n_2-1;2}}{r_2|S'_2|}
<
\frac{r_1|S'_2|}{c_{n_2-1}^{[r_2]}} 
\le \xi_2 (A_{n_2-1;2}-|S'_1|)-(\xi_2 r_2 - \omega_2)|S'_2|,\end{equation}
where the last inequality follows from Lemma~\ref{0413lem} below.
Using $r_2=\xi_1$ and the definition of $A_{n_2-1;2}$, we get
$$\aligned
& r_1s_{n_1-3;1}-A_{n_1-2;1} + \xi_2(A_{n_2-1;2}-|S'_1|)-(\xi_2 r_2 - \omega_2)|S'_2| \\
&> r_1s_{n_1-3;1}-A_{n_1-2;1} + \frac{r_1\left( c_{n_2}^{[\xi_1]}(A_{n_1-1;1}-r_1s_{n_1-2;1} ) + c_{n_2-1}^{[\xi_1]}(\omega_1 s_{n_1-2;1} - \xi_1 s_{n_1-3;1})\right)}{c_{n_2-1}^{[\xi_1]} \xi_1}\\
&= -A_{n_1-2;1} + \frac{r_1\left( c_{n_2}^{[\xi_1]}(A_{n_1-1;1}-r_1s_{n_1-2;1} ) + c_{n_2-1}^{[\xi_1]}\omega_1 s_{n_1-2;1}\right)}{c_{n_2-1}^{[\xi_1]} \xi_1}\\
&= -A_{n_1-2;1} + \frac{r_1\left( c_{n_2}^{[\xi_1]}A_{n_1-1;1}-(c_{n_2}^{[\xi_1]}r_1- c_{n_2-1}^{[\xi_1]}\omega_1) s_{n_1-2;1}\right)}{c_{n_2-1}^{[\xi_1]} \xi_1}\\
&\underset{(s_{n_1-2;1}\leq A_{n_1-1;1}/r_1)}{\geq}  -A_{n_1-2;1} + \frac{r_1\left( c_{n_2}^{[\xi_1]}A_{n_1-1;1}-(c_{n_2}^{[\xi_1]}r_1- c_{n_2-1}^{[\xi_1]}\omega_1) \frac{A_{n_1-1;1}}{r_1}\right)}{c_{n_2-1}^{[\xi_1]} \xi_1}\endaligned$$
\begin{equation}\label{eq 44.1}
 \aligned 
&= \frac{1}{\xi_1}(\omega_1A_{n_1-1;1} -\xi_1A_{n_1-2;1})\\
&=\frac{1}{\xi_1}( \omega_1(pc^{[r_1]}_{n_1}+qc^{[r_1]}_{{n_1}-1}) -\xi_1(pc^{[r_1]}_{{n_1}-1} + qc^{[r_1]}_{{n_1}-2}))\\
&=\frac{1}{\xi_1}( p(\omega_1 c^{[r_1]}_{n_1} -\xi_1c^{[r_1]}_{{n_1}-1}) + q (\omega_1 c^{[r_1]}_{{n_1}-1} -\xi_1 c^{[r_1]}_{{n_1}-2}))\\
&>0,
\endaligned\end{equation}
where the last inequality holds, because from the first case of Lemma \ref{lem nonacyclic} we see that  
$$\aligned
&\{\omega_1 c^{[r_1]}_{n_1} -\xi_1c^{[r_1]}_{n_1-1}, \omega_1 c^{[r_1]}_{n_1-1} -\xi_1 c^{[r_1]}_{n_1-2}\}\\&=\{(\text{the number of arrows between } d'\text{ and } d\text{ in the seed }t_{w_E} ), \\
&\,\,\,\,\,\,\,\,\,\,\,(\text{the number of arrows between }d' \text{ and } e\text{ in the seed }t_{w_E})\}. \endaligned$$

 Thus in case (a), the exponent of $x_{e;t}$ in the expansion of (\ref{new18}) is positive. Thus the terms with negative exponents on $x_{e;t}$ in (\ref{0303eq5c}) are all in $B'_t$. The proof that  the terms with negative exponents on $x_{e;t}$ in  (\ref{0303eq5d}) are also all in $B'_t$ uses a similar argument. For the sake of completeness, we include the details. 
 
Let $p_3=r_1|S_{2}|-A_{n_1-1;1}$ 
and $q_3=\xi_1(A_{n_1-1;1} - |S_{1}|) - (\xi_1r_1-\omega_1)|S_{2}|$,  and let $A_{i;3}=p_3 c_{i+1}^{[r_2]}+q_3 c_{i}^{[r_2]} $.
Applying Theorem~\ref{mainthm12062011} to $x_{d;t_{\omega_Z}}^{p_3}
x_{d';t_{\omega_Z}}^{q_3} $ in (\ref{0303eq5d}), the exponents of 
$x_{e;t}$ in the part 
of (\ref{0303eq5d}) that contributes to  $C_t'$ are of the form
\begin{equation}\label{20140624eq1}
\xi_2(A_{n_2-2;3}-|S'_{1}|) - (\xi_2r_2-\omega_2)|S'_{2}|
+r_1|S_{1}|-A_{n_1-2;1}.
\end{equation}We will derive the same conclusion as above, namely $\eqref{20140624eq1}> \frac{1}{\xi_1}(\omega_1A_{n_1-1;1} -\xi_1A_{n_1-2;1})$.

Using $r_2=\xi_1$, the definition of $A_{n_2-2;3}$ and the inequality \eqref{20140625eq5} with $n_2-1$ replaced by $n_2-2$,
we get 
{\tiny $$\aligned
&\eqref{20140624eq1}
> r_1|S_{1}|-A_{n_1-2;1}+ \frac{r_1}{\xi_1c_{n_2-2}^{[\xi_1]}}
\left(c_{n_2-1}^{[\xi_1]}(r_1|S_{2}|-A_{n_1-1;1}) + c_{n_2-2}^{[\xi_1]}(\xi_1(A_{n_1-1;1} - |S_{1}|) - (\xi_1r_1-\omega_1)|S_{2}|) \right)\\
&= -A_{n_1-2;1}+ \frac{r_1}{\xi_1c_{n_2-2}^{[\xi_1]}}
\left(c_{n_2-1}^{[\xi_1]}(r_1|S_{2}|-A_{n_1-1;1}) + c_{n_2-2}^{[\xi_1]}(\xi_1A_{n_1-1;1} - (\xi_1r_1-\omega_1)|S_{2}|) \right)\\
&= -A_{n_1-2;1}+ \frac{r_1}{\xi_1c_{n_2-2}^{[\xi_1]}}
\left(c_{n_2-1}^{[\xi_1]}r_1|S_{2}|+ c_{n_2-3}^{[\xi_1]}A_{n_1-1;1} - c_{n_2-2}^{[\xi_1]}(\xi_1r_1-\omega_1)|S_{2}|) \right)\\
&=\frac{1}{\xi_1c_{n_2-2}^{[\xi_1]}}\left(r_1c_{n_2-3}^{[\xi_1]}A_{n_1-1;1} - \xi_1c_{n_2-2}^{[\xi_1]}A_{n_1-2;1}
+r_1|S_{2}|(r_1c_{n_2-1}^{[\xi_1]} -(\xi_1r_1-\omega_1)c_{n_2-2}^{[\xi_1]} ) \right)\\
&\underset{(r_1|S_{2}|-A_{n_1-1;1}>0)}{>} \frac{1}{\xi_1c_{n_2-2}^{[\xi_1]}}\left(r_1c_{n_2-3}^{[\xi_1]}A_{n_1-1;1} - \xi_1c_{n_2-2}^{[\xi_1]}A_{n_1-2;1}
+A_{n_1-1;1}(r_1c_{n_2-1}^{[\xi_1]} -(\xi_1r_1-\omega_1)c_{n_2-2}^{[\xi_1]})  \right)\\
&= \frac{1}{\xi_1c_{n_2-2}^{[\xi_1]}}\left(r_1c_{n_2-3}^{[\xi_1]}A_{n_1-1;1} - \xi_1c_{n_2-2}^{[\xi_1]}A_{n_1-2;1}
+A_{n_1-1;1}(\omega_1c_{n_2-2}^{[\xi_1]} -r_1c_{n_2-3}^{[\xi_1]})  \right)\\
&= \frac{1}{\xi_1c_{n_2-2}^{[\xi_1]}}\left(A_{n_1-1;1}\omega_1c_{n_2-2}^{[\xi_1]}  - \xi_1c_{n_2-2}^{[\xi_1]}A_{n_1-2;1}\right)\\
&=\frac{1}{\xi_1}(\omega_1A_{n_1-1;1} -\xi_1A_{n_1-2;1}).
\endaligned
$$}

 In case (b), applying Theorem~\ref{mainthm12062011}  to $\widetilde{x_{d;t_{\omega_Z}}}^ {p_2} x_{d';t_{\omega_Z}}^{q_2}
$ in (\ref{0303eq5c}), we see that the part of (\ref{0303eq5c}) that contributes to $B'_t$ is equal to
\begin{equation}\label{07252012eq1}
\aligned
&\sum_{\tau_{0;1},\tau_{1;1},\cdots,\tau_{n_1-3;1}\atop  A_{n_1-1;1} - r_1s_{n_1-2;1}\geq 0}   \left( \prod_{w=0}^{n_{1}-3}\gchoose{A_{w+1;1} - r_1s_{w;1}}{\tau_{w;1}}   \right) x_{e;t}^ {r_1s_{n_1-3;1}-A_{n_1-2;1} } 
 \\
&\times \sum_{\begin{array}{c}\scriptstyle\tau_{0;2},\tau_{1;2},\cdots,\tau_{n_{2}-3;2}\\\scriptstyle { A_{n_2-1;2}}-{r_2} s_{n_2-2;2}> 0 \end{array}}   \left( \prod_{w=0}^{n_{2}-3}\gchoose{A_{w+1;2} - r_{2}s_{w;2}}{\tau_{w;2}}   \right) \\
&\times  \widetilde{x_{d;t}}^ {A_{n_2-1;2}-r_{2} s_{n_{2}-2;2} }
 x_{d';t}^{ r_2s_{n_2-3;2}-A_{n_2-2;2}} x_{e;t}^{\omega_2s_{n_2-2;2}-\xi_2s_{n_2-3;2}-M_e}\\
  \\
 =
&\sum_{\tau_{0;1},\tau_{1;1},\cdots,\tau_{n_1-3;1}\atop  A_{n_1-1;1} - r_1s_{n_1-2;1}\geq 0}   \left( \prod_{w=0}^{n_{1}-3}\gchoose{A_{w+1;1} - r_1s_{w;1}}{\tau_{w;1}}   \right)  \\
&\times \sum_{\begin{array}{c}\scriptstyle\tau_{0;2},\tau_{1;2},\cdots,\tau_{n_{2}-3;2}\\\scriptstyle { A_{n_2-1;2}}-{r_2} s_{n_2-2;2}> 0 \end{array}}   \left( \prod_{w=0}^{n_{2}-3}\gchoose{A_{w+1;2} - r_{2}s_{w;2}}{\tau_{w;2}}   \right) \\
&\times  \widetilde{x_{d;t}}^ {A_{n_2-1;2}-r_{2} s_{n_{2}-2;2} }
 x_{d';t}^{ r_2s_{n_2-3;2}-A_{n_2-2;2}}
 x_{e;t}^{\omega_2s_{n_2-2;2}-\xi_2s_{n_2-3;2}+r_1s_{n_1-3;1}-A_{n_1-2;1}-M_e}  \\
 \endaligned\end{equation} 
where $M_e=\omega_2{A_{n_2-2;2}} - \xi_2A_{n_2-3;2}$,
and $s_{i;2}$ is as defined before Lemma \ref{lem sn} but in terms of $p_2, q_2,$ and $r_2$, thus  $s_{i;2}=\sum_{j=0}^{i-1}c_{i-j+1}^{[r_2] }\tau_{j;2}$.

In this last expression, we have  \begin{equation}\label{07272012_eq02}
\frac{ A_{n_2-1;2}}{r_2} > s_{n_2-2;2}.
\end{equation}
We want to show that the exponent of $x_{e;t}$ is positive, that is,
\begin{equation}\label{07272012_eq03}\omega_2s_{n_2-2;2}-\xi_2s_{n_2-3;2}+r_1s_{n_1-3;1}-A_{n_1-2;1}-(\omega_2{A_{n_2-2;2}} - \xi_2A_{n_2-3;2})>0. \end{equation}

Thanks to Lemma~\ref{lem 514} below, we have $$s_{n_2-2;2} A_{n_2-2;2} >  s_{n_2-3;2} A_{n_2-1;2},
$$ 
and thus
it suffices to show that 
\begin{equation}\label{eq 7.27} \left(\omega_2  -\xi_2\frac{A_{n_2-2;2}}{A_{n_2-1;2}}\right)s_{n_2-2;2}+r_1s_{n_1-3;1}-A_{n_1-2;1}-(\omega_2{A_{n_2-2;2}} - \xi_2A_{n_2-3;2})>0.\end{equation}
%
%
$$\Longleftrightarrow \quad {r_1s_{n_1-3;1}-A_{n_1-2;1}-(\omega_2{A_{n_2-2;2}} - \xi_2A_{n_2-3;2})} > \left({ \xi_2\frac{A_{n_2-2;2}}{A_{n_2-1;2}} - \omega_2 }\right)s_{n_2-2;2}$$
and by (\ref{07272012_eq02}) it suffices to show
$$
 r_1s_{n_1-3;1}-A_{n_1-2;1}-(\omega_2{A_{n_2-2;2}} - \xi_2A_{n_2-3;2}) >  \left( \xi_2\frac{A_{n_2-2;2}}{A_{n_2-1;2}} - \omega_2 \right)\frac{A_{n_2-1;2}}{r_2}$$


$${\Longleftrightarrow} \quad \left(r_1s_{n_1-3;1}-A_{n_1-2;1}+ (\xi_2A_{n_2-3;2}- \omega_2{A_{n_2-2;2}})\right)r_2 >  \xi_2A_{n_2-2;2}- \omega_2{A_{n_2-1;2}} $$

\begin{equation}\label{07272012_eq01}
\stackrel{\textup{Lemma}~\ref{negone}}{\Longleftrightarrow}  \quad \xi_2A_{n_2-4;2}- \omega_2{A_{n_2-3;2}}> r_2\left(A_{n_1-2;1}-r_1s_{n_1-3;1}\right). \end{equation}

Since
$$
A_{i;2}=p_2c_{i+1}^{[r_2]} + q_2 c_{i}^{[r_2]} = (A_{n_1-1;1}-r_1s_{n_1-2;1})c_{i+1}^{[r_2]} + (\omega_1 s_{n_1-2;1}-\xi_1s_{n_1-3;1})c_{i}^{[r_2]}, 
$$
the inequality (\ref{07272012_eq01}) is equivalent to 
 $$\begin{array}{c}(A_{n_1-1;1}-r_1s_{n_1-2;1})(\xi_2c_{n_2-3}^{[r_2]} - \omega_2  c_{n_2-2}^{[r_2]}) + (\omega_1 s_{n_1-2;1}-\xi_1s_{n_1-3;1})(\xi_2c_{n_2-4}^{[r_2]} - \omega_2  c_{n_2-3}^{[r_2]})\\
 > r_2\left(A_{n_1-2;1}-r_1s_{n_1-3;1}\right)      
 \end{array}  $$

$$\stackrel{\textup{Lemma}~\ref{lem nonacyclic}}{\Longleftrightarrow} \quad (A_{n_1-1;1}-r_1s_{n_1-2;1})\omega_1 + (\omega_1 s_{n_1-2;1}-\xi_1s_{n_1-3;1})r_1 > \xi_1\left(A_{n_1-2;1}-r_1s_{n_1-3;1}\right)       $$

\begin{equation}\label{20140625eq7}\Longleftrightarrow \quad A_{n_1-1;1}\omega_1  > \xi_1 A_{n_1-2;1}  .      \end{equation}
Note that cases (a) and (b) agree on the first mutation sequence, and so equation (\ref{eq 44.1}) is valid in both cases, and it implies $ A_{n_1-1;1}\omega_1  > \xi_1 A_{n_1-2;1}  $. 
The proof that the exponents of $x_{e;t}$ in the part of (\ref{0303eq5d}) that contributes to $B'_t$ are non-negative
uses a similar argument. We give an outline as follows.

{
Applying Theorem~\ref{mainthm12062011} to $x_{d;t_{\omega_Z}}^{p_3}
x_{d';t_{\omega_Z}}^{q_3} $ in (\ref{0303eq5d}), the exponents of 
$x_{e;t}$ in the part 
of (\ref{0303eq5d}) that contributes to  $B_t'$ are of the form
\begin{equation}\label{20140624eq2}
\omega_2 s_{n_2-3;2}- \xi_2 s_{n_2-4;2}
+r_1|S_{1}|-A_{n_1-2;1} - (\omega_2 A_{n_2-3;2}
- \xi_2 A_{n_2-4;2}).
\end{equation}
Using the same process as above, we get the following analogue of (\ref{07272012_eq01}):
\begin{equation}\label{20140624eq3}
\xi_2 A_{n_2-5;2}- \omega_2 A_{n_2-4;2}
> r_2(A_{n_1-2;1} - r_1|S_{1}|),
\end{equation}
which is obtained by replacing $n_2$ with $n_2-1$ and $s_{n_1-3;1}$ with $|S_{1}|$.
Since $A_{i;2}=q_3c_{i+1}^{[r_2]} + p_3c_{i}^{[r_2]}$,
the inequality \eqref{20140624eq3} is equivalent to 
$$
p_3(\xi_2 c_{n_2-5}^{[r_2]} - \omega_2 c_{n_2-4}^{[r_2]})
+q_3(\xi_2 c_{n_2-4}^{[r_2]} - \omega_2 c_{n_2-3}^{[r_2]})> r_2(A_{n_1-2;1} - r_1|S_{1}|)
$$
{\tiny$$
\overset{Lemma~\ref{lem nonacyclic}}{\Longleftrightarrow}
(r_1|S_{2}|-A_{n_1-1;1})(\xi_1r_1-\omega_1) + 
(\xi_1(A_{n_1-1;1} - |S_{1}|) - (\xi_1r_1-\omega_1)|S_{2}|)r_1
> \xi_1(A_{n_1-2;1} - r_1|S_{1}|)
$$}
$$
\Longleftrightarrow
\omega_1A_{n_1-1;1} > \xi_1A_{n_1-2;1}, 
$$which is exactly the same as \eqref{20140625eq7}.}

 Thus the exponents of $x_{e;t}$ in $B_t'$ are non-negative.

This completes the proof of Proposition \ref{prop2.13} in cases (a) and (b), modulo the following two Lemmas.

\begin{lemma}\label{lem 514} $$s_{n_2-2;2} A_{n_2-2;2} >  s_{n_2-3;2} A_{n_2-1;2}.
$$ 
\end{lemma}
\begin{proof}
It follows from Definition~\ref{20120121}, that  $s'_{n_2-2} c_{n_2-2}^{[r_2]}  = s'_{n_2-3} c_{n_2-1}^{[r_2]}$, and using $(c_{n_2-1}^{[r_2]})^2>c_{n_2}^{[r_2]}c_{n_2-2}^{[r_2]}$ from Lemma~\ref{lem cn}, this implies
\begin{equation}\label{eq 514} s'_{n_2-2} c_{n_2-1}^{[r_2]}  > s'_{n_2-3} c_{n_2}^{[r_2]}  .\end{equation}

On the other hand, the second line of (\ref{cond502b0303}) in Theorem \ref{cor 26} together with  $A_{n_2-2}A_{n_2-2}>A_{n_2-1}A_{n_2-3}$ from Lemma~\ref{negone} implies

\[(s_{n_2-2} - s'_{n_2-2}) A_{n_2-2} > (s_{n_2-3} - s'_{n_2-3}) A_{n_2-1},     
\]
and, using the definition of $A_{n_2-i}$, we get
\[(s_{n_2-2} - s'_{n_2-2}) (p_2c^{[r_2]}_{n_2-1} + q_2c^{[r_2]}_{n_2-2}) > (s_{n_2-3} - s'_{n_2-3})(p_2c^{[r_2]}_{n_2} + q_2c^{[r_2]}_{n_2-1}).   \]
Now (\ref{eq 514}) implies  the statement.
\end{proof}

\begin{lemma}\label{0413lem} 
Let $(S_1=\cup_{i=1}^{p+q}\, S_{1}^i,S_2=\cup_{i=1}^{p+q}\, S_{2}^i)$ such that  
$$(S_{1}^i,S_{2}^i) \textup{ is a compatible pair in }\left\{\begin{array}{ll}\mathcal{D}^{c_{n_2}^{[r_2]}\times c_{n_2-1}^{[r_2]}}  &\textup{ if $1\le i \le p_2$}; \\  
 \mathcal{D}^{c_{n_2-1}^{[r_2]}\times c_{n_2-2}^{[r_2]}} &\textup{ if $p_2+1\le i\le p_2+q_2$;}
\end{array}\right.$$
where $p_2, q_2$ are arbitrary non-negative integers. 
Then 
\[
\frac{r_1|S_2|}{c_{n_2-1}^{[r_2]}} 
\le \xi_2 (A_{n_2-1;2}-|S_1|)-(\xi_2 r_2 - \omega_2)|S_2|\]
\end{lemma}
\begin{proof} This is proved in \cite[Lemma 4.10]{LS3v1} using colored subpaths with $|\beta|_2=A_{n_2-1;2} -|S_1| $ and $|\beta|_1=|S_2|$.  This is also proved in \cite[Proposition 4.1]{LLZ} case 6  with $a_2=A_{n_2-1;2}, b=c=r_2$.
\end{proof}

Now suppose we are in the case (c). Using Theorem \ref{mainthm12062011}, we see that $x_{d;t_{\omega_E}}^p x_{e;t_{\omega_E}}^q$ is equal to the following expression, which is almost the same as the expression in (\ref{0303eq5c}) and (\ref{0303eq5d}) up to the exponent $M'_{d'}$.
\begin{equation}\label{0729eq5a}
 \sum_{\begin{array}{c}\scriptstyle\tau_{0;1},\cdots,\tau_{n_1-3;1}\\ \scriptstyle A_{n_1-1;1}-r_1s_{n_1-2;1}\geq 0\end{array}}  \!\!\!\!\!\!\!\!\!\!\!\! \left( \prod_{w=0}^{n_1-3}\gchoose{\!\!\!A_{w+1;1} - r_1s_{w;1}\!\!\!}{\tau_{w;1}} \!\!  \right){\widetilde{x_{d;t_{\omega_Z}}}}^ {A_{n_1-1;1}-r_1s_{n_1-2;1}}x_{e;t_{\omega_Z}}^ {r_1s_{n_1-3;1}-A_{n_1-2;1}   }  x_{d';t_{\omega_Z}}^{\omega_1 s_{n_1-2;1} - \xi_1  s_{n_1-3;1}   -M'_{d'}}
 \end{equation}
\begin{equation}\label{0729eq5b}
+ \,\, x_{d;t_{\omega_Z}}^{-A_{n_1-1;1}} x_{e;t_{\omega_Z}}^{-A_{n_1-2;1}}
\sum_{\begin{array}{c}\scriptstyle(S_1,S_2)\\ \scriptstyle-A_{n_1-1;1}+r_1|S_2|> 0\end{array}}
x_{d;t_{\omega_Z}}^{r_1|S_2|}
x_{e;t_{\omega_Z}}^{r_1|S_1|}
 x_{d';t_{\omega_Z}}^{\xi_1(A_{n_1-1;1}-|S_1|)-(\xi_1r_1-\omega_1)|S_2|-M_{d'}}.
\end{equation}

Let $p_2=r_1|S_2|-A_{n_1-1;1}$ and $q_2={\xi_1(A_{n_1-1;1}-|S_1|)-(\xi_1r_1-\omega_1)|S_2|-M_{d'}}$ 
be the exponents of $x_{d;t_{\omega_Z}}$ and $x_{d';t_{\omega_Z}}$ in (\ref{0729eq5b}), respectively. Applying Theorem~\ref{mainthm12062011} to $x_{d;t_{\omega_Z}}^ {p_2} x_{d';t_{\omega_Z}}^{q_2}$, we see that the part of (\ref{0729eq5b}) that contributes to $C'_t$ is equal to    $$\aligned
&\sum_{\begin{array}{c}\scriptstyle(S_1,S_2)\\ \scriptstyle-A_{n_1-1;1}+r_1|S_2|> 0\end{array}} x_{e;t_{\omega_Z}}^{r_1|S_1|-A_{n_1-2;1}}\\
&\times \sum_{\begin{array}{c} \scriptstyle(S_1',S_2')\\\scriptstyle { A_{n_2-2;2} <r_2|S_2'| } \end{array}} 
x_{d;t}^{r_2|S_2'| - A_{n_2-2;2}} 
x_{d';t}^{r_2|S_1'|-A_{n_2-3;2} } 
x_{e;t}^{ \xi_2(A_{n_2-2;2}-|S'_1|) - (\xi_2r_2 -\omega_2 )|S_2'| } \endaligned
$$
\\
\begin{equation}\label{eq28}
\aligned&=\sum_{\begin{array}{c}\scriptstyle(S_1,S_2)\\ \scriptstyle-A_{n_1-1;1}+r_1|S_2|> 0\end{array}}  \sum_{\begin{array}{c} \scriptstyle(S_1',S_2')\\\scriptstyle { A_{n_2-2;2} <r_2|S_2'| } \end{array}}
x_{d;t}^{r_2|S_2'| - A_{n_2-2;2}} 
x_{d';t}^{r_2|S_1'|-A_{n_2-3;2} } \\&\times
x_{e;t}^{ \xi_2(A_{n_2-2;2}-|S'_1|) - (\xi_2r_2 -\omega_2 )|S_2'|+{r_1|S_1|-A_{n_1-2;1}}} 
\endaligned\end{equation}

Since  $A_{n_2-2;2}<r_2|S_2'|$ in this expression, we have 
$A_{n_2-2;2} / r_2|S_2'|<1$ and thus
\begin{equation}\label{eq29}\frac{r_1A_{n_2-2;2}}{c_{n_2-2}^{[r_2]} \,r_2} =
 \frac{r_1|S_2'|}{c_{n_2-2}^{[r_2]} } \frac{\,A_{n_2-2;2}}{r_2|S_2'|}
<
\frac{r_1|S_2'|}{c_{n_2-2}^{[r_2]}} 
\le \xi_2(A_{n_2-2;2}-|S_1'|)-(\xi_2 r_2 - \omega_2)|S_2'|,\end{equation}
where the last inequality is proved in Lemma~\ref{0413lem}. On the other hand, 

\begin{equation}\label{eq30} \aligned
A_{n_2-2;2} &= p_2 c_{n_2-1}^{[r_2]} + q_2 c_{n_2-2}^{[r_2]} \\
&=(r_1|S_2|-A_{n_1-1;1}) c_{n_2-1}^{[r_2]} +    {\xi_1(A_{n_1-1;1}-|S_1|)-(\xi_1r_1-\omega_1)|S_2|-M_{d'}})c_{n_2-2}^{[r_2]} 
\endaligned \end{equation}

The exponent of $x_{e;t}$ is
\[\aligned  &{ \xi_2(A_{n_2-2;2}-|S'_1|) - (\xi_2r_2 -\omega_2 )|S_2'|+{r_1|S_1|-A_{n_1-2;1}}}\\
\stackrel{(\ref{eq29})}{>}&\ \frac{r_1A_{n_2-2;2}}{c_{n_2-2}^{[r_2]} \,r_2} +  r_1|S_1|-A_{n_1-2;1} \\
\stackrel{(\ref{eq30})}{=}&\  \frac{r_1}{c_{n_2-2}^{[r_2]} \,r_2} \left((r_1|S_2|-A_{n_1-1;1}) c_{n_2-1}^{[r_2]} +    ({\xi_1(A_{n_1-1;1}-|S_1|)-(\xi_1r_1-\omega_1)|S_2|-M_{d'}})c_{n_2-2}^{[r_2]} \right)\\ 
& \qquad + r_1|S_1|-A_{n_1-2;1} \\
\stackrel{\textup{Lemma~\ref{negone}}}{=} & \  \frac{r_1}{c_{n_2-2}^{[r_2]} \,r_2} \left((r_1|S_2|-A_{n_1-1;1}) c_{n_2-1}^{[r_2]} -((\xi_1r_1-\omega_1)|S_2|-M_{d'})c_{n_2-2}^{[r_2]} \right) 
 + A_{n_1;1}\\
\stackrel{\textup{Lemma~\ref{lem cn}}}{=}& \  \frac{r_1}{c_{n_2-2}^{[r_2]} \,r_2} \left((\omega_1c_{n_2-2}^{[r_2]} - r_1 c_{n_2-3}^{[r_2]})|S_2| -A_{n_1-1;1} c_{n_2-1}^{[r_2]} - M_{d'}c_{n_2-2}^{[r_2]} \right) + A_{n_1;1} \\
 \stackrel{(*)}{\geq}&\   \frac{r_1}{c_{n_2-2}^{[r_2]} \,r_2} \left((\omega_1c_{n_2-2}^{[r_2]} - r_1 c_{n_2-3}^{[r_2]})\frac{A_{n_1-1;1}}{r_1} -A_{n_1-1;1} c_{n_2-1}^{[r_2]} - M_{d'}c_{n_2-2}^{[r_2]} \right) + A_{n_1;1}\\
= & \ \frac{r_1}{c_{n_2-2}^{[r_2]} \,r_2} \left((\omega_1c_{n_2-2}^{[r_2]} - r_1 c_{n_2-3}^{[r_2]})\frac{A_{n_1-1;1}}{r_1}-A_{n_1-1;1} c_{n_2-1}^{[r_2]}\right. \\
&\left. \ - ( \xi_1A_{n_1-1;1}-(\xi_1r_1-\omega_1)A_{n_1-2;1}  )c_{n_2-2}^{[r_2]} \right) + A_{n_1;1}
\endaligned\]
where inequality (*) holds \textup{because $p_2=r_1|S_2| - A_{n_1-1;1}\ge 0$}.
We shall show below that the quiver $Q_2$ is cyclic, and therefore $\xi_2$ and $r_2$ have the same sign, so $r_2>0$, since $\xi_2>0$. This implies that the above expression has the same sign as 
\[\aligned
&r_1 \left((\omega_1c_{n_2-2}^{[r_2]} - r_1 c_{n_2-3}^{[r_2]})\frac{A_{n_1-1;1}}{r_1}-A_{n_1-1;1} c_{n_2-1}^{[r_2]} - ( \xi_1A_{n_1-1;1}-(\xi_1r_1-\omega_1)A_{n_1-2;1}  )c_{n_2-2}^{[r_2]} \right) +A_{n_1;1}c_{n_2-2}^{[r_2]} \,\xi_1\\
&=(\omega_1c_{n_2-2}^{[r_2]} - r_1 c_{n_2-3}^{[r_2]})A_{n_1-1;1}+ ( r_1(\xi_1r_1-\omega_1)A_{n_1-2;1} -r_1\xi_1A_{n_1-1;1}  +A_{n_1;1} \,\xi_1)c_{n_2-2}^{[r_2]}  - r_1 A_{n_1-1;1} c_{n_2-1}^{[r_2]}\\
&=(\omega_1c_{n_2-2}^{[r_2]} - r_1 c_{n_2-3}^{[r_2]})A_{n_1-1;1}+ (\xi_1(r_1^2-1)-\omega_1\,r_1)A_{n_1-2;1} c_{n_2-2}^{[r_2]}  - r_1 A_{n_1-1;1} c_{n_2-1}^{[r_2]}
\\
&=(\omega_1c_{n_2-2}^{[r_2]} - r_1\xi_1 c_{n_2-2}^{[r_2]})A_{n_1-1;1}+ (\xi_1(r_1^2-1)-\omega_1\,r_1)A_{n_1-2;1} c_{n_2-2}^{[r_2]} \\
\endaligned\]
and this expression has the same sign as 
\begin{equation}\label{20140626eq1}\aligned
&(\omega_1 - r_1\xi_1)A_{n_1-1;1}+ (\xi_1(r_1^2-1)-\omega_1\,r_1)A_{n_1-2;1}  \\
&=\xi_1A_{n_1-4;1}-\omega_1A_{n_1-3;1}.
\endaligned\end{equation}
This expression is positive because of Lemma~\ref{lem nonacyclic}. 

Let us now show that the quiver $Q_2$ is cyclic in this case. 
Since we are in case (c), the mutation sequence from 
$t^{**}$ to $t$ is of almost cyclic type. Moreover this sequence is of length at least 3 and the quiver $Q_2$ is the quiver at the seed $t'$ one step before reaching the seed $t$ in this sequence.
We need to consider conditions (1) and (2) of Definition \ref{def 12}. If condition (2) holds then all the quivers after the second mutation in this sequence are cyclic, so in particular $Q_2$ is cyclic. Suppose now that condition (1) holds. In our situation this condition says 
$c_n^{[r_2]} \xi_2 - c_{n-1}^{[r_2]} \omega_2 >0$ for $1\le n\le m$, where $m$ is the length of the mutation sequence. 
Using $n=1$ and $2$, this implies that $\omega_2>0$ and $\xi_2>0$, and thus $Q_1$ is cyclic. Because we are in case (c), the mutation sequence from $t_{w_E}$ to $t^*$ is of acyclic type, and it follows that there is an acyclic quiver in one of the seeds preceding $Q_1$ in that sequence. 
The facts that $r_1>1$ and that the last mutation to get to $Q_1$ is in direction $e$ then imply that $\xi_1>\omega_1$, because $\xi_1$ is the number of arrows opposite to the vertex $e$. But then the first mutation in the sequence from $Q_1$ to $Q_2$ is in direction $d$ and afterwards the sequence alternates between directions $d'$ and $d$, and therefore the number of arrows in the quivers of this sequence grows. In particular all quivers in that sequence are cyclic.

The proof that the exponents of $x_{e;t}$ in the part of (\ref{0729eq5a}) that contributes to $C'_t$ are non-negative
uses a similar argument, which is given below. 
 
Let $p_3=A_{n_1-1;1} - r_1s_{n_1-2;1}$ and
$q_3=\omega_1s_{n_1-2;1} - \xi_1s_{n_1-3;1} -M'_{d'}$,  and let $A_{i;3}=p_3 c_{i+1}^{[r_2]}+q_3 c_{i}^{[r_2]} $.
Applying Theorem~\ref{mainthm12062011} to $\widetilde{x_{d;t_{\omega_Z}}}^{p_3}
x_{d';t_{\omega_Z}}^{q_3} $ in (\ref{0729eq5a}), the exponents of 
$x_{e;t}$ in the part 
of (\ref{0729eq5a}) that contributes to  $C_t'$ are of the form
\begin{equation}\label{20140625eq1}
\xi_2(A_{n_2-2;3}-|S_{1}'|) - (\xi_2r_2-\omega_2)|S_{2}'|
+r_1s_{n_1-3;1}-A_{n_1-2;1}.
\end{equation}
Using (\ref{eq29}) with $n_2-2$ replaced by $n_2-1$, we have
{\tiny$$\aligned
&\eqref{20140625eq1}
> \frac{r_1A_{n_2-1;3}}{c_{n_2-1}^{[r_2]}r_2}
+r_1s_{n_1-3;1}-A_{n_1-2;1}\\
&=\frac{r_1}{c_{n_2-1}^{[r_2]}\xi_1}
\left(p_3 c_{n_2}^{[r_2]}+q_3 c_{n_2-1}^{[r_2]} \right)+r_1s_{n_1-3;1}-A_{n_1-2;1}\\
&=\frac{r_1}{c_{n_2-1}^{[r_2]}\xi_1}
\left(p_3 c_{n_2}^{[r_2]}+(\omega_1s_{n_1-2;1} - \xi_1s_{n_1-3;1} -M'_{d'}) c_{n_2-1}^{[r_2]} \right)+r_1s_{n_1-3;1}-A_{n_1-2;1}\\
&=\frac{r_1}{c_{n_2-1}^{[r_2]}\xi_1}
\left(p_3 c_{n_2}^{[r_2]}+(\omega_1s_{n_1-2;1} -M'_{d'}) c_{n_2-1}^{[r_2]} \right)-A_{n_1-2;1}\\
&=\frac{r_1}{c_{n_2-1}^{[r_2]}\xi_1}
\left((A_{n_1-1;1} - r_1s_{n_1-2;1}) c_{n_2}^{[r_2]}+(\omega_1s_{n_1-2;1} -M'_{d'}) c_{n_2-1}^{[r_2]} \right)-A_{n_1-2;1}\\
&=\frac{r_1}{c_{n_2-1}^{[r_2]}\xi_1}
\left(A_{n_1-1;1}c_{n_2}^{[r_2]} - s_{n_1-2;1}(r_1c_{n_2}^{[r_2]}-\omega_1c_{n_2-1}^{[r_2]}) -M'_{d'}c_{n_2-1}^{[r_2]}  \right)-A_{n_1-2;1}\\
&\underset{(A_{n_1-1;1} - r_1s_{n_1-2;1}\geq 0)}{>} \ \ 
\frac{r_1}{c_{n_2-1}^{[r_2]}\xi_1}
\left(A_{n_1-1;1}c_{n_2}^{[r_2]} - \frac{A_{n_1-1;1}}{r_1}(r_1c_{n_2}^{[r_2]}-\omega_1c_{n_2-1}^{[r_2]}) -M'_{d'}c_{n_2-1}^{[r_2]}  \right)-A_{n_1-2;1}\\
&=
\frac{r_1}{c_{n_2-1}^{[r_2]}\xi_1}
\left(\frac{A_{n_1-1;1}}{r_1}\omega_1c_{n_2-1}^{[r_2]} -M'_{d'}c_{n_2-1}^{[r_2]}  \right)-A_{n_1-2;1}\\
&=
\frac{r_1}{c_{n_2-1}^{[r_2]}\xi_1}
\left(\frac{A_{n_1-1;1}}{r_1}\omega_1c_{n_2-1}^{[r_2]} -(\omega_1A_{n_1-2;1} - \xi_1A_{n_1-3;1})c_{n_2-1}^{[r_2]}  \right)-A_{n_1-2;1}\\
&=
\frac{\omega_1A_{n_1-1;1}}{\xi_1} -\frac{r_1}{\xi_1}
(\omega_1A_{n_1-2;1} - \xi_1A_{n_1-3;1}) -A_{n_1-2;1}\\
&=
-\frac{\omega_1A_{n_1-3;1}}{\xi_1} +A_{n_1-4;1},\\
\endaligned
$$}
which has the same sign as 
$$
\xi_1A_{n_1-4;1} - \omega_1A_{n_1-3;1},
$$which is equal to \eqref{20140626eq1}.

 Thus the exponents of $x_{e;t}$ in $C_t'$ are non-negative, and this shows Proposition \ref{prop2.13} in the case (c).

To complete the proof of the proposition, we analyze  the case (d).  Suppose that the sequence from $t_{w_B}$ to $t^*$  is of acyclic type and consider the quiver $Q_1$ at the seed $t^{**}=\mu_d (t^*)$. 
Our first goal is to show that $\mu_e Q_1$ is acyclic. 
Suppose the contrary. 
Thus condition (7) in Definition \ref{def 12} does not hold. Hence condition (6) in Definition \ref{def 12}
implies
 that the number of arrows $r_1$ from $d$ to $e$ in $Q_1$ is at least 2.  In this case, the acyclic quivers in the  sequence from $t_{w_B}$ to $t^*$  form a  connected subsequence, and thus $\mu_e Q_1$ being cyclic implies that $Q_1$ and $\mu_d Q_1 $ are cyclic too. 
 The third case of Lemma \ref{lem nonacyclic} with $\omega=\xi_1$ and $ r=\omega_1$ implies that 
 $$\xi_1\ge \frac{c_n^{[r_1]}}{c_{n-1}^{[r_1]}} \,\omega_1 > \omega_1,$$ where the last inequality holds since ${c_n^{[r_1]}}>{c_{n-1}^{[r_1]}} $.
 Moreover, using the third case of Lemma \ref{lem nonacyclic} with $\omega=\omega_1$ and $ r=r_1$ together with the fact that the mutation sequence from $Q_1$ to $Q_2$ is of acyclic type, we also see that $\omega_1\ge r_1$.
 Thus $\xi_1> \omega_1\ge r_1\ge 2$, and therefore Lemma \ref{rem 12} 
 implies that all quivers in the sequence from $t^*$ to $t$ are cyclic, a contradiction to the assumption that we are in case (d).

Thus  $\mu_e Q_1$ is acyclic. By the same reasoning, we get that $Q_1 $ is acyclic, and then by symmetry, we also have that the quivers in the seeds $t^*$ and $\mu_{d'} (t^*)$ are acyclic.

Thus we have shown that the case (d) occurs if and only if the four quivers in the consecutive seeds  $\mu_{e} (t^{**})$, $t^{**}$, $t^*$ and $\mu_{d'} (t^*)$ are acyclic. In this case, the sequence from $t^*$ to $t$ is of almost cyclic type and the result can be shown by the same argument as in the case (c).
This completes the proof of Proposition \ref{prop2.13}.

\subsection{Proof of Proposition \ref{prop2.14}}\label{proof prop2.14}
Proposition \ref{prop2.14} applies in cases (a) and (c). We  assume that we are in the case (a). The case (c) is similar.
 
 Let   $g$ be an arbitrary vertex different from $d,d',e$ and let
$$Q_1=\xymatrix@R20pt@C80pt{ d  \ar@{<-}[r]^{\xi_1}\ar@{<-}[dr]|(0.3){\rho_1}
& d' \ar@{<-}[dl]|(0.3){\omega_1}\ar@{<-}[d]  \\ e \ar@{<-}[u]^{r_1} &g\ar@{<-}[l]^{\nu_1} }$$ 
be the full subquiver with vertices $d,d',e'g$ of the quiver at the seed $\mu_d(t_{w_Z})=t^{**}$, where $r_1$ (respectively $\omega_1$, $\xi_1$, $\nu_1$, $\rho_1$) is the number of arrows from $d$ to $e$ (respectively from $e$ to $d'$, from  $d'$ to $d$, from $e$ to $g$, from $g$ to $d$ ). Recall that $r_1\ge 0$.
Although the mutations in the sequence $\mu$ are in directions $d,d',e$ only, we need to study how the variable $x_{g,t}$ behaves in the expansion formulas in order to show certain divisibility properties. On the other hand, it suffices to consider only one of the $x_{g;t}$ with $g\ne d,d',e$.

We present the case where the subquiver with vertices $d, e, g$ in some seed between $\mu_d({t_{w_E}})$ and $\mu_d(t_{w_Z})$ is acyclic. The case where all these subquivers are non-acyclic is easier; in fact, only the exponent of $x_{g,t_{w_Z}}$ would change.
Thus for the rest of this proof, we set $x_{f;t}=1$ for all $f\ne d,d',e'g$.

Since the mutation sequence relating the seeds $t_{w_E}$ and $t_{w_Z}=t^*$ consists in mutations in directions 
$d$ and $e$,  
applying Theorem \ref{cor 26} to $x_{d;t_{w_E}}^p x_{e;t_{w_E}}^q$ yields\footnote{The term $(\nu_1 A_{n_1-2;1} - \rho_1 A_{n_1-3;1})$ in the exponent of $x_{g;t_{w_Z}}$ is equal to the $M_g$ in Theorem~\ref{cor 26}, and it is non-zero, because of the assumption that the subquiver with vertices $d, e, g$ in some seed between $\mu_d({t_{w_E}})$ and $\mu_d(t_{w_Z})$ is acyclic. On the other hand, the term $M_{d'}$ in the exponent of $x_{d';t_{w_Z}}$ is zero, because we are in case (a).}

\begin{equation}\label{0303eq5}
\aligned
\sum_{\tau_{0;1},\tau_{1;1},\cdots,\tau_{n_1-3;1}}  \!\!\!\! \left( \prod_{w=0}^{n_1-3}\!\gchoose{\!\!A_{w+1;1} - r_1s_{w;1}\!\!}{\tau_{w;1}} \!\!  \right)
 & \widetilde{x_{d;t_{w_Z}}}^ {A_{n_1-1;1}-r_1s_{n_1-2;1}}x_{e;t_{w_Z}}^ {r_1s_{n_1-3;1}-A_{n_1-2;1}   } \\
& \times x_{d';t_{w_Z}}^{\omega_1 s_{n_1-2;1} - \xi_1 s_{n_1-3;1}   } x_{g;t_{w_Z}}^{\nu_1 s_{n_1-2;1} - \rho_1 s_{n_1-3;1} - (\nu_1 A_{n_1-2;1} - \rho_1 A_{n_1-3;1})  }.
\endaligned\end{equation} 

As before, let $n_2$ be the number of seeds between $\mu_d(t_{w_Z})=t^{**}$ and $t$ inclusive.  Recall that $n_2$ is even. Let 
$$Q_2=\xymatrix@R20pt@C80pt{ d  \ar@{<-}[r]^{r_2}
& d'\ar@{<-}[dl]|(0.3){\omega_2} \ar@{<-}[d]^{\nu_2} \\ e \ar@{<-}[r]_{\gamma_2}\ar@{<-}[u]^{\xi_2} &g\ar@{<-}[ul]|(0.7){\rho_2} }$$ 
be the quiver at the seed $\mu_d(t)=t'$, where $r_2$ (respectively $\omega_2$, $\xi_2$, $\nu_2$, $\rho_2$, $\gamma_2$) is the number of arrows from $d'$ to $d$ (respectively from $e$ to $d'$, from  $d$ to $e$, from $g$ to $d'$, from $d$ to $g$, from $g$ to $e$). Recall that $r_2\ge 0$.

 Next we want to use Lemma~\ref{lem nonacyclic} to compute the number of arrows in $Q_2$ in terms of the number of arrows in $Q_1$. The mutation sequence relating the quivers $Q_1$ and $Q_2$ is a rank 2 sequence in directions $d$ and $d'$, starting with $d$ and ending with $d'$. In particular, the number of mutations in this sequence (which is denoted  $n$ in Lemma~\ref{lem nonacyclic}) is even and we have $n=n_2-2$.
First we apply   Lemma  \ref{lem nonacyclic} to the subquiver on vertices $d,d',e$. Since we are in the case (a),  the condition $c_{n_2-1}^{[r_2]}r_1- c_{n_2-2}^{[r_2]}\omega_1>0$ holds.
Therefore we see from  Lemma \ref{lem nonacyclic} that

\begin{eqnarray} \label{eq 4.8.1}
 r_2&=&\xi_1,\nonumber\\
\omega_2&=&c_{n_2-1}^{[r_2]}r_1- c_{n_2-2}^{[r_2]}\omega_1,\\
 \xi_2&=&c_{n_2}^{[r_2]}r_1- c_{n_2-1}^{[r_2]}\nonumber\omega_1,\end{eqnarray}
 
 Next, we use Lemma \ref{lem nonacyclic}  on the subquiver on vertices $d,d' $ and $g$. Since we assumed that the mutation sequence on this subquiver is of acyclic type and because of Remark \ref{Rem acyclic}, we use the third case of  Lemma~\ref{lem nonacyclic} with $\bar{\omega}(n)=-\rho_1$, $\xi=r_2$, $\omega=\nu_2$, and $r=\rho_2$ to obtain
 
 \[ \rho_1=    c_{n_2-3}^{[r_2]}\nu_2 - c_{n_2-4}^{[r_2]}\rho_2. \]   
  Finally, we show by induction on $n_2$ that 
\begin{equation}\label{gamma2}\gamma_2  = ( c_{n_2-2}^{[r_2]}\rho_2 - c_{n_2-1}^{[r_2]}\nu_2 )r_1 -  ( c_{n_2-3}^{[r_2]}\rho_2 - c_{n_2-2}^{[r_2]}\nu_2 )\omega_1 -\nu_1.\end{equation}
If $n_2=4$, then the above formula becomes $\gamma_2  = (\rho_2 - r_2\nu_2 )r_1 -  ( - \nu_2 )\omega_1 -\nu_1$. This is easily checked, since $Q_1$ and $Q_2$ are related by sequence of two mutations. Suppose $n_2>4$. We want to show  that if equation (\ref{gamma2}) computes the number of arrows from $e$ to $g$ in the quiver $Q_2$ then it also computes the number of arrows from $e$ to $g$ in the quiver $\mu_{d'}\mu_d Q_2$, if $n_2$ is replaced by $n_2+2$, and $\rho_2$ (respectively $\nu_2$) is replaced by the number of arrows in $\mu_{d'}\mu_d Q_2$ from $g$ to $d$ (respectively $d'$ to $g$). 

Denoting by $\rho_2'$ (respectively $\nu_2'$)  the number of arrows from $g$ to $d$  (respectively $d'$ to $g$)
in  $\mu_{d'}\mu_d Q_2$, we observe that $\rho_2'=(r_2(r_2\rho_2-\nu_2) - \rho_2)$,
 $\nu_2'=(r_2\rho_2 - \nu_2)$ and the number of arrows from $e$ to $g$ is still $\gamma_2$.

Now using the definition $c_{n-2}^{[r_2]}=r_2c_{n-1}^{[r_2]}-c_{n}^{[r_2]}$ we see that

\[\begin{array}{rcl}c_{n_2-2}^{[r_2]}\rho_2 - c_{n_2-1}^{[r_2]}\nu_2 
&=&(r_2c_{n-1}^{[r_2]}-c_n^{[r_2]}) \rho_2 - (r_2c_{n}^{[r_2]}-c_{n+1}^{[r_2]})\nu_2 
\\
&=&(r_2(r_2c_{n}^{[r_2]}-c_{n+1}^{[r_2]})-c_n^{[r_2]}
) \rho_2 - (r_2c_{n}^{[r_2]}-c_{n+1}^{[r_2]})\nu_2 

\\
&= &c_{n}^{[r_2]}(r_2(r_2\rho_2-\nu_2) - \rho_2) - c_{n+1}^{[r_2]}(r_2\rho_2 - \nu_2)
\\&= &c_{n}^{[r_2]}\rho_2'- c_{n+1}^{[r_2]} \nu_2'
\end{array}\]
Similarly
\[ c_{n_2-3}^{[r_2]}\rho_2 - c_{n_2-2}^{[r_2]}\nu_2 
=  c_{n_2-1}^{[r_2]}\rho_2' - c_{n_2}^{[r_2]}\nu'_2 .
\]
This proves formula (\ref{gamma2}).

Now let 
$p_2 ={A_{n_1-1;1}-r_1s_{n_1-2;1}}$, $q_2={\omega_1 s_{n_1-2;1} - \xi_1 s_{n_1-3;1}   }$ be the exponents of $\widetilde{x_{d;t_{w_Z}}}$ and $x_{d';t_{w_Z}}$ in (\ref{0303eq5}) respectively.
 Applying Theorem~\ref{cor 26} to $\widetilde{x_{d;t_{w_Z}}}^{p_2}x_{d';t_{w_Z}}^{q_2}$,
  we see that (\ref{0303eq5}) is equal to
\begin{equation}\label{0421eq1}
\aligned
&\sum_{\tau_{0;1},\tau_{1;1},\cdots,\tau_{n_1-3;1}}   \left( \prod_{w=0}^{n_{1}-3}\gchoose{A_{w+1;1} - r_1s_{w;1}}{\tau_{w;1}}   \right) x_{e;t}^ {r_1s_{n_1-3;1}-A_{n_1-2;1} } x_{g;t}^{\nu_1 s_{n_1-2;1} - \rho_1 s_{n_1-3;1} - (\nu_1 A_{n_1-2;1} - \rho_1 A_{n_1-3;1})  } \\
&\times \sum_{\tau_{0;2},\tau_{1;2},\cdots,\tau_{n_{2}-3;2}}   \left( \prod_{w=0}^{n_{2}-3}\gchoose{A_{w+1;2} - r_{2}s_{w;2}}{\tau_{w;2}}   \right) \\
&\times  \widetilde{x_{d;t}}^ {A_{n_2-1;2}-r_{2} s_{n_{2}-2;2} }
 x_{d';t}^{ r_2s_{n_2-3;2}-A_{n_2-2;2}} x_{e;t}^{\omega_2s_{n_2-2;2}-\xi_2s_{n_2-3;2}}\\
&\times  x_{g;t}^{\nu_2s_{n_2-2;2}-\rho_2s_{n_2-3;2}-(\nu_2A_{n_2-2;2}-\rho_2A_{n_2-3;2})}\\
  \\
 =
&\sum_{\tau_{0;1},\tau_{1;1},\cdots,\tau_{n_1-3;1}}   \left( \prod_{w=0}^{n_{1}-3}\gchoose{A_{w+1;1} - r_1s_{w;1}}{\tau_{w;1}}   \right)  \\
&\times \sum_{\tau_{0;2},\tau_{1;2},\cdots,\tau_{n_{2}-3;2}}   \left( \prod_{w=0}^{n_{2}-3}\gchoose{A_{w+1;2} - r_{2}s_{w;2}}{\tau_{w;2}}   \right) \\
&\times  \widetilde{x_{d;t}}^ {A_{n_2-1;2}-r_{2} s_{n_{2}-2;2} }
 x_{d';t}^{ r_2s_{n_2-3;2}-A_{n_2-2;2}}
 x_{e;t}^{\omega_2s_{n_2-2;2}-\xi_2s_{n_2-3;2}+r_1s_{n_1-3;1}-A_{n_1-2;1} }  \\
 &\times  x_{g;t}^{\nu_2s_{n_2-2;2}-\rho_2s_{n_2-3;2}-(\nu_2A_{n_2-2;2}-\rho_2A_{n_2-3;2}) + \nu_1 s_{n_1-2;1} - \rho_1 s_{n_1-3;1}- (\nu_1 A_{n_1-2;1} - \rho_1 A_{n_1-3;1})}
 \endaligned\end{equation}
where $A_{i;2}$ and $s_{i;2}$ are as defined before Lemma~\ref{negone} and Lemma \ref{lem sn} but in terms of $p_2, q_2,$ and $r_2$, thus $A_{i;2}=p_2c_{i+1}^{[r_2]}+q_2c_i^{[r_2]}$ and $s_{i;2}=\sum_{j=0}^{i-1}c_{i-j+1}^{[r_2] }\tau_{j;2}$. 

Let $\theta$ be a positive integer, and let $P_\theta$ be the sum of all the terms in the sum above for which  the exponent of $x_{e;t}$ is equal to $-\theta$. Thus $-\theta $ is equal to
\[{\omega_2s_{n_2-2;2}-\xi_2s_{n_2-3;2}+r_1s_{n_1-3;1}-A_{n_1-2;1} }.
\]

\subsubsection{Computation of exponents}\label{sect exponents} In this subsubsection, we compute the exponents in the expression in (\ref{0421eq1}). The main computation continues in \ref{sect back to main}.
It is convenient to introduce $\varsigma$ such that $\tau_{0;2}=\varsigma -s_{n_1-3;1}$. Then 
$$\aligned 
s_{n_2-2;2}&=c_{n_2-1}^{[r_2]} (\varsigma -s_{n_1-3;1})+ \sum_{j=1}^{n_2-3} c_{n_2-1-j}^{[r_2]} \tau_{j;2},\text{ and }\\
s_{n_2-3;2}&=c_{n_2-2}^{[r_2]} (\varsigma -s_{n_1-3;1})+ \sum_{j=1}^{n_2-4} c_{n_2-2-j}^{[r_2]}\tau_{j;2}.
\endaligned$$

Using equation (\ref{eq 4.8.1}), the expressions for $s_{n_2-2;2}$ and $s_{n_2-3;2}$ and the fact that $c_1^{[\xi]}=0$, we have

\begin{equation*}
\aligned & \omega_2s_{n_2-2;2}-\xi_2s_{n_2-3;2 } \\
=&\ (c_{n_2-1}^{[\xi_1]}r_1-c_{n_2-2}^{[\xi_1]} \omega_1)
\left[c_{n_2-1}^{[\xi_1]}(\varsigma -s_{n_1-3;1}) +  \sum_{j=1}^{n_2-3} c_{n_2-1-j}^{[\xi_1]}\tau_{j;2} \right] \\
&-(c_{n_2}^{[\xi_1]}r_1-c_{n_2-1}^{[\xi_1]} \omega_1)
\left[c_{n_2-2}^{[\xi_1]}(\varsigma -s_{n_1-3;1}) +  \left(\sum_{j=1}^{n_2-3} c_{n_2-2-j}^{[\xi_1]}\tau_{j;2}\right) - c_1^{[\xi_1]}\tau_{n_2-3;2}\right] \\
=&\ (\varsigma -s_{n_1-3;1})\,r_1\left( (c_{n_2-1}^{[\xi_1]})^2-c_{n_2}^{[\xi_1]}c_{n_2-2}^{[\xi_1]}\right) \\
&+\sum_{j=1}^{n_2-3} \tau_{j;2} \left[ r_1\left(c_{n_2-1}^{[\xi_1]}c_{n_2-1-j}^{[\xi_1]} -c_{n_2}^{[\xi_1]}c_{n_2-2-j}^{[\xi_1]}\right)
+\omega_1\left(-c_{n_2-2}^{[\xi_1]}c_{n_2-1-j}^{[\xi_1]}+c_{n_2-1}^{[\xi_1]}c_{n_2-2-j}^{[\xi_1]}
\right)
\right] \\
{=}
&\ (\varsigma -s_{n_1-3;1})\,r_1
+\sum_{j=1}^{n_2-3} \tau_{j;2} \left[ r_1\left(-c_{-j}^{[\xi_1]}\right)
+\omega_1c_{1-j}^{[\xi_1]}\right] \qquad \textup{(by Lemma \ref{lem cn})}
\\
\endaligned
\end{equation*}
And since $-c_{-j}^{[\xi_1]}=c_{j+2}^{[\xi_1]}$, we get

\begin{equation}\label{theta}\aligned-\theta 
&=r_1(\varsigma -s_{n_1-3;1}) + \sum_{j=1}^{n_2-3} \tau_{j;2}(c_{j+2}^{[\xi_1]}r_1 - c_{j+1}^{[\xi_1]}\omega_1)+ r_1s_{n_1-3;1}-A_{n_1-2;1} \\
&=-A_{n_1-2;1} +r_1\varsigma + \sum_{j=1}^{n_2-3} \tau_{j;2}(c_{j+2}^{[\xi_1]}r_1 - c_{j+1}^{[\xi_1]}\omega_1).
\endaligned\end{equation}
Also, the exponents of $\widetilde{x_{d;t}}$ and  $x_{d';t}$ in (\ref{0421eq1}) can be expressed as follows:
$$\aligned
A_{n_2-1;2}-r_{2} s_{n_{2}-2;2} = \ & c_{n_2}^{[\xi_1]}p_2 + c_{n_2-1}^{[\xi_1]}q_2
- \xi_1\left(c_{n_2-1}^{[\xi_1]} (\varsigma -s_{n_1-3;1})+ \sum_{j=1}^{n_2-3} c_{n_2-1-j}^{[\xi_1]} \tau_{j;2}\right)\\
= \ & c_{n_2}^{[\xi_1]}(A_{n_1-1;1}-r_1s_{n_1-2;1} ) + c_{n_2-1}^{[\xi_1]}(\omega_1 s_{n_1-2;1} - \xi_1 s_{n_1-3;1}) \\ 
& - \xi_1\left(c_{n_2-1}^{[\xi_1]} (\varsigma -s_{n_1-3;1})+ \sum_{j=1}^{n_2-3} c_{n_2-1-j}^{[\xi_1]} \tau_{j;2}\right)\\
 =\ & c_{n_2}^{[\xi_1]}(A_{n_1-1;1}-r_1s_{n_1-2;1} ) + c_{n_2-1}^{[\xi_1]}\omega_1 s_{n_1-2;1}\\ 
&- \xi_1\left(c_{n_2-1}^{[\xi_1]} \varsigma+ \sum_{j=1}^{n_2-3} c_{n_2-1-j}^{[\xi_1]} \tau_{j;2}\right),\endaligned$$
and similarly
\[\aligned \xi_1s_{n_2-3;2}-A_{n_2-2;2} 
= \ & \xi_1\left(c_{n_2-2}^{[\xi_1]} \varsigma + \sum_{j=1}^{n_2-4} c_{n_2-2-j}^{[\xi_1]}\tau_{j;2}\right) \\
& - \left( c_{n_2-1}^{[\xi_1]}(A_{n_1-1;1}-r_1s_{n_1-2;1} ) + c_{n_2-2}^{[\xi_1]}\omega_1 s_{n_1-2;1} \right).
\endaligned\]

Recall that $\rho_1=c_{n_2-3}^{[r_2]}\nu_2-c_{n_2-4}^{[r_2]}\rho_2$. The exponent of $x_{g;t}$ in (\ref{0421eq1}) is equal to

\[\aligned & \nu_2s_{n_2-2;2}-\rho_2s_{n_2-3;2}-(\nu_2A_{n_2-2;2}-\rho_2A_{n_2-3;2}) + \nu_1 s_{n_1-2;1} - \rho_1 s_{n_1-3;1}- (\nu_1 A_{n_1-2;1} - \rho_1 A_{n_1-3;1})\\
= \ & \nu_2\left(c_{n_2-1}^{[\xi_1]} (\varsigma -s_{n_1-3;1})+ \sum_{j=1}^{n_2-3} c_{n_2-1-j}^{[\xi_1]} \tau_{j;2}\right) \\
& -\rho_2 \left(c_{n_2-2}^{[\xi_1]} (\varsigma -s_{n_1-3;1})+ \sum_{j=1}^{n_2-4} c_{n_2-2-j}^{[\xi_1]}\tau_{j;2} \right)\\
& -\nu_2\left( c_{n_2-1}^{[\xi_1]}(A_{n_1-1;1}-r_1s_{n_1-2;1} ) + c_{n_2-2}^{[\xi_1]}(\omega_1 s_{n_1-2;1} - \xi_1 s_{n_1-3;1})\right)\\
& +\rho_2\left(  c_{n_2-2}^{[\xi_1]}(A_{n_1-1;1}-r_1s_{n_1-2;1} ) + c_{n_2-3}^{[\xi_1]}(\omega_1 s_{n_1-2;1} - \xi_1 s_{n_1-3;1}) \right)\\
& + \nu_1 s_{n_1-2;1} - \rho_1 s_{n_1-3;1}\\
& - (\nu_1 A_{n_1-2;1} - \rho_1 A_{n_1-3;1})\\
= \ & -\gamma_2s_{n_1-2;1} + (\nu_2 c_{n_2-1}^{[\xi_1]} -\rho_2 c_{n_2-2}^{[\xi_1]}) (\varsigma-A_{n_1-1;1}) + \nu_2\sum_{j=1}^{n_2-3} c_{n_2-1-j}^{[\xi_1]} \tau_{j;2} -\rho_2 \sum_{j=1}^{n_2-4} c_{n_2-2-j}^{[\xi_1]}\tau_{j;2}\\
& - (\nu_1 A_{n_1-2;1} - \rho_1 A_{n_1-3;1}). 
\endaligned\]

\subsubsection{Back to main computation} \label{sect back to main}
Using the computations from  \ref{sect exponents} and  fixing $\varsigma,\tau_{1;2},\cdots,\tau_{n_{2}-3;2}$ in (\ref{0421eq1}), we obtain  
$$\aligned  (\ref{0421eq1}) =
&\sum_{\tau_{0;1},\tau_{1;1},\cdots,\tau_{n_1-3;1}}   \left( \prod_{w=0}^{n_{1}-3}\gchoose{A_{w+1;1} - r_1s_{w;1}}{\tau_{w;1}}   \right)  \left( \prod_{w=0}^{n_{2}-3}\gchoose{A_{w+1;2} - r_{2}s_{w;2}}{\tau_{w;2}}   \right) \\
&\times  \widetilde{x_{d;t}}^ {c_{n_2}^{[\xi_1]}(A_{n_1-1;1}-r_1s_{n_1-2;1} ) + c_{n_2-1}^{[\xi_1]}\omega_1 s_{n_1-2;1}- \xi_1\left(c_{n_2-1}^{[\xi_1]} \varsigma+ \sum_{j=1}^{n_2-3} c_{n_2-1-j}^{[\xi_1]} \tau_{j;2}\right) }\\
&\times   x_{d';t}^{ \xi_1\left(c_{n_2-2}^{[\xi_1]} \varsigma + \sum_{j=1}^{n_2-4} c_{n_2-2-j}^{[\xi_1]}\tau_{j;2}\right) - \left( c_{n_2-1}^{[\xi_1]}(A_{n_1-1;1}-r_1s_{n_1-2;1} ) + c_{n_2-2}^{[\xi_1]}\omega_1 s_{n_1-2;1} \right)} \\
&\times   x_{e;t}^{-A_{n_1-2;1} +r_1\varsigma + \sum_{j=1}^{n_2-3} (c_{j+2}^{[\xi_1]}r_1 - c_{j+1}^{[\xi_1]}\omega_1)\tau_{j;2} }  \\
&\times   x_{g;t}^{ -\gamma_2s_{n_1-2;1} + (\nu_2 c_{n_2-1}^{[\xi_1]} -\rho_2 c_{n_2-2}^{[\xi_1]}) (\varsigma-A_{n_1-1;1}) + \nu_2\sum_{j=1}^{n_2-3} c_{n_2-1-j}^{[\xi_1]} \tau_{j;2} -\rho_2 \sum_{j=1}^{n_2-4} c_{n_2-2-j}^{[\xi_1]}\tau_{j;2}- (\nu_1 A_{n_1-2;1} - \rho_1 A_{n_1-3;1})}
 \endaligned$$
 Now we collect all powers involving $s_{n_1-2;1}$ and write  (\ref{0421eq1}) as a product $\phi \varphi$ where $\phi$ is a Laurent monomial in $ \widetilde{x_{d;t}}, x_{d';t},x_{e;t},x_{g;t}$ which in the expression of Proposition \ref{prop2.14} is absorbed either in the first summation, if the  exponent of $x_{e;t}$ is positive, or in the second summation inside $x_{e;t'}^{-\theta}\mathfrak{r}$ if the exponent of $x_{e;t}$ is negative. On the other hand, $\varphi$ is equal to 
\begin{equation*}\aligned
\sum_{\tau_{0;1},\tau_{1;1},\cdots,\tau_{n_1-3;1}}   \left( \prod_{w=0}^{n_{1}-3}\gchoose{A_{w+1;1} - r_1s_{w;1}}{\tau_{w;1}}   \right)   \left( \prod_{w=0}^{n_{2}-3}\gchoose{A_{w+1;2} - r_{2}s_{w;2}}{\tau_{w;2}}   \right) \\
\times \left(\frac{\widetilde{x_{d;t}}^{c_{n_2}^{[\xi_1]}r_1- c_{n_2-1}^{[\xi_1]}\omega_1} x_{g;t}^{\gamma_2}  }{ x_{d';t}^{c_{n_2-1}^{[\xi_1]}r_1- c_{n_2-2}^{[\xi_1]}\omega_1}}\right)^{ \left\lfloor (A_{n_1-1;1}-\varsigma) \frac{A_{n_1-1;1}}{A_{n_1;1}}\right\rfloor -s_{n_1-2;1}}.
\endaligned
\end{equation*}
Note that the 0-th term of the second product can be identified with an $(n_1-2)$-nd term in the first product, since $A_{1;2}= A_{n_1-1;1}-r_1s_{n_1-2;1}$ by the definition of $A_{w;2}$ right after equation (\ref{0303eq5d}) and $s_{0,2}=0$. The above expression is therefore equal to 
\begin{equation}\label{eq goal2}\aligned
\sum_{\tau_{0;1},\tau_{1;1},\cdots,\tau_{n_1-3;1}}   \left( \prod_{w=0}^{n_{1}-2}\gchoose{A_{w+1;1} - r_1s_{w;1}}{\tau_{w;1}}   \right)   \left( \prod_{w=1}^{n_{2}-3}{A_{w+1;2} - r_{2}s_{w;2}\choose\tau_{w;2}}   \right) 
\\ 
  \times \left(\frac{\widetilde{x_{d;t}}^{c_{n_2}^{[\xi_1]}r_1- c_{n_2-1}^{[\xi_1]}\omega_1} x_{g;t}^{\gamma_2} }{ x_{d';t}^{c_{n_2-1}^{[\xi_1]}r_1- c_{n_2-2}^{[\xi_1]}\omega_1}}\right)^{ \left\lfloor (A_{n_1-1;1}-\varsigma) \frac{A_{n_1-1;1}}{A_{n_1;1}}\right\rfloor -s_{n_1-2;1}},
 \endaligned
 \end{equation}
where $\tau_{n_1-2;1}=A_{n_1-1;1}-r_1s_{n_1-2;1}-\tau_{0;2}=A_{n_1-1;1}-r_1s_{n_1-2;1}-\varsigma +s_{n_1-3;1}$. The exponent ${ \left\lfloor (A_{n_1-1;1}-\varsigma) \frac{A_{n_1-1;1}}{A_{n_1;1}}\right\rfloor -s_{n_1-2;1}}$
is non-negative by Lemma~\ref{0302lem1} below. 

Now we show that the Laurent monomial $\phi$ has non-negative degree on $\widetilde{x_{d;t}}$. 
We want to show that
\begin{equation}\label{expphi}
c^{[\xi_1]}_{n_2}A_{n_1-1;1} - \xi_1 \left( c^{[\xi_1]}_{n_2-1} \varsigma  + \sum_{j=1}^{n_2-3} c^{[\xi_1]}_{n_2-1-j}\tau_{j;2} \right)  - (c^{[\xi_1]}_{n_2} r_1 - c^{[\xi_1]}_{n_2-1} \omega_1)(A_{n_1-1;1} -  \varsigma)\frac{A_{n_1-1;1}}{A_{n_1;1}}\ge0.\end{equation}

The   inequality (\ref{eq page 33}) yields
$$
(c^{[\xi_1]}_{j+2} r_1 - c^{[\xi_1]}_{j+1} \omega_1)\frac{A_{n_1-1;1}}{A_{n_1-2;1}}\left( c^{[\xi_1]}_{n_2} - (c^{[\xi_1]}_{n_2} r_1 - c^{[\xi_1]}_{n_2-1} \omega_1)\frac{A_{n_1-1;1}}{A_{n_1;1}}  \right)\tau_{j;2}
> \xi_1 c^{[\xi_1]}_{n_2-1-j}\tau_{j;2}
$$
so 
$$
\sum_{j=1}^{n_2-3}    (c^{[\xi_1]}_{j+2} r_1 - c^{[\xi_1]}_{j+1} \omega_1)\frac{A_{n_1-1;1}}{A_{n_1-2;1}}\left( c^{[\xi_1]}_{n_2} - (c^{[\xi_1]}_{n_2} r_1 - c^{[\xi_1]}_{n_2-1} \omega_1)\frac{A_{n_1-1;1}}{A_{n_1;1}}  \right)\tau_{j;2}
> \sum_{j=1}^{n_2-3}     \xi_1 c^{[\xi_1]}_{n_2-1-j}\tau_{j;2}.
$$
Thanks to $\theta>0$ and   equation (\ref{theta}), this implies that
\begin{equation}\label{20130523eq5}
 (A_{n_1-2;1} - r_1 \varsigma)\frac{A_{n_1-1;1}}{A_{n_1-2;1}}\left( c^{[\xi_1]}_{n_2} - (c^{[\xi_1]}_{n_2} r_1 - c^{[\xi_1]}_{n_2-1} \omega_1)\frac{A_{n_1-1;1}}{A_{n_1;1}}  \right)
> \sum_j     \xi_1 c^{[\xi_1]}_{n_2-1-j}\tau_{j;2}.
\end{equation}
Then to prove (\ref{expphi}) it is enough to show that 
$$
c^{[\xi_1]}_{n_2}A_{n_1-1;1} - \xi_1 c^{[\xi_1]}_{n_2-1} \varsigma - (c^{[\xi_1]}_{n_2} r_1 - c^{[\xi_1]}_{n_2-1} \omega_1)(A_{n_1-1;1} -  \varsigma)\frac{A_{n_1-1;1}}{A_{n_1;1}}
$$ is greater than the left hand side of \eqref{20130523eq5}. The terms without $\varsigma$ are all canceled. Hence it remains to show that
$$
r_1 \frac{A_{n_1-1;1}}{A_{n_1-2;1}}\left( c^{[\xi_1]}_{n_2} - (c^{[\xi_1]}_{n_2} r_1 - c^{[\xi_1]}_{n_2-1} \omega_1)\frac{A_{n_1-1;1}}{A_{n_1;1}}  \right)
>
 \xi_1 c^{[\xi_1]}_{n_2-1}  - (c^{[\xi_1]}_{n_2} r_1 - c^{[\xi_1]}_{n_2-1} \omega_1) \frac{A_{n_1-1;1}}{A_{n_1;1}}
$$  
$$
\Longleftrightarrow \ \ \ \ \ \ 
r_1 c^{[\xi_1]}_{n_2}  \frac{A_{n_1-1;1}}{A_{n_1-2;1}} -  \xi_1 c^{[\xi_1]}_{n_2-1} > (c^{[\xi_1]}_{n_2} r_1 - c^{[\xi_1]}_{n_2-1} \omega_1)\left(  r_1 \frac{A_{n_1-1;1}}{A_{n_1-2;1}} - 1    \right)\frac{A_{n_1-1;1}}{A_{n_1;1}}
$$
$$
\Longleftrightarrow \ \ \ \ \ \ 
r_1 c^{[\xi_1]}_{n_2}  \frac{A_{n_1-1;1}}{A_{n_1-2;1}} -  \xi_1 c^{[\xi_1]}_{n_2-1} > (c^{[\xi_1]}_{n_2} r_1 - c^{[\xi_1]}_{n_2-1} \omega_1)\left(   \frac{r_1A_{n_1-1;1} - A_{n_1-2;1}}{A_{n_1-2;1}}     \right)\frac{A_{n_1-1;1}}{A_{n_1;1}}
$$
$$
\stackrel{\textup{Lemma}~\ref{negone}}{\Longleftrightarrow} \ \ \ \ \ \ 
r_1 c^{[\xi_1]}_{n_2}  \frac{A_{n_1-1;1}}{A_{n_1-2;1}} -  \xi_1 c^{[\xi_1]}_{n_2-1} > (c^{[\xi_1]}_{n_2} r_1 - c^{[\xi_1]}_{n_2-1} \omega_1)\frac{A_{n_1-1;1}}{A_{n_1-2;1}}
$$
$$
\Longleftrightarrow \ \ \ \ \ \ 
\omega_1 \frac{A_{n_1-1;1}}{A_{n_1-2;1}} -  \xi_1 > 0.$$The last inequality follows from the first case of Lemma \ref{lem nonacyclic}. This shows (\ref{expphi}) and thus the exponent of $\widetilde{x_{d;t}}$ in $\phi$ is non-negative.

Now Proposition~\ref{prop2.14} follows from  the following Lemma.

 \begin{lemma}\label{0303lem945}
$$
 \prod_{w=1}^{n_{2}-3}{{A_{w+1;2} - r_{2}s_{w;2}}\choose{\tau_{w;2}}}
=
\sum_{i=0}^{\sum_{w=1}^{n_2-3}\tau_{w;2} } d_i {\left\lfloor (A_{n_1-1;1}-\varsigma) \frac{A_{n_1-1;1}}{A_{n_1;1}}\right\rfloor-s_{n_1-2;1}\choose i}
$$ for some $d_i\in \mathbb{N}$, which are independent of $s_{n_1-2;1}$.
\end{lemma}
\begin{proof}
First suppose that one of $\xi_1, r_1, \omega_1 $ is at most one. The equation in Lemma 5.8 without the requirement that $d_i\in\mathbb{N}$  is always true in any case (a),(b),(c),(d).       
     The divisibility in Lemma 5.14 without the positivity statement is always true, so the argument in subsection \ref{sect 5.3} shows that Proposition 5.5 without the positivity statement is always true.
     Therefore $x_{d;t_{\omega_E}}^p x_{e;t_{\omega_E}}^q$ is a linear combination of the following rank 2 cluster monomials:
$
\widetilde {x_{d;t}}^{p'}  \widetilde{\widetilde {x_{e;t}}}^{q'}
,
\widetilde {x_{d;t}}^{p'}  {{x_{e;t}}}^{q'}
,
{x_{d;t}}^{p'}    {{x_{e;t}}}^{q'}
,
{x_{d;t}}^{p'}   \widetilde{x_{e;t}}^{q'}.
$
Now if one of $\xi_1, r_1, \omega_1 $ is at most one, then some quiver which is mutation equivalent to the subquiver with vertices $d, d', e$ is acyclic \cite{LS3}, and
it follows from \cite[Equation (5.10)]{SZ} that the coefficients of the linear combination are non-negative. 
This implies that $d_i\in\mathbb{N}$.
  
 In what follows, we assume that $\xi_1, r_1, \omega_1 \ge 2.$

Once we know that there are nonnegative integers $a$ and $b$ such that $$A_{w+1;2}-r_{2} s_{w;2}=a\left( \left\lfloor (A_{n_1-1;1}-\varsigma) \frac{A_{n_1-1;1}}{A_{n_1;1}}\right\rfloor-s_{n_1-2;1}\right) +b,$$ then it is clear, by Lemma~\ref{0423lem1} below, that
$$
{{A_{w+1;2} - r_{2}s_{w;2}}\choose{\tau_{w;2}}} =
\sum_{i=0}^{\tau_{w;2}} d_i' {\left\lfloor (A_{n_1-1;1}-\varsigma) \frac{A_{n_1-1;1}}{A_{n_1;1}}\right\rfloor-s_{n_1-2;1}\choose i}
$$ for some $d_i'\in \mathbb{N}$, and by Lemma~\ref{0423lem2} below, for any nonnegative integers $j$ and $k$,
$$\aligned
&{\left\lfloor (A_{n_1-1;1}-\varsigma) \frac{A_{n_1-1;1}}{A_{n_1;1}}\right\rfloor-s_{n_1-2;1}\choose j}
{\left\lfloor (A_{n_1-1;1}-\varsigma) \frac{A_{n_1-1;1}}{A_{n_1;1}}\right\rfloor-s_{n_1-2;1}\choose k}\\
&=\sum_{i=0}^{j+k} d_i'' {\left\lfloor (A_{n_1-1;1}-\varsigma) \frac{A_{n_1-1;1}}{A_{n_1;1}}\right\rfloor-s_{n_1-2;1}\choose i}
\endaligned$$ for some $d_i''\in \mathbb{N}$.
Then it follows that 
$$
 \prod_{w=1}^{n_{2}-3}{{A_{w+1;2} - r_{2}s_{w;2}}\choose{\tau_{w;2}}}
=
\sum_{i=0}^{\sum_{w=1}^{n_2-3}\tau_{w;2} } d_i {\left\lfloor (A_{n_1-1;1}-\varsigma) \frac{A_{n_1-1;1}}{A_{n_1;1}}\right\rfloor-s_{n_1-2;1}\choose i}.
$$

Thus we need to show the existence of the nonnegative integers $a$ and $b$.
Using the definitions of $A_{w+1;2}$ and $\varsigma$  as well as  the fact that $r_2=\xi_1$, we get
$$\aligned
A_{w+1;2}-r_{2} s_{w;2}
 = c_{w+2}^{[\xi_1]}(A_{n_1-1;1}-r_1s_{n_1-2;1} ) + c_{w+1}^{[\xi_1]}\omega_1 s_{n_1-2;1}- \xi_1\left(c_{w+1}^{[\xi_1]} \varsigma+ \sum_{j=1}^{w-1} c_{w+1-j}^{[\xi_1]} \tau_{j;2}\right)\endaligned$$ which can be written as
\begin{equation}\label{eq 45}
 A_{w+1;2}-r_{2} s_{w;2}
 = (c_{w+2}^{[\xi_1]}r_1-c_{w+1}^{[\xi_1]}\omega_1)\left(\left\lfloor (A_{n_1-1;1}-\varsigma) \frac{A_{n_1-1;1}}{A_{n_1;1}}\right\rfloor-s_{n_1-2;1} \right)+C(w),
 \end{equation}
 where $C(w)$ is some function of $w$, which is independent of $s_{n_1-2;1}$ and which we give explicitly below. Note that $$c_{w+2}^{[\xi_1]}r_1-c_{w+1}^{[\xi_1]}\omega_1 >0,$$ because, by Lemma \ref{lem nonacyclic}, this is the number of arrows between some pair of vertices in some seed between $t_{\omega_Z}$ and $t$. Thus is suffices to show that $C(w)$ is nonnegative.

$$\aligned
C(w)=&(c_{w+2}^{[\xi_1]}r_1-c_{w+1}^{[\xi_1]}\omega_1)\left( (A_{n_1-1;1}-\varsigma) \frac{A_{n_1-1;1}}{A_{n_1;1}}-\left\lfloor (A_{n_1-1;1}-\varsigma) \frac{A_{n_1-1;1}}{A_{n_1;1}}\right\rfloor\right)+ \tilde{C}(w) \theta(w), \endaligned$$
where
$$\aligned
 \tilde{C}(w)&=c_{w+2}^{[\xi_1]} - (c_{w+2}^{[\xi_1]}r_1-c_{w+1}^{[\xi_1]}\omega_1)\frac{A_{n_1-1;1}}{A_{n_1;1}} \text{ and}  \\
\theta(w) &= A_{n_1-1;1} - \frac{\xi_1c_{w+1}^{[\xi_1]} - (c_{w+2}^{[\xi_1]}r_1-c_{w+1}^{[\xi_1]}\omega_1)\frac{A_{n_1-1;1}}{A_{n_1;1}}}{c_{w+2}^{[\xi_1]} - (c_{w+2}^{[\xi_1]}r_1-c_{w+1}^{[\xi_1]}\omega_1)\frac{A_{n_1-1;1}}{A_{n_1;1}}}\varsigma- \sum_{j=1}^{w-1}  \frac{\xi_1c_{w+1-j}^{[\xi_1]}}{c_{w+2}^{[\xi_1]} - (c_{w+2}^{[\xi_1]}r_1-c_{w+1}^{[\xi_1]}\omega_1)\frac{A_{n_1-1;1}}{A_{n_1;1}}} \tau_{j;2}.
  \endaligned$$
We want to show that $C(w)$ is nonnegative for $1\le w\le n_2-3$, for which it suffices to show that $\tilde{C}(w)$ and $\theta(w)$ are nonnegative.

First we show that $\tilde{C}(w)$ are nonnegative for $w\geq 1$. Note that $\tilde{C}(w)=\xi_1\tilde{C}(w-1)-\tilde{C}(w-2)$. Then
if we show $\tilde{C}(1)>0\geq \tilde{C}(0)$ then the induction on $w$ will show that $\tilde{C}(w)$ is increasing with $w$. By Lemma \ref{negone}, we have $\tilde{C}(0)=1-r_1\frac{A_{n_1-1;1}}{A_{n_1;1}} \leq 0$. On the other hand,
$$\begin{array}{rcccl}\tilde{C}(1)&=&
\xi_1-(\xi_1r_1-\omega_1)\frac{A_{n_1-1;1}}{A_{n_1;1}}&=&
\xi_1(\frac{{A_{n_1;1}}-r_1A_{n_1-1;1}}{A_{n_1;1}}) + \omega_1\frac{A_{n_1-1;1}}{A_{n_1;1}}\\
&=&
\xi_1(\frac{-{A_{n_1-2;1}}}{A_{n_1;1}}) + \omega_1\frac{A_{n_1-1;1}}{A_{n_1;1}}\end{array} $$
which is positive because of equation~(\ref{eq 44.1}).

Next we show that $\theta(w)$ are nonnegative for all $w$ such that $1\le w\le n_2-3$.   Recall from (\ref{theta}) that
\begin{equation*}\label{02292012eq2}\theta=A_{n_1-2;1} -r_1\varsigma - \sum_{j=1}^{n_2-3} (c_{j+2}^{[\xi_1]}r_1 - c_{j+1}^{[\xi_1]}\omega_1)\tau_{j;2}  >0.\end{equation*}
Multiplying with $\frac{A_{n_1-1;1} }{A_{n_1-2;1} }$ yields
\begin{equation*}\label{04222012eq5}A_{n_1-1;1} -\frac{r_1 A_{n_1-1;1} }{A_{n_1-2;1} }\varsigma - \sum_{j=1}^{w-1} \frac{(c_{j+2}^{[\xi_1]}r_1 - c_{j+1}^{[\xi_1]}\omega_1) A_{n_1-1;1} }{A_{n_1-2;1} }\tau_{j;2}  >0.\end{equation*}
So it is enough to show that  if $\xi_1, r_1, \omega_1 \ge 2$
then

\begin{equation}\label{eq page 33 bis}\frac{r_1 A_{n_1-1;1} }{A_{n_1-2;1} } >  \frac{\xi_1c_{w+1}^{[\xi_1]} - (c_{w+2}^{[\xi_1]}r_1-c_{w+1}^{[\xi_1]}\omega_1)\frac{A_{n_1-1;1}}{A_{n_1;1}}}{c_{w+2}^{[\xi_1]} - (c_{w+2}^{[\xi_1]}r_1-c_{w+1}^{[\xi_1]}\omega_1)\frac{A_{n_1-1;1}}{A_{n_1;1}}}\end{equation} and 
\begin{equation}\label{eq page 33}\frac{(c_{j+2}^{[\xi_1]}r_1 - c_{j+1}^{[\xi_1]}\omega_1) A_{n_1-1;1} }{A_{n_1-2;1} } > \frac{\xi_1c_{w+1-j}^{[\xi_1]}}{c_{w+2}^{[\xi_1]} - (c_{w+2}^{[\xi_1]}r_1-c_{w+1}^{[\xi_1]}\omega_1)\frac{A_{n_1-1;1}}{A_{n_1;1}}},\end{equation}


 This ends the proof of Lemma~\ref{0303lem945}, modulo the inequalities (\ref{eq page 33 bis}) and (\ref{eq page 33}), which are proved in the following subsection.
\end{proof}

\subsubsection{Proof of (\ref{eq page 33 bis})  and (\ref{eq page 33}).
} 
It follows from Lemma \ref{negone} that the left hand side of (\ref{eq page 33 bis}) is equal to 
$({A_{n_1-2;1} + {A_{n_1;1}}}) / {{A_{n_1-2;1}}} = 1 + {A_{n_1;1}}/ {{A_{n_1-2;1}}}  $.
Thus (\ref{eq page 33 bis}) is equivalent to
 $$1+\frac{A_{n_1;1}}{A_{n_1-2;1}} > \frac{\xi_1 c_{w+1}^{[\xi_1]} A_{n_1;1}- (c_{w+2}^{[\xi_1]} r_1 - c_{w+1}^{[\xi_1]} \omega_1)A_{n_1-1;1} }{c_{w+2}^{[\xi_1]} A_{n_1;1}- (c_{w+2}^{[\xi_1]} r_1 - c_{w+1}^{[\xi_1]} \omega_1)A_{n_1-1;1} }$$
\medskip
$$ \Longleftrightarrow\ \ \
1+\frac{A_{n_1;1}}{A_{n_1-2;1}} > 1+\frac{(\xi_1 c_{w+1}^{[\xi_1]} -c_{w+2}^{[\xi_1]})A_{n_1;1} }{c_{w+2}^{[\xi_1]} A_{n_1;1}- (c_{w+2}^{[\xi_1]} r_1 - c_{w+1}^{[\xi_1]} \omega_1)A_{n_1-1;1} }$$
\medskip
$$\Longleftrightarrow\ \ \
\frac{1}{A_{n_1-2;1}} > \frac{\xi_1 c_{w+1}^{[\xi_1]} -c_{w+2}^{[\xi_1]} }{c_{w+2}^{[\xi_1]} A_{n_1;1}- (c_{w+2}^{[\xi_1]} r_1 - c_{w+1}^{[\xi_1]} \omega_1)A_{n_1-1;1} }$$
\medskip
$$\overset{\text{by recursive definition of }c_{w}^{[\xi_1]}}\Longleftrightarrow\ \ \
\frac{1}{A_{n_1-2;1}} > \frac{c_{w}^{[\xi_1]} }{c_{w+2}^{[\xi_1]} A_{n_1;1}- (c_{w+2}^{[\xi_1]} r_1 - c_{w+1}^{[\xi_1]} \omega_1)A_{n_1-1;1} }$$
\medskip
\begin{equation}\label{10192013eq1}\overset{\text{Lemma }\ref{negone}}\Longleftrightarrow\ \ \
\frac{1}{A_{n_1-2;1}} > \frac{c_{w}^{[\xi_1]} }{c_{w+1}^{[\xi_1]} \omega_1A_{n_1-1;1}  -c_{w+2}^{[\xi_1]} A_{n_1-2;1}}\end{equation}
\medskip
$$\Longleftrightarrow\ \ \
c_{w+1}^{[\xi_1]} \omega_1A_{n_1-1;1}  -c_{w+2}^{[\xi_1]} A_{n_1-2;1}   > c_{w}^{[\xi_1]}A_{n_1-2;1}$$
\medskip
$$\Longleftrightarrow\ \ \
c_{w+1}^{[\xi_1]} \omega_1A_{n_1-1;1}     - A_{n_1-2;1} (c_{w}^{[\xi_1]}+ c_{w+2}^{[\xi_1]}) >0$$
\medskip
$$\overset{\text{by recursive definition of }c_{w}^{[\xi_1]}}\Longleftrightarrow\ \ \
c_{w+1}^{[\xi_1]} (\omega_1A_{n_1-1;1} -  \xi_1 A_{n_1-2;1} )  > 0.$$
This last inequality holds by equation \ref{eq 44.1}, and thus this completes the prove of (\ref{eq page 33 bis}).

To prove the inequality (\ref{eq page 33}) we start with 
the inequality~\eqref{10192013eq1} above. 
Suppose first that
$${c_{w}^{[\xi_1]}}\ge \frac{\xi_1 c_{w+1-j}^{[\xi_1]} A_{n_1;1}}{(c_{j+2}^{[\xi_1]} r_1 - c_{j+1}^{[\xi_1]} \omega_1)A_{n_1-1;1}} $$
Then \eqref{10192013eq1}  implies
$$
\frac{(c_{j+2}^{[\xi_1]} r_1 - c_{j+1}^{[\xi_1]} \omega_1)A_{n_1-1;1}}{A_{n_1-2;1}} > \frac{\xi_1 c_{w+1-j}^{[\xi_1]} A_{n_1;1} }{c_{w+1}^{[\xi_1]} \omega_1A_{n_1-1;1}-c_{w+2}^{[\xi_1]}A_{n_1-2;1}  }$$
which is equivalent to the inequality (\ref{eq page 33}), because $A_{n_1-2;1}=rA_{n_1-1;1}-A_{n_1;1}$, by  Lemma~\ref{negone}.

Suppose to the contrary that 

$${c_{w}^{[\xi_1]}}< \frac{\xi_1 c_{w+1-j}^{[\xi_1]} A_{n_1;1}}{(c_{j+2}^{[\xi_1]} r_1 - c_{j+1}^{[\xi_1]} \omega_1)A_{n_1-1;1}}. $$
Then
$$c_{j+2}^{[\xi_1]} r_1 - c_{j+1}^{[\xi_1]} \omega_1 < \frac{\xi_1 c_{w+1-j}^{[\xi_1]} A_{n_1;1}}{c_{w}^{[\xi_1]}A_{n_1-1;1}} $$
and dividing by ${c_{j+1}^{[\xi_1]}}$ and using $ A_{n_1;1}=r_1A_{n_1-1;1}-A_{n_1-2;1}$ yields
\begin{equation}\label{case j>1 eq1}
\frac{c_{j+2}^{[\xi_1]}}{c_{j+1}^{[\xi_1]}} r_1 - \omega_1 <  \frac{\xi_1 c_{w+1-j}^{[\xi_1]} (r_1A_{n_1-1;1}-A_{n_1-2;1})}{c_{j+1}^{[\xi_1]} c_{w}^{[\xi_1]}A_{n_1-1;1}}  \le \frac{\xi_1 c_{w+1-j}^{[\xi_1]} r_1}{c_{j+1}^{[\xi_1]} c_{w}^{[\xi_1]}}.\end{equation}

\noindent\textbf{Case 1: } Suppose that $j\geq 2$.
Recall that $c_1^{[\xi_1]}=0, c_2^{[\xi_1]}=1, c_3^{[\xi_1]}=\xi_1,$ and then  (\ref{case j>1 eq1}) implies
$$\frac{c_{j+2}^{[\xi_1]}}{c_{j+1}^{[\xi_1]}} r_1 - \omega_1 <  \frac{c_{w-1}^{[\xi_1]}r_1}{c_{w}^{[\xi_1]}},$$
so
\begin{equation*}
\omega_1>\left(\frac{c_{j+2}^{[\xi_1]}}{c_{j+1}^{[\xi_1]}}- \frac{c_{w-1}^{[\xi_1]}}{c_{w}^{[\xi_1]}}\right)r_1>\left(\lim_{j\rightarrow\infty}\frac{c_{j+2}^{[\xi_1]}}{c_{j+1}^{[\xi_1]}}- \lim_{w\rightarrow\infty} \frac{c_{w-1}^{[\xi_1]}}{c_{w}^{[\xi_1]}}\right)r_1,
\end{equation*}
where the last inequality holds since  $\frac{c_{j+2}^{[\xi_1]}}{c_{j+1}^{[\xi_1]}}<\frac{c_{j+3}^{[\xi_1]}}{c_{j+2}^{[\xi_1]}} $, for all $j\ge 1$.
Computing the limits, we obtain

%
 
\begin{equation}\label{10192013eq2}
\omega_1>\left(\frac{\xi_1+\sqrt{\xi_1^2-4}}{2}-\frac{\xi_1-\sqrt{\xi_1^2-4}}{2}\right)r_1=r_1\sqrt{\xi_1^2-4}.
\end{equation}

\begin{itemize}
\item 
If $\xi_1\geq 3$ then \eqref{10192013eq2} implies $\omega_1>r_1(\xi_1-1)$, 
thus $\omega_1\ge 5$ and $\omega_1-\xi_1>\xi_1r_1 -r_1-\xi_1 \geq 1$. 
%

\item If $\xi_1=2$ and $\omega_1\geq 4$ then, since $0<{A_{n_1-2;1}}/{A_{n_1-1;1}} <1$, we  have $$\omega_1-\xi_1\frac{A_{n_1-2;1}}{A_{n_1-1;1}}\geq 2.$$

\item If $\xi_1=2$, $\omega_1=3$ and $r_1\geq 3$  then we still have $$\omega_1-\xi_1\frac{A_{n_1-2;1}}{A_{n_1-1;1}}\geq 2.$$

\item If $\xi_1=2$, $\omega_1=3$ and $r_1=2$  then the subquiver obtained from $t^*$ by mutating at $d'$ is acyclic, so we do not consider this case.

\item If $\xi_1=\omega_1=2$ and $r_1\geq 3$  then $c_{j+2}^{[\xi_1]} r_1 - c_{j+1}^{[\xi_1]} \omega_1\geq 2$ and $\omega_1-\xi_1\frac{A_{n_1-2;1}}{A_{n_1-1;1}}\geq 1$.

\end{itemize}

In any of the above cases we have
\begin{equation}\label{same argument}
(c_{j+2}^{[\xi_1]} r_1 - c_{j+1}^{[\xi_1]} \omega_1)( \omega_1-\xi_1\frac{A_{n_1-2;1}}{A_{n_1-1;1}})\geq  2.\end{equation}
\medskip
$$\Longrightarrow\ \ \
(c_{j+2}^{[\xi_1]} r_1 - c_{j+1}^{[\xi_1]} \omega_1)( \omega_1-\xi_1\frac{A_{n_1-2;1}}{A_{n_1-1;1}})\geq  \frac{\xi_1 c_{w+1-j}^{[\xi_1]}}{c_{w+1}^{[\xi_1]}}$$
\medskip
$$\Longrightarrow\ \ \
(c_{j+2}^{[\xi_1]} r_1 - c_{j+1}^{[\xi_1]} \omega_1)( \omega_1A_{n_1-1;1}-\xi_1A_{n_1-2;1})\geq  \frac{\xi_1 c_{w+1-j}^{[\xi_1]} A_{n_1-1;1} }{c_{w+1}^{[\xi_1]}}.$$
Since  $c_{w}^{[\xi_1] }=\xi_1c_{w-1}^{[\xi_1]}-c_{w-2}^{[\xi_1]  }$, and thus $\xi_1c_{w-1}^{[\xi_1]}>c_{w-2}^{[\xi_1]  }$, it follows that
$$
(c_{j+2}^{[\xi_1]} r_1 - c_{j+1}^{[\xi_1]} \omega_1)(c_{w+1}^{[\xi_1]} \omega_1A_{n_1-1;1}-c_{w+2}^{[\xi_1]}A_{n_1-2;1})\geq  {\xi_1 c_{w+1-j}^{[\xi_1]} A_{n_1-1;1} }$$
\medskip
$$\Longleftrightarrow\ \ \
(c_{j+2}^{[\xi_1]} r_1 - c_{j+1}^{[\xi_1]} \omega_1)(c_{w+1}^{[\xi_1]} \omega_1A_{n_1-1;1}-c_{w+2}^{[\xi_1]}A_{n_1-2;1})A_{n_1-1;1}\geq  {\xi_1 c_{w+1-j}^{[\xi_1]} A_{n_1-1;1}^2 }$$
\medskip
$$\overset{\text{Lemma }\ref{negone}}\Longrightarrow\ \ \
(c_{j+2}^{[\xi_1]} r_1 - c_{j+1}^{[\xi_1]} \omega_1)(c_{w+1}^{[\xi_1]} \omega_1A_{n_1-1;1}-c_{w+2}^{[\xi_1]}A_{n_1-2;1})A_{n_1-1;1}> {\xi_1 c_{w+1-j}^{[\xi_1]} A_{n_1;1} }{A_{n_1-2;1}}$$
\medskip
$$\Longleftrightarrow\ \ \
\frac{(c_{j+2}^{[\xi_1]} r_1 - c_{j+1}^{[\xi_1]} \omega_1)A_{n_1-1;1}}{A_{n_1-2;1}} > \frac{\xi_1 c_{w+1-j}^{[\xi_1]} A_{n_1;1} }{c_{w+1}^{[\xi_1]} \omega_1A_{n_1-1;1}-c_{w+2}^{[\xi_1]}A_{n_1-2;1}  }$$
\medskip
$$\overset{\text{Lemma }\ref{negone}}\Longleftrightarrow\ \ \
\frac{(c_{j+2}^{[\xi_1]} r_1 - c_{j+1}^{[\xi_1]} \omega_1)A_{n_1-1;1}}{A_{n_1-2;1}} > \frac{\xi_1 c_{w+1-j}^{[\xi_1]} A_{n_1;1} }{c_{w+2}^{[\xi_1]} A_{n_1;1}- (c_{w+2}^{[\xi_1]} r_1 - c_{w+1}^{[\xi_1]} \omega_1)A_{n_1-1;1} }$$
which proves the inequality  (\ref{eq page 33}) in these cases.

Next suppose that $\xi_1=r_1=\omega_1=2$. In this case, $c_{j+1}^{[\xi_1]} =j$, and thus $(c_{j+2}^{[\xi_1]} r_1 - c_{j+1}^{[\xi_1]} \omega_1)=2$ and $(c_{w+2}^{[\xi_1]} r_1 - c_{w+1}^{[\xi_1]} \omega_1)=2$. It therefore suffices  to show that 
$$
\frac{2A_{n_1-1;1}}{A_{n_1-2;1}} > \frac{2 (w-j) A_{n_1;1} }{(w+1) A_{n_1;1}- 2A_{n_1-1;1} },$$
but this is equivalent to
$$(w+1)A_{n_1;1}A_{n_1-1;1} > (w-j)A_{n_1;1}A_{n_1-2;1} + 2A_{n_1-1;1}A_{n_1-1;1},$$
which holds true for $j\geq 1$, by Lemma \ref{negone}. 
This completes the proof in the case $j\ge 2$.

\noindent\textbf{Case 2: } Suppose that $j=1$. Let $\widehat{\xi_1}=\lim_{w\rightarrow\infty}\frac{c_{w+1}^{[\xi_1]}}{c_{w}^{[\xi_1]}}$, in other words, $\widehat{\xi_1}=\frac{\xi_1+\sqrt{\xi_1^2-4}}{2}.$

\noindent\textbf{Case 2-1: }  Suppose that 
\begin{equation}\label{10242013eq1} \frac{A_{n_1-1;1}}{A_{n_1-2;1}}\omega_1 -  \frac{(\xi_1 r_1 - \omega_1)\widehat{\xi_1}(\omega_1A_{n_1-1;1}-\xi_1A_{n_1-2;1})A_{n_1-1;1}}{A_{n_1-2;1}^2}\geq \xi_1.\end{equation}

If $r_1=2$ then one can check (\ref{eq page 33}) directly. Suppose $r_1\geq 3$. Then  $\omega_1 c_{0}^{[r_1]}-\xi_1 c_{-1}^{[r_1]} = - w_1+\xi_1 r_1$ is the number of arrows between $e$ and $d'$ in the seed $t^*$, and 
$ \omega_1 c_{n-1}^{[r_1]}-\xi_1 c_{n-2}^{[r_1]}$  is the number of arrows between some pair of vertices in the seed $t_{w_{E}}.$ In particular, these numbers are positive.
Then by Lemma~\ref{10242013lem2},  we have  $(\omega_1 c_{0}^{[r_1]}-\xi_1 c_{-1}^{[r_1]})(\omega_1 c_{n-1}^{[r_1]}-\xi_1 c_{n-2}^{[r_1]})>c_{n-2}^{[r_1]}$ for $n\in\{n_1-2, n_1-1\}$. 
So $$(\xi_1 r_1 - \omega_1)(\omega_1 c_{n-1}^{[r_1]}-\xi_1 c_{n-2}^{[r_1]})>c_{n-2}^{[r_1]}$$ hence
$$(\xi_1 r_1 - \omega_1)(\omega_1 A_{n_1-1;1}-\xi_1 A_{n_1-2;1})>A_{n_1-2;1},$$
because $A_{i;1}$ is a positive linear combination of $c_i^{[r_1]}$ and $c_{i+1}^{[r_1]}.$ 

Then \eqref{10242013eq1} implies that
$$
\left(\omega_1- \widehat{\xi_1}\right) \frac{A_{n_1-1;1}}{A_{n_1-2;1}}\geq  \xi_1$$
$$\Longrightarrow\ \ \
\left(\omega_1-\frac{2}{\xi_1 r_1 - \omega_1}\right) \frac{A_{n_1-1;1}}{A_{n_1-2;1}}\geq  \xi_1$$
\medskip$$\Longleftrightarrow\ \ \
\omega_1-\frac{2}{\xi_1 r_1 - \omega_1} \geq  \xi_1\frac{A_{n_1-2;1}}{A_{n_1-1;1}}$$
\medskip$$\Longleftrightarrow\ \ \
\omega_1-\xi_1\frac{A_{n_1-2;1}}{A_{n_1-1;1}}\geq  \frac{2}{\xi_1 r_1 - \omega_1}$$
\medskip$$\Longleftrightarrow\ \ \
(\xi_1 r_1 - \omega_1)( \omega_1-\xi_1\frac{A_{n_1-2;1}}{A_{n_1-1;1}})\geq  2,
$$ and we can use the argument following (\ref{same argument}).

\begin{lemma}\label{10242013lem2}
Fix a positive integer $r\geq 3$ and let $c_i=c_i^{[r]}$. Let $\omega, \xi, a, b$ be integers such that $b-a$, $\omega c_{a}-\xi c_{a-1}$ and $\omega c_{b}-\xi c_{b-1}$ are positive. Then $(\omega c_{a}-\xi c_{a-1})(\omega c_{b}-\xi c_{b-1})> c_{b-a-1}$.
\end{lemma}
\begin{proof}
Suppose that $\omega^2 - r\omega\xi+\xi^2 <0$. If $b-a\leq 2$  then $c_{b-a-1}\leq 0$, so there is nothing to show.
Thanks to Lemma \ref{lem cn}, we get $$(\omega c_{a}-\xi c_{a-1})(\omega c_{b+1}-\xi c_{b})-(\omega c_{a+1}-\xi c_{a})(\omega c_{b}-\xi c_{b-1})=-c_{b-a+1}(\omega^2 - r\omega\xi+\xi^2),$$ and the desired statement follows from induction on $b-a$.

Suppose that $\omega^2 - r\omega\xi+\xi^2 >0$. Since $c_i=-c_{2-i}$ for $i\in\mathbb{Z}$, we assume $(\omega c_{a}-\xi c_{a-1})<(\omega c_{b}-\xi c_{b-1})$ without loss of generality. Then there exists an integer $e<a$ such that $\omega c_{e+1}-\xi c_{e}>0$ and $\omega c_{e}-\xi c_{e-1}\leq 0$. Then again using $c_i=-c_{2-i}$ and Lemma 3.1, we get 
$$
\omega c_{b}-\xi c_{b-1}=(\omega c_{e+1}-\xi c_{e})c_{b-e+1} + |\omega c_{e}-\xi c_{e-1}|c_{b-e}, 
$$which is clearly greater than $c_{b-a-1}$.
\end{proof}

\noindent\textbf{Case 2-2: }  
Suppose that 
$$ \frac{A_{n_1-1;1}}{A_{n_1-2;1}}\omega_1 -  \frac{(\xi_1 r_1 - \omega_1)\widehat{\xi_1}(\omega_1A_{n_1-1;1}-\xi_1A_{n_1-2;1})A_{n_1-1;1}}{A_{n_1-2;1}^2}< \xi_1. $$
Then $$\xi_1\frac{A_{n_1-2;1}}{A_{n_1-1;1}} > \omega_1 -\frac{(\xi_1 r_1 - \omega_1)\widehat{\xi_1}(\omega_1A_{n_1-1;1}-\xi_1A_{n_1-2;1})}{A_{n_1-2;1}}$$
$$\Longleftrightarrow\ \ \
\frac{(\xi_1 r_1 - \omega_1)\widehat{\xi_1}(\omega_1A_{n_1-1;1}-\xi_1A_{n_1-2;1})}{A_{n_1-2;1}} + (\xi_1 r_1 - \omega_1) > \xi_1r_1-\xi_1\frac{A_{n_1-2;1}}{A_{n_1-1;1}}$$
$$\overset{\text{Lemma }3.10(a)}\Longleftrightarrow\ \ \
\frac{(\xi_1 r_1 - \omega_1)\widehat{\xi_1}(\omega_1A_{n_1-1;1}-\xi_1A_{n_1-2;1})}{A_{n_1-2;1}} + (\xi_1 r_1 - \omega_1) > \xi_1\frac{A_{n_1;1}}{A_{n_1-1;1}}$$
\medskip$$\Longleftrightarrow\ \ \
(\xi_1 r_1 - \omega_1)A_{n_1-1;1} >\frac{ \xi_1 A_{n_1;1}}{\frac{\widehat{\xi_1}(\omega_1A_{n_1-1;1}-\xi_1A_{n_1-2;1})}{A_{n_1-2;1}} +1  }$$
\medskip$$\Longleftrightarrow\ \ \
\frac{(\xi_1 r_1 - \omega_1)A_{n_1-1;1}}{A_{n_1-2;1}} >\frac{ \xi_1 A_{n_1;1}}{\widehat{\xi_1}(\omega_1A_{n_1-1;1}-\xi_1A_{n_1-2;1} ) +A_{n_1-2;1}  }$$
\medskip$$\overset{\text{Lemma}~\ref{10242013lem1}}\Longrightarrow\ \ \
\frac{(\xi_1 r_1 - \omega_1)A_{n_1-1;1}}{A_{n_1-2;1}} >\frac{ \xi_1 c_{w}^{[\xi_1]} A_{n_1;1}}{c_{w+1}^{[\xi_1]}(\omega_1A_{n_1-1;1}-\xi_1A_{n_1-2;1} ) +c_{w}^{[\xi_1]}A_{n_1-2;1}  }$$
\medskip$$\overset{\text{by recursive definition of }c_{w}^{[\xi_1]}}\Longleftrightarrow\ \ \
\frac{(\xi_1 r_1 - \omega_1)A_{n_1-1;1}}{A_{n_1-2;1}} >\frac{ \xi_1 c_{w}^{[\xi_1]} A_{n_1;1}}{c_{w+1}^{[\xi_1]}\omega_1A_{n_1-1;1}-c_{w+2}^{[\xi_1]}A_{n_1-2;1}  }$$
\medskip$$\overset{\text{Lemma }\ref{negone}}\Longleftrightarrow\ \ \
\frac{(\xi_1 r_1 - \omega_1)A_{n_1-1;1}}{A_{n_1-2;1}} > \frac{\xi_1 c_{w}^{[\xi_1]} A_{n_1;1} }{c_{w+2}^{[\xi_1]} A_{n_1;1}- (c_{w+2}^{[\xi_1]} r_1 - c_{w+1}^{[\xi_1]} \omega_1)A_{n_1-1;1} },$$
which is the inequality (\ref{eq page 33}) for $j=1$.
\begin{lemma}\label{10242013lem1}We have
$$\frac{ \xi_1 c_{w+1}^{[\xi_1]} A_{n_1;1}}{c_{w+2}^{[\xi_1]}(\omega_1A_{n_1-1;1}-\xi_1A_{n_1-2;1} ) +c_{w+1}^{[\xi_1]}A_{n_1-2;1}  } >\frac{ \xi_1 c_{w}^{[\xi_1]} A_{n_1;1}}{c_{w+1}^{[\xi_1]}(\omega_1A_{n_1-1;1}-\xi_1A_{n_1-2;1} ) +c_{w}^{[\xi_1]}A_{n_1-2;1}  }.$$
In particular,
$$ \lim_{w\rightarrow\infty}\frac{ \xi_1 c_{w}^{[\xi_1]} A_{n_1;1}}{c_{w+1}^{[\xi_1]}(\omega_1A_{n_1-1;1}-\xi_1A_{n_1-2;1} ) +c_{w}^{[\xi_1]}A_{n_1-2;1}  }>\frac{ \xi_1 c_{w}^{[\xi_1]} A_{n_1;1}}{c_{w+1}^{[\xi_1]}(\omega_1A_{n_1-1;1}-\xi_1A_{n_1-2;1} ) +c_{w}^{[\xi_1]}A_{n_1-2;1}  }.$$
\end{lemma}
\begin{proof}
The statement is equivalent to
$$\frac{ \xi_1 c_{w+1}^{[\xi_1]} A_{n_1;1}}{c_{w+2}^{[\xi_1]}\omega_1A_{n_1-1;1}-c_{w+3}^{[\xi_1]}A_{n_1-2;1}  } >\frac{ \xi_1 c_{w}^{[\xi_1]} A_{n_1;1}}{c_{w+1}^{[\xi_1]}\omega_1A_{n_1-1;1}-c_{w+2}^{[\xi_1]}A_{n_1-2;1}  }$$
\medskip$$\Longleftrightarrow\ \ \
c_{w+1}^{[\xi_1]}c_{w+1}^{[\xi_1]}\omega_1A_{n_1-1;1}-c_{w+1}^{[\xi_1]}c_{w+2}^{[\xi_1]}A_{n_1-2;1} > c_{w}^{[\xi_1]}c_{w+2}^{[\xi_1]}\omega_1A_{n_1-1;1}-c_{w}^{[\xi_1]}c_{w+3}^{[\xi_1]}A_{n_1-2;1}$$
\medskip
$$\overset{\text{Lemma }3.1}\Longleftrightarrow\ \ \
\omega_1A_{n_1-1;1}-\xi_1A_{n_1-2;1} > 0.$$
\end{proof}

This completes the proof of Proposition~\ref{prop2.14} modulo the following Lemmas which are proved in \cite{LS3}. 

\begin{lemma}\cite[Lemma 4.8]{LS3}\label{0423lem1}
Let $a,b,c$ be any nonnegative integers. Then there are nonnegative integers $d_0,...,d_c$ such that
$$
{aX+b\choose c} = \sum_{i=0}^c d_i{X\choose i}
$$for all nonnegative integers $X$.
\end{lemma}

\begin{lemma}\label{0423lem2}\cite[Lemma 4.9]{LS3}
Let $a,b$ be any nonnegative integers. Then there are nonnegative integers $e_0,...,e_{a+b}$ such that
$$
{X\choose a} {X\choose b} = \sum_{i=0}^{a+b} e_i{X\choose i}
$$for all nonnegative  integers $X$.
\end{lemma}


\begin{lemma}\label{0302lem1}\cite[Lemma 4.5]{LS3}
 $$\left\lfloor (A_{n_1-1;1}-\varsigma) \frac{A_{n_1-1;1}}{A_{n_1;1}}\right\rfloor -s_{n_1-2;1}\ge 0.$$
\end{lemma}

\begin{lemma}\label{0303lem944}\cite[Lemma 4.6]{LS3}
Let $I$ be a subset of the integers and let $q:I\to \mathbb{R}$ be a function.  Suppose that a polynomial of $x$ of the form $$\sum_{m\in I} q(m)x^{b-m}$$is divisible by $(1+x)^g$ and its quotient has nonnegative coefficients. Let
$$
p(m)=\sum_{i=0}^h d_i {b-m\choose i}
$$ be a polynomial of $m$ with $d_i\geq 0$.
Then  $$\sum_{m\in I} p(m)q(m)x^{b-m}$$is divisible by $(1+x)^{g-h}$ and its quotient has nonnegative coefficients. 
\end{lemma}
\subsection{Proof of Proposition \ref{cor2.17}}\label{sect 5.3}

\begin{proof}
By Theorem~\ref{thm01312012} and replacing $x_1$ by $x_{d,t_{w_Z}}^{-1}$, we have 
$$
\sum_{\begin{array}{c}\scriptstyle \tau_{0;1},\tau_{1;1},\cdots,\tau_{n_1-2;1}\\ \scriptstyle s_{n_1-1;1}=A_{n_1-1;1}-\varsigma \end{array}}   \left( \prod_{w=0}^{n_1-2}\gchoose{A_{w+1;1} - r_1s_{w;1}}{\tau_{w;1}}   \right){x_{d;t_{\omega_Z}}}^ {A_{n_1-1;1}-r_1s_{n_1-2;1}}
$$
is   divisible by $(1+{x_{d;t_{\omega_Z}}}^{r_1})^{r_1(A_{n_1-1;1}-\varsigma) -A_{n_1;1}}$  in $\mathbb{Z}[{x_{d;t_{\omega_Z}}}^{\pm 1}]$, and the resulting quotient has nonnegative coefficients.
Multiplying the sum with ${x_{d;t_{\omega_Z}}}^{r_1  \left\lfloor (A_{n_1-1;1}-\varsigma) \frac{A_{n_1-1;1}}{A_{n_1;1}}\right\rfloor -A_{n_1-1;1}}$ shows that
\begin{equation}\label{double star1}
\sum_{\begin{array}{c} \scriptstyle \tau_{0;1},\tau_{1;1},\cdots,\tau_{n_1-2;1}\\\scriptstyle  s_{n_1-1;1}=A_{n_1-1;1}-\varsigma\end{array}}   \left( \prod_{w=0}^{n_1-2}\gchoose{A_{w+1;1} - r_1s_{w;1}}{\tau_{w;1}}   \right)({x_{d;t_{\omega_Z}}}^{r_1})^{ \left\lfloor (A_{n_1-1;1}-\varsigma) \frac{A_{n_1-1;1}}{A_{n_1;1}}\right\rfloor -s_{n_1-2;1}}
\end{equation}
is also divisible by $(1+{x_{d;t_{\omega_Z}}}^{r_1})^{r_1(A_{n_1-1;1}-\varsigma) -A_{n_1;1}}$, and the resulting quotient has nonnegative coefficients.
Moreover,  Lemma \ref{0302lem1} above implies that the exponents in the expression (\ref{double star1}) are non-negative,
and, since the divisor has constant term 1, this shows that the quotient is a polynomial.

Note that the statement about the divisibility of (\ref{double star1}) also holds when we replace $(x_{d;t_{\omega_Z}}^r)$ with any other expression $X$. 
We can write the second sum of (\ref{eq2.14}) as follows:
\[ \sum_{\theta>0}x_{e;t'}^{-\theta}\sum_{\varsigma\ge 0} \sum_ \mathfrak{r} \lambda_\mathfrak{r} \mathfrak{r} \sum_{b,m}q(m)p(m)X^{b-m},\]
where
\[
\begin{array}
 {rcl}
q(m) &=& \displaystyle\prod_{w=0}^{n_{\nuisone}-2}\gchoose{A_{w+1;\nuisone} - r_\nuisone s_{w;\nuisone}}{\tau_{w;\nuisone}} \\
& \\
p(m)&=&\displaystyle  \sum_{j=0}^{\sum_{w=1}^{n_2-3} \tau_{w,2} } d_j {{ \left\lfloor (A_{n_\nuisone-1;\nuisone}-\varsigma) \frac{A_{n_\nuisone-1;\nuisone}}{A_{n_\nuisone;\nuisone}}\right\rfloor -s_{n_\nuisone-2;\nuisone}}
 \choose j}  \\
\\
b &=& \left\lfloor (A_{n_\nuisone-1;\nuisone}-\varsigma) \frac{A_{n_\nuisone-1;\nuisone}}{A_{n_\nuisone;\nuisone}}\right\rfloor \\
& \\
m&=& s_{n_\nuisone-2;\nuisone} \\
\\
X&=&\left(\frac{\prod_i x_{i;t'}^{ [b_{i,e}^{t'}]_+ }  }{\prod_i x_{i;t'}^{ [-b_{i,e}^{t'}]_+ } }\right)^{\text{sgn}(2b_{d,e}^{t'}+1)}.
\end{array}\]

Moreover, we can replace the upper bound $\sum_{w=1}^{n_2-3} \tau_{w,2}$ of the sum in $p(m)$ by the larger integer $A_{n_1-2;1}-r_1\varsigma-\theta$ and setting $d_j=0$, whenever   $j>\sum_{w=1}^{n_2-3} \tau_{w,2}$. 
The fact that
 $\sum_{w=1}^{n_2-3} \tau_{w,2} \le A_{n_1-2;1}-r_1\varsigma-\theta$
  follows from equation (\ref{theta}).
 
Using Lemma~\ref{0303lem944}  
with
$g=r_\nuisone(A_{n_\nuisone-1;\nuisone}-\varsigma) -A_{n_\nuisone;\nuisone}$ and 
$h=A_{n_\nuisone-2;\nuisone} - r_\nuisone \varsigma -\theta$, we get that the second sum in the expression  in (\ref{eq2.14}) 
is divisible by 
\begin{equation}\label{divisor1}
\aligned
&\left(1+\left(\frac{\prod_i x_{i;t'}^{ [b_{i,e}^{t'}]_+ }  }{\prod_i x_{i;t'}^{ [-b_{i,e}^{t'}]_+ } }\right)^{\text{sgn}(2b_{d,e}^{t'}+1)}\right)^{r_\nuisone(A_{n_\nuisone-1;\nuisone}-\varsigma) -A_{n_\nuisone;\nuisone}-(A_{n_\nuisone-2;\nuisone} - r_\nuisone \varsigma -\theta)}\\
&\stackrel{\small \textup Lemma~ \ref{negone}}{=}\left(1+\left(\frac{\prod_i x_{i;t'}^{ [b_{i,e}^{t'}]_+ }  }{\prod_i x_{i;t'}^{ [-b_{i,e}^{t'}]_+ } }\right)^{\text{sgn}(2b_{d,e}^{t'}+1)}\right)^\theta,\endaligned\end{equation} and the resulting quotient has nonnegative coefficients.
Finally, dividing (\ref{divisor1}) by $x_{e;t'}^\theta$ and using the fact that  
\[ \widetilde{\widetilde{x_{e;t}}} =\left.\left(\prod_i x_{i;t'}^{ [b_{i,e}^{t'}]_+ } +\prod_i x_{i;t'}^{ [-b_{i,e}^{t'}]_+ } \right)\right/x_{e;t'}\]
we see that the second sum in \eqref{eq2.14} is divisible by 
 $\widetilde{\widetilde{x_{e;t}}}^\theta$.
This completes the proof of Proposition~\ref{cor2.17}(1). Part (2) can be proved in a similar way using Proposition~\ref{prop2.16}.\end{proof}

\section{Application to quiver Grassmannians}\label{sect Grass}

Let $(Q,S) $ be a quiver with potential and
 let $M$ be an indecomposable representation of  $(Q,S)$ which is obtained by a mutation sequence starting from a negative simple representation,  see \cite[Section 5]{DWZ2}. Let $\textup{Gr}_{\mathbf{e}} M$ denote the Grassmannian of subrepresentations of $M$ of dimension vector $\mathbf{e}$.
 \begin{theorem} 
  The Euler-Poincar\'e characteristic of $\textup{Gr}_{\mathbf{e}} M$ is non-negative.
\end{theorem}

\begin{proof}  It has been shown in 
 \cite{DWZ2}  that $M$ corresponds to a cluster variable whose $F$-polynomial  is equal to a sum of monomials whose  coefficients are given by the  Euler-Poincar\'e characteristic of $\textup{Gr}_{\mathbf{e}} (M)$.
 Theorem~\ref{11192011thm} implies that these coefficients are non-negative.
\end{proof}

\end{document}